\documentclass[11pt]{amsart}
\usepackage[english]{babel}
\usepackage{amsthm,amsmath,graphicx}
\usepackage{amssymb,bbm,esint}
\numberwithin{equation}{section}
\usepackage{appendix}
\usepackage{geometry}
\newcommand{\hn}{\mathbb{H}^N}
\newcommand{\dist}[2][\xi,\eta]{|{#2}^{-1}\circ{#1}|_{\hn}}
\newcommand{\gagfrac}[3][u,x,y]{\frac{{#1}({#2},t)-{#1}({#3},t)}{\dist[{#2}]{#3}^{Q+2s}}}

\usepackage[dvipsnames]{xcolor}
\usepackage[colorlinks=true,allcolors=ForestGreen]{hyperref}
\usepackage{amsthm}
\DeclareMathOperator*{\esssup}{ess\,sup}
\DeclareMathOperator*{\essinf}{ess\,inf}

\usepackage[capitalize,nameinlink]{cleveref}
\crefname{section}{Section}{Sections}
\crefname{subsection}{Subsection}{Subsections}
\crefname{subsection}{subsection}{subsections}
\crefname{condition}{Condition}{Conditions}
\crefname{hypothesis}{Hypothesis}{Conditions}
\crefname{assumption}{Assumption}{Assumptions}
\crefname{lemma}{Lemma}{Lemmas}
\crefname{prop}{Proposition}{Proposition}
\crefname{claim}{Claim}{Claims}
\crefname{observation}{Observation}{Observations}
\crefname{example}{Example}{Examples}

\crefformat{equation}{\textup{#2(#1)#3}}
\crefrangeformat{equation}{\textup{#3(#1)#4--#5(#2)#6}}
\crefmultiformat{equation}{\textup{#2(#1)#3}}{ and \textup{#2(#1)#3}}
{, \textup{#2(#1)#3}}{, and \textup{#2(#1)#3}}
\crefrangemultiformat{equation}{\textup{#3(#1)#4--#5(#2)#6}}%
{ and \textup{#3(#1)#4--#5(#2)#6}}{, \textup{#3(#1)#4--#5(#2)#6}}%
{, and \textup{#3(#1)#4--#5(#2)#6}}

\Crefformat{equation}{#2Equation~\textup{(#1)}#3}
\Crefrangeformat{equation}{Equations~\textup{#3(#1)#4--#5(#2)#6}}
\Crefmultiformat{equation}{Equations~\textup{#2(#1)#3}}{ and \textup{#2(#1)#3}}
{, \textup{#2(#1)#3}}{, and \textup{#2(#1)#3}}
\Crefrangemultiformat{equation}{Equations~\textup{#3(#1)#4--#5(#2)#6}}%
{ and \textup{#3(#1)#4--#5(#2)#6}}{, \textup{#3(#1)#4--#5(#2)#6}}%
{, and \textup{#3(#1)#4--#5(#2)#6}}

\crefdefaultlabelformat{#2\textup{#1}#3}
\newtheorem{theorem}{Theorem}[section]
\newtheorem{lemma}[theorem]{Lemma}

\newtheorem{proposition}[theorem]{Proposition}

\newtheorem{remark}[theorem]{Remark}        % Use {\rm ...}

\numberwithin{equation}{section}

\def\Yint#1{\mathchoice
	{\YYint\displaystyle\textstyle{#1}}%
	{\YYint\textstyle\scriptstyle{#1}}%
	{\YYint\scriptstyle\scriptscriptstyle{#1}}%
	{\YYint\scriptscriptstyle\scriptscriptstyle{#1}}%
	\!\iint}
\def\YYint#1#2#3{{\setbox0=\hbox{$#1{#2#3}{\iint}$}
		\vcenter{\hbox{$#2#3$}}\kern-.51\wd0}}
\def\longdash{{-}\mkern-3.5mu{-}} 
\def\fiint{\Yint\longdash}

\title[Harnack inequality for Mixed Local-Nonlocal Parabolic equations in Heisenberg Group]{{Mixed Local-Nonlocal Parabolic Equations in Heisenberg Group : Harnack Inequality with an Optimal Tail}}
\begin{document}
\author[ Kar ]{Debraj Kar}
\address{Department of Mathematics,  University of Kalyani }
\email{ kardebrajofficial@gmail.com, debrajmath22@klyuniv.ac.in}
\begin{abstract}
    We prove Harnack inequality with an optimal tail for linear mixed local-nonlocal parabolic equation on the Heisenberg group. For this purpose, we make use of the expansion of positivity method and a clustering lemma, thereby bypassing Moser iteration and logarithmic estimates. \\[2mm]
\end{abstract}
\maketitle

\setcounter{secnumdepth}{1}
\setcounter{tocdepth}{1}
\tableofcontents
\section{Introduction}
In this paper, we consider the regularity theory, specifically the Harnack estimate, for the following mixed local-nonlocal parabolic equation on the Heisenberg group
\begin{align}\label{eq1.1}
    \partial_tu-\Delta_{\hn}u+\mbox{ p.v. } \int_{\hn}\gagfrac[u]{\xi}{\eta}d\eta dt=0,
\end{align}
on $\Omega_T:=\Omega\times(0,T]$ where $\Omega$ is a bounded open subset in the Heisenberg Group $\mathbb{H}^N$, $T>0$ and
p.v. denotes the Cauchy principal value. {%\textcolor{red}{see \cite{Nak23_3,CN25_2,Lia24_4,MPPP23_1,SZ23_5}. }}
\subsection{History}
\subsubsection{Context} The parabolic mixed local nonlocal equation arises in the stochastic process that involves superposition of classical random walk and a Le\'vy flight. This problem has gained attention due to its extensive applications in phenomena, such as, in plasma physics \cite{BdCN13_29} and in biology \cite{DLV22_30}.

Harnack inequality is a cornerstone in the field of regularity theory. In 1964, Moser\cite{Mos64_32} proved the Harnack inequality for the linear parabolic equation and H\"older continuity, as its consequence. Later, this result was extended to the quasilinear counterpart by Trudinger \cite{Tru68_33}. For the $p$-Laplace parabolic problem, DiBenedetto \cite{DB88_34} proved intrinsic Harnack inequality for non-negative solutions. Still later, DiBenedetto, Gianazza, and Vespri \cite{DBGV08_36} gave a new approach using the {\em expansion of positivity} and an exponential change of variable to establish Harnack Inequality. In the setting of measure metric spaces, Kinnunen and co-authors \cite{KMMJP21_24} have shown Harnack inequality for Parabolic De Giorgi-class of functions.

Next, we provide a brief overview of the nonlocal case. For the elliptic nonlocal p-Laplace type equations, one can see the work of Di Castro, Kuusi, and Palatucci \cite{DCKP14_35}. In parabolic contexts, we start with the paper of Kassmann and Schwab \cite{KS13_39} where authors prove the Harnack result for $p=2$. This is followed by a series of papers such as those of Felsinger, Kassmann \cite{FK13_42}, Prasad \cite{Pra23_43}, Nakamura \cite{Nak23_3}, Ciani and Nakamura \cite{CN25_2}, Adimuthi, Prasad, and Tewary \cite{APT25_38}, Kar\cite{Kar26_9} where various regularity estimates for nonlocal parabolic equations are established. The regularity theory for the nonlocal heat equation culminates in the paper of Kassmann and Weidner \cite{KW24_40} where Harnack inequality with an optimal tail is proved.

Due to their interesting structure and applications, mixed local and non-local type of problems have attained importance and purely analytic methods have been developed as opposed to their probabilistic origin. In their paper \cite{GK22_44}, Garain and Kinnunen studied the Harnack inequality using {\em tail} estimate and expansion of positivity. Foondum \cite{Foo09_45} proves the Harnack inequality using probabilistic method for the mixed local nonlocal elliptic equations. Moreover, for boundary Harnack result, see \cite{CKSV12_46}. Also, of note is the paper of Nakamura \cite{Nak23_48} fox mixed local-nonlocal doubly nonlinear equations. Also one can look into papers of Shang and Zhang \cite{SZ23_5} for the Harnack result when $p>2$ and Bhowmick et al \cite{BGKL26_52} for the Harnack estimate of nonlinear mixed fractional superposition operator with nonlocal superposition tail.

\subsubsection{Novelty} To the best of our knowledge, this is the first ever result where the Harnack inequality for the mixed local nonlocal parabolic equation on the Heisenberg Group  has been studied taking into account an optimal tail. This type of Harnack result with optimal tail condition is inspired from the result of Kassmann and Weidner \cite{KW24_40} where the nonlocal heat equation  was considered in a Euclidean setup. Unlike their paper, we rely on the De Giorgi method and expansion of positivity to deduce our result. Also, as compared to the paper of Ciani and Nakamura \cite{CN25_2}, where authors prove Harnack type result for nonlocal parabolic De Giorgi class of functions under $L^{p-1+\varepsilon}$ -  tail condition, where $\varepsilon>0$ and $p>1$, our result is optimal as far as tail is concerned. Additionally, Nakamura \cite{Nak23_3} proves their results with additional restriction over $s$ whereas our result is free from such restriction. It becomes possible due to the structure of the equation (\ref{eq1.1}). Moreover, one can find a clustering lemma in the Heisenberg group in this paper (\cref{L6.1}).
\subsection{Notations and Definition}
\subsubsection{Notations} We are going to define some notations which we will use in sequel. 
\begin{itemize}
    \item Take $(x,t)$ to be a point in $\hn\times (0,T]$.
    \item $B_r(\xi_0):=\{\eta\in\hn:\dist[\xi_0]{\eta}<r\}$ represents a ball with center at $\xi_0$ and radius $r$. For simplicity, sometimes we may use $B_r$ instead of $B_r(\xi_0)$. We use $B$ to denote the generic ball. 
    \item Let $z_0:=(\xi_0,t_0)\in \hn\times (0,T]$. Let us define {\em forward} and {\em backward} parabolic cylinder as follows
    \begin{align*}
        \begin{cases}
            & z_0+\mathcal{Q}^\oplus_r(\theta):=B_r(\xi_0)\times (t_0,t_0+\theta r^{2}],\\
            & z_0+\mathcal{Q}^\ominus_r(\theta):=B_r(\xi_0)\times (t_0-\theta r^{2},t_0].
        \end{cases}
    \end{align*}
    Moreover, for simplicity, we denote $z_0+\mathcal{Q}^\oplus_r:=z_0+\mathcal{Q}^\oplus_r(1)$ and $z_0+\mathcal{Q}^\ominus_r:=z_0+\mathcal{Q}^\ominus_r(1)$.
    \item We shall frequently use $f_t$, $\partial_tf$, $\frac{\partial f}{\partial  t}$ to indicate the time derivative of $f$.
    \item For any set $A$, $A^c$ denotes the {\em complement} of the set $A$. 
    \item Throughout this paper, we use \texttt{data} to represent the constant terms $s,Q$.
    \item We denote 
    \begin{align*}
        a_+:=\max\{a,0\},\,a_-:=-\min\{a,0\}\,\,\mbox{ for }a\in\mathbb{R}.
    \end{align*}
\end{itemize}
\subsubsection{Definitions}
Here we want to provide some definitions. Let $\Omega$ be an open subset of $\hn$ and $I$ be a subset of $\mathbb{R}$. 
\begin{itemize}
    \item The {\em Sobolev space} on the Heisenberg group, denoted by $W^{1,p}(\Omega)$, is defined as follows
    \begin{align*}
        W^{1,p}(\Omega):=\Bigg\{v\in L^p(\Omega): \int_{\Omega}|\nabla_{\hn}u|^p\,dy<\infty\Bigg\}.
    \end{align*}
    \item The {\em fractional Sobolev space} on the Heisenberg group, denoted by $W^{s,p}(\Omega)$, is defined by 
    \begin{align*}
        W^{s,p}(\Omega):=\Bigg\{v\in L^p(\Omega): \iint_{\Omega\times\Omega}\frac{|v(x)-v(y)|^p}{\dist[x]{y}^{Q+sp}}\,dx\,dy<\infty\Bigg\}.
    \end{align*}
    We denote $W^{s,p}_0(\Omega)$ as closure of $C_0^\infty(\Omega)$ in $W^{s,p}(\Omega)$.
    \item We introduce the parabolic solution space $$V^{p,q}(B(\xi_0,R)\times I):=L^p(I;W^{1,p}(B(\xi_0,R))\cap L^\infty(I;L^q(B(\xi_0,R)),$$ for $q>1$, $R>1$, $\xi_0\in\hn$ and we define 
	\begin{align*}
		\|f\|_{V^{p,q}(B_R(\xi_0)\times I)}:=\left[\int_I[f(\cdot,t)]^p_{W^{1,p}(B_R(\xi_0))}\,dt+\sup_{t\in I}\|f(\cdot,t)\|_{L^q(B_R(\xi_0))}^q\right]^{\frac1p} .
	\end{align*}
    \item The {\em parabolic tail space}, denoted by $L^m(I;L^m_\gamma(\Omega))$, on the Heisenberg group is defined by 
    $$L^m(I;L^m_\gamma(\Omega)):=\Bigg\{v\in L^m(I;L^m_\textup{loc}(\Omega)):\iint_{I\times\hn}\frac{|v(x,t)|^m}{1+|x|_{\hn}^{Q+\gamma}}\,dx\,dt<\infty\Bigg\},$$ for some $m,\gamma>0$.
    \item {\em Nonlocal parabolic tail : } For $z_0=(x_0,t_0)\in \hn\times I$ and $r,\tau>0$ with $(t_0-\tau,t_0)\subset I$, the nonlocal parabolic tail is defined by 
    $$\textup{ Tail }(v;z_0,r,\tau):=\Bigg(\fint_{t_0-\tau}^{t_0}r^2\int_{\hn\setminus B_r(x_0)}\frac{|v(y,t)|}{\dist[x_0]{y}^{Q+2s}}\,dy\,dt\Bigg).$$
    Note that if $v\in L^1(I;L^1_{2s}(\hn))$ then it holds that $\textup{Tail }(v;z_0,r,\tau)<\infty$. Throughout the paper, we mainly use the following time dependent tail on a parabolic cylinder as follows,$$\textup{Tail}(v(t);r,x_0):=r^2\int_{\hn\setminus B_r(x_0)}\frac{|v(y,t)|}{\dist[x_0]{y}^{Q+2s}}\,dy\,dt.$$
    \item {\em Weak Solution:} A function $u\in \hn\times (0,T]\xrightarrow{} \mathbb{R}$ satisfying
     \begin{align*}
         u\in L^1_{\textup{loc}}(0,T;L^1_{2s}(\hn))\cap L^2_{\textup{loc}}(0,T;W^{1,2}_{\textup{loc}}(\hn))\cap C_{\textup{loc}}([0,T];L^2_{\textup{loc}}(\hn))
     \end{align*}
     is called the weak sub(super) solution of (\ref{eq1.1}) if for every compact set $\mathcal{K}\subset \Omega$, sub interval $[T_1,T_2]\subset (0,T]$ and for every test function $$ \varphi \in W^{1,2}_{\textup{loc}}(0,T;L^2(\mathcal{K}))\cap L^2_{\textup{loc}}(0,T;W^{1,2}_0(\mathcal{K})),$$ following holds
     \begin{align}\tag{WeakForm}\label{wf}
         &\int_\mathcal{K}u\,\varphi\,d\eta\Biggr\vert_{T_1}^{T_2}-\int_{T_1}^{T_2}\int_\mathcal{K} u\partial_t\varphi\,d\eta\,dt+\int_{T_1}^{T_2}\int_\mathcal{K}\nabla_{\hn}u\nabla_{\hn}\varphi\,d\xi\,dt\nonumber\\
			&\quad+\int_{T_1}^{T_2}\iint_{\hn\times\hn}(u(\xi,t)-u(\eta,t))(\varphi(\xi,t)-\varphi(\eta,t))\, K(\xi,\eta,t)d\xi\,d\eta\,dt\leq (\geq)\,0. \nonumber
     \end{align}
\end{itemize}
\subsection{Main Result}
Our first result is local boundedness result proved via De Giorgi method. This is the refined version of its counterpart in \cite{KT25_6}.
\begin{theorem}\label{T1.1}(Local Boundedness Result)
    Suppose $z_0=(\xi_0,t_0)\in\Omega_T$ and $R>0$ . Let $u\in L^1_{\textup{loc}}(0,T;L^1_{2s}(\hn))\cap L^2_{\textup{loc}}(0,T;W^{1,2}_{\textup{loc}}(\hn))\cap C_{\textup{loc}}([0,T];L^2_{\textup{loc}}(\hn))$ be a weak subsolution of the problem (\ref{eq1.1}). Then there exist some constant $C=C(\textup{\texttt{data}}),\, C_1=C_1(Q)$ and $\mathfrak{C}=\mathfrak{C}(\textup{\texttt{data}},\sigma)$ for some $\sigma\in (0,1)$ such that 
    \begin{align*}
         \sup_{z_0+\mathcal{Q}_{\sigma R}^\ominus}u(\xi,t)&\leq \frac{C}{(1-\sigma)^{C_1}}\left[\left(\fiint_{z_0+\mathcal{Q}_{R}^\ominus}u_+^2\,d\xi  dt\right)^\frac{1}{2}+\left(\fint_{t_0-R^2}^{t_0}\left(\fint_{B_R(\xi_0)}u_+(\eta,t)\,d\eta\right)^2\,dt\right)^\frac{1}{2}+1\right]\\
    &\quad\quad\quad\quad +\mathfrak{C}R^{-2}\int_{t_0-R^2}^{t_0}\textup{Tail}(u_+(t);R,\xi_0)dt.
    \end{align*}
\end{theorem}
Next, we draw attention to various forms of Harnack inequalities proved in this paper. Starting with Weak Harnack Inequality I, we prove this result with $L^1$ {\em tail} and time gap of order $\rho^2$, for some $\rho>0$:
\begin{theorem}\label{T1.2}(Weak Harnack Inequality I)
    Let $(\xi_0,t_0)\in\Omega_T$, $0<T_1<T_2$ and $R>0$. Let $u$ be local weak supersolution of the problem (\ref{eq1.1}) such that $u\geq 0$ in $\mathcal{Q}:=B_R(\xi_0)\times (T_1,T_2]\Subset\Omega_T$. Then there exist $\gamma>1$ and $q,\delta_0\in (0,1)$, all depending on \textup{\texttt{data}}, such that
    \begin{align*}
        \left(\fint_{B_\rho(\xi_0)}u(\xi,t_0)^q\,d\xi\right)^\frac{1}{q}\leq \gamma\left[\inf_{B_\rho(\xi_0)\times (t_0+\frac{1}{2}\delta_0\rho^2,t_0+\delta_0\rho^2]}u+R^{-2}\int_{T_1}^{T_2}\textup{Tail}[u_-(t);R,\xi_0]\, dt\right],
    \end{align*}
    for some $0<\rho<\frac{R}{2}$. 
\end{theorem}
The following result improves the previous result in Theorem \ref{T1.2}. Theorem \ref{T1.2} can be deduced from Theorem \ref{T1.3}: 
\begin{theorem}\label{T1.3}(Weak Harnack Inequality II)
    Let $(\xi_0,t_0)\in\Omega_T$, $0<T_1<T_2$ and $R>0$. Let $u$ be local weak solution of the problem (\ref{eq1.1}) such that $u\geq 0$ in $\mathcal{Q}$. Suppose $B_{2\rho}(\xi_0)\times (-(2\rho)^2,2(4\rho)^2]\subseteq \mathcal{Q}\Subset \Omega_T$, for some $\rho>0$. Then there exists $\varkappa\in (0,1)$, depending on \textup{\texttt{data}}, such that 
    \begin{align*}
      \inf_{B_\rho(\xi_0)}u(t) +R^{-2}\int_{T_1}^{T_2}\textup{Tail}[u_-(t);R,0]\, dt&\geq \varkappa\Bigg[\left(\fiint_{\mathcal{Q}_{2\rho}^\ominus}u_+^2\,d\xi  dt\right)^\frac{1}{2}+\left(\fint_{-{2\rho}^2}^{0}\left(\fint_{B_{2\rho}}u_+(\xi,t)\,d\xi\right)^2\,dt\right)^\frac{1}{2}\\&\quad\quad+\fint_{-{2\rho}^2}^0\textup{Tail}[u_+(t);2\rho,0]\, dt+1\Bigg],
    \end{align*}
    for $t\in (\frac{5}{4}\rho^2,(4\rho)^2]$.
\end{theorem}
Take $\sigma=\frac{1}{2}$, merging two results in Theorem \ref{T1.1} and Theorem \ref{T1.3} with suitable choice of constants, we have the following
\begin{theorem}\label{T1.4}(Harnack Inequality)
    Let $(\xi_0,t_0)\in\Omega_T$, $0<T_1<T_2$ and $R>0$. Let $u$ be local weak solution of the problem (\ref{eq1.1}) such that $u\geq 0$ in $\mathcal{Q}$. Then there exists a constant $\gamma=\gamma(\textup{\texttt{data}})>0$ such that
    \begin{align*}
        \sup_{B_\rho(\xi_0)\times (t_0-\rho^2,t_0]}u\leq \gamma \left[\inf_{B_\rho(\xi_0)\times (t_0+\frac{5}{4}\rho^2,t_0+(4\rho)^2]}u+R^{-2}\int_{T_1}^{T_2}\textup{Tail}[u_-(t);R,0]\,dt\right],
    \end{align*}
    provided $B_{2\rho}(\xi_o)\times (t_0-(2\rho)^2,t_0+2(4\rho)^2]\subseteq \mathcal{Q}\Subset\Omega_T$.
\end{theorem}
As far as the Harnack estimate is concerned, the basic difference between elliptic and parabolic problem  is that in the elliptic Harnack inequality, the supremum of the positive solution is controlled by the infimum of the same solution over the same ball whereas in the parabolic case, a {\em waiting time} or {\em delay time} or {\em time lag} is needed. In a counterexample, Moser \cite{Mos64_32} also proves the necessity  of this time gap for caloric functions.  There is also a time lag in our result. In contrast, recently Liao and Weidner \cite{LW25_49} proves time insensitive Harnack for  parabolic nonlocal equation where sup and inf are on overlapping cylinders. 
\subsection{Heisenberg Group} The results in this subsection is based on the book \cite{BLU07_25}. The Heisenberg group is a nilpotent Lie Group.  This is the set $\mathbb{R}^{2N+1}$ equipped with the following group operation
 $$\xi\circ \xi'=(x+x',y+y',\Upsilon+\Upsilon'+2\left(\big<x',y\big>-\big<x,y'\big>\right),$$
 where $$\xi=(z,\Upsilon)=(x,y,\Upsilon)=(x_1,x_2,...,x_N,y_1,y_2,...,y_N,\Upsilon)$$
 and $$\xi'=(z',\Upsilon')=(x',y',\Upsilon')=(x'_1,x'_2,...,x'_N,y'_1,y'_2,...,y'_N,\Upsilon').$$
 A one parameter group of automorphism on $\hn$ is defined as 
 $$\Phi_\lambda(x,y,\Upsilon)=(\lambda x,\lambda y ,\lambda^2\Upsilon),$$
 so that it has a {\em homogeneous dimension} of $Q=2N+2$. Also a {\em homogeneous norm} $d_0$ such that $d_0:\hn\rightarrow [0,\infty]$ is a function satisfying
 \begin{itemize}
     \item $d_0(\Phi_\lambda(\xi))=\lambda d_0(\xi)$ for any $\xi\in\hn$ and $\lambda>0$.
     \item $d_0(\xi)=0$ if and only if $\xi=0$. 
 \end{itemize}
 Define standard {\em homogeneous norm} on $\hn$ by
 $$|\xi|_{\hn}=(|z|^4+|\Upsilon|^2)^{1/4}.$$
 It is well known that all homogeneous norm on $\hn$ are equivalent (see \cite[Corollary 5.1.5]{BLU07_25}). Additionally, Heisenberg group with any homogeneous norm $d_0$ satisfies a pseudo-triangle inequality as in the following lemma. 
 \begin{lemma}\cite[Proposition 5.1.7]{BLU07_25}
     Consider Heisenberg group $\hn$ with homogeneous norm $d_0$. Then there exists a constant $c>0$ such that for all $\xi,\eta\in\hn$, following holds
     \begin{itemize}
         \item $d_0(\xi\circ \eta)\geq \frac{1}{c}d_0(\xi)-d_0(\eta^{-1}),$
         \item $d_0(\xi\circ \eta)\leq c(d_0(\xi)+d_0(\eta)),$
         \item $d_0(\xi\circ \eta)\geq \frac{1}{c}d_0(\xi)-cd_0(\eta).$
     \end{itemize}
 \end{lemma}
 \begin{remark}
     In the case of standard homogeneous norm, the constant $c$ above is replaced by $1$. In that case we may assume that $\hn$ is a metric space equipped with the metric $|\cdot^{-1}\circ \cdot|_{\hn}$ (for details see \cite[Example 5.1]{BFS18_28}). Moreover, this metric is symmetric i.e. $|\xi^{-1}\circ \eta|_{\hn}=|\eta^{-1}\circ \xi|_{\hn}$ for any $\xi,\eta\in\hn.$
 \end{remark}
 Also note that the {\em Haar measure} on $\hn$ is the same as the {\em Lebesgue measure} on $\mathbb{R}^{2N+1}$ and satisfies the following the {\em doubling measure} property, that is for any $R>0$ and $x_0\in\hn$, 
 $$|B_{2R}(x_0)|\leq C(Q)|B_R(x_0)|$$
 This properties make $(\hn,|\cdot^{-1}\circ \cdot|_{\hn},\mathfrak{L}^{2N+1})$ into a $Q$-regular measure metric space with a doubling measure (for more, see \cite[Section 11.3]{HaK00_26} and \cite[Proposition 14.2.9]{HKST15_27}).
\subsection{Plan of the Paper}
This paper is divided into several sections. In Section \ref{S2}, we discuss some known and standard results including parabolic embedding results followed by Section \ref{S3}, where we have refined the Caccioppoli inequality for the mixed local-nonlocal case. In Section \ref{S4}, we deduce the local boundedness result. Section \ref{S5} is dedicated to the results like De Giorgi Lemma, Measure Theoretical results, Expansion of Positivity etc. These type of results are useful during the proofs of results in Section \ref{S6}, where we have shown Weak Harnack Inequality I and II results. 
\subsection{Acknowledgment}
The author is financially supported by CSIR under the file number: $09/0106$
$(13571)/2022$-EMR-I. The author is grateful to Vivek Tewary for suggesting this problem. The author also expresses his gratitude to Krea University for providing the necessary funding and hospitality during the author's visit at Krea University, where this work was conceived.
\section{Some Technical Results}\label{S2}
Regularization in time is one of the important idea in the regularity theory of parabolic PDE. This idea is taken from the paper \cite{BLS21_7}. Let us consider nonnegative, smooth even function $\gamma:\mathbb{R}\rightarrow\mathbb{R}$ with compact support in $(-\frac{1}{2},\frac{1}{2})$ and $\int_\mathbb{R}\gamma=1$. Then for $\phi \in L^1((a,b))$ and $0<\varepsilon<\min \{b-t,t-a\}$, set the convolution as follows
\begin{align}\label{eq2.1}
    \phi^\varepsilon(t):=\frac{1}{\varepsilon}\int_{t-\varepsilon/2}^{t+\varepsilon/2}\gamma\Big(\frac{t-\tau}{\varepsilon}\Big)\phi(\tau)d\tau=\frac{1}{\varepsilon}\int_{-\varepsilon/2}^{\varepsilon/2}\gamma\Big(\frac{\alpha}{\varepsilon}\Big)\phi(t-\alpha)d\alpha,\;\;\;\;\;t\in(a,b).
\end{align}
We now state some technical lemmas regarding this regularization in time. 
\begin{lemma}\cite[Lemma 2.2]{KT25_6}\label{L2.1}
   Let $\Omega\subset\hn$ be a open set. Consider $0<\varepsilon<\varepsilon_0<\frac{T_2-T_2}{2}$ and $q\geq 1$. If $f\in C([T_1,T_2];L^q(\Omega))$, then for every $t\in(T_1+\frac{\varepsilon_0}{2},T_2-\frac{\varepsilon_0}{2})$, the sequence $\phi^\epsilon(\cdot,t)$ converges to $\phi(0,t)$ in $L^q(\Omega)$ as $\varepsilon\rightarrow 0$.
\end{lemma}
\begin{lemma}\cite[Lemma 2.3]{KT25_6}\label{L2.2}
    Let $0<\varepsilon<\varepsilon_0<\frac{T_1-T_2}{2}$ and $p,q\in [1,\infty]$. Consider $\phi\in L^q\left(T_1,T_2;L^p(\Omega)\right)$. Then there exists a positive constant $C=C(p,q)$ such that for any $\varepsilon\leq \varepsilon_0$
    \begin{align*}
        \|\phi^\varepsilon\|_{L^q\left(T_1+\frac{\varepsilon_0}{2},T_2-\frac{\varepsilon_0}{2};L^p(\Omega)\right)}\leq C.
    \end{align*}
\end{lemma}
Additionally, we want to mention some useful lemmas which will be used in the sequel. 
\begin{lemma}\cite[Lemma 2.6]{MPPP23_1}\label{L2.3}
    Let $\xi_0\in\hn$ and $|\cdot|_{\hn}$ refers the standard homogeneous norm in $\hn$. Then 
    \begin{align*}
        \int_{B_\rho(\xi_0)}\frac{1}{\dist[\xi_0]{\xi}^{Q+\gamma}}\,d\xi\leq \frac{C(Q,\gamma)}{\rho^\gamma},
    \end{align*}
    for some $\gamma,\rho>0$.
\end{lemma}
Next, we are going to state the lemma concerning the geometric convergence of sequence of numbers.
\begin{lemma}\cite[Chapter 1, Lemma 4.2]{DiB93_10}\label{L2.4}
    Let consider $M,b>1$ and $\kappa,\delta>0$ be given. For every $j\in\mathbb{N}$, assume that
    \begin{align*}
        Y_{j+1}\leq Mb^j\left(Y_j^{1+\delta}+Z_j^{1+\kappa}Y_j^\delta\right),\,\,\textup{ and }\,\, Z_{j+1}\leq Mb^j\left(Y_j+Z_j^{1+\kappa}\right).
    \end{align*}
    Additionally, assume that $Y_0+Z_0^{1+\kappa}\leq (2M)^{-\frac{1+\kappa}{\alpha}}b^{-\frac{1+\kappa}{\alpha^2}}$ for $\alpha=\min\{\kappa,\delta\}$. Then we have the following 
    \begin{align*}
        \lim_{j\rightarrow\infty}Y_j=0=\lim_{j\rightarrow 0} Z_j.
    \end{align*} 
\end{lemma}
The following lemma gives a relation between the measure of the three level sets of function. This prevents a function from jumping abruptly between two levels without having a significant intermediate region. This lemma is known as {\em De Giorgi Isoperimetric Lemma}.
\begin{lemma}\label{L2.5}
    Let $k<l$ be real numbers and $u\in W^{1,1}(B_\rho(\xi))$. Then there exists a constant $\gamma=\gamma(Q,p)$, independent of $u,\rho,\xi,k,l$ such that
    \begin{align*}
        (l-k)|\{u>l\}|\leq \gamma\frac{\rho^{Q+1}}{|\{u<k\}|}\int_{B_\rho(\xi)\cap \{k<u<l\}}|\nabla_{\hn}u|\,d\xi.
    \end{align*}
\end{lemma}
\begin{proof}
    The significant parts of this proof follows the steps of \cite[Lemma 2.3]{KMMJP21_24}. Set, $A:=\{x\in B_\rho(\xi):u\leq k\}$. Consider $|A|>0$, otherwise the result will be obvious. Define
    \begin{align*}
        u:=\begin{cases}
            &\min\{u,l\}-k,\,\,\,\,\,\,\,\textup{if }\,u>k,\\
            & 0,\,\,\,\,\,\,\,\,\,\,\,\,\,\,\,\,\,\,\,\,\,\,\,\,\,\,\,\,\,\,\,\,\,\,\,\,\textup{ if }\,u\leq k.
        \end{cases}
    \end{align*}
    Then, we have
    \begin{align*}
        \int_{B_\rho(\xi)}|u-u_{B_\rho(\xi)}|dx=\int_{B_\rho(\xi)\setminus A}|u-u_{B_\rho(\xi)}|dx+\int_{A}|u_{B_\rho(\xi)}|dx,
    \end{align*}
    where $u_{B_\rho(\xi)}:=\fint_{B_\rho(\xi)}|u|dx$. Consequently, 
    \begin{align}\label{eq2.2}
        |u_{B_\rho(\xi)}|\leq \frac{1}{|A|}\int_{B_\rho(\xi)}|u-u_{B_\rho(\xi)}|dx.
    \end{align}
    Also, it is obvious that 
    \begin{align}\label{eq2.3}
        \int_{B_\rho(\xi)}|u|dx=\int_{\{u>l\}\cap B_\rho(\xi)}|u|dx+\int_{\{k<u\leq l\}\cap B_\rho(\xi)}|u|dx.
    \end{align}
 Now, with the help of (\ref{eq2.3})
 \begin{align}\label{eq2.4}
     (l-k)|\{u>l\}\cap B_\rho(\xi)|\leq \int_{B_\rho(\xi)}|u|dx&\leq\int_{B_\rho(\xi)}|u-u_{B_\rho(\xi)}|dx+\int_{B_\rho(\xi)}|u_{B_\rho(\xi)}|dx\nonumber\\
     &\leq 2\frac{|B_\rho(\xi)|}{|A|}\int_{B_\rho(\xi)}|u-u_{B_\rho(\xi)}|dx.
 \end{align}
At this point, we recall the Poincar\'e type inequality in Heisenberg group (see \cref{T2.6}) i.e. 
\begin{align*}
    \int_{B_\rho(\xi)}|u-u_{B_\rho(\xi)}|^pdx\leq C(Q,p)\rho^p\int_{B_\rho(\xi)}|\nabla_{\hn}u|^pdx,
\end{align*}
for all $u\in W^{1,p}(B_\rho(\xi))$ with $1\leq p<\infty$. Finally, considering this Poincar\'e inequality with $p=1$ on the right hand side of (\ref{eq2.4}), we obtain the result. 
\end{proof}

Next, we will discuss some important series of results, the {\em parabolic Sobolev embedding} type results. This type of results have been extensively used during the proof of {\em local boundedness} result and more. We start with a version of {\em Sobolev-Poincar\'e inequality} followed by {\em Sobolev inequality}, both on the Heisenberg group.
\begin{theorem}(see \cite{Jer86_15,Lu94_16})\label{T2.6}
    Let $B\subset \mathbb{H}^N$ be any ball in $\mathbb{H}^N$ and let $1
    leq p<\infty$. Then for every $u\in W^{1,p}(B)$, there exists a constant $C=C(Q)$ such that
    \begin{align*}
        \int_B|u-u_B|^p\,d\xi\leq C|\textup{diam}\,B|^p\int_B|\nabla_{\hn}u|^p\,d\xi,
    \end{align*}
\end{theorem}
where $u_B:=\fint_{B}u(\xi)\,d\xi$.
\begin{theorem}(\cite{MZ21_17})\label{T2.7}
    Let $u\in W^{1,2}(\Omega)$, $B\Subset\Omega$. Then there exists a constant $C=C(Q)$ such that 
    \begin{align*}
        \|u\|_{L^{2^*}(B)}\leq C\|\nabla_{\hn}u\|_{L^2(B)},
    \end{align*}
    where $2^*=\frac{2Q}{Q-2}$. 
\end{theorem}
The next result shows that $W^{1,2}(\Omega)$ is contained in $W^{s,2}(\Omega)$ for any bounded set $\Omega\subset\hn$. This result is a direct consequence of the result \cite[Theorem 2.4]{CTR24_18}.
\begin{lemma}(\cite[Lemma 1.1]{Gar25_19}, \cite[Theorem 2.4]{CTR24_18})\label{L2.8}
    Let $\Omega$ be a bounded domain in $\hn$ and $0<s<1$ and $u\in W_0^{1,2}(\Omega)$. Then there exists a constant $C=C(Q,s,\Omega)>0$ such that
    \begin{align*}
        \iint_{\Omega\times\Omega}\frac{|u(\xi)-u(\eta)|^2}{\dist[\xi]{\eta}^{Q+2s}}\,d\xi d\eta\leq C\int_{\Omega}|\nabla_{\hn}u(\xi)|^2\,d\xi.
    \end{align*}
\end{lemma}

Next all the theorems and lemmas will be useful for the result {\em Parabolic Sobolev Inequality}. 
\begin{lemma}\label{L2.9}
    Let, $p>q>1$ satisfying $r<2^*$. Let, $\theta\in (0,1)$ be such that 
    \begin{align}\label{eq2.5}
        \theta\left(\frac{1}{2}-\frac{1}{Q}\right)+\frac{1-\theta}{q}=\frac{1}{r}.
    \end{align}
    Then there exists a positive constant $C=C(Q,q,p,\theta)$ such that
    \begin{align*}
        \|u\|_{L^r(B)}\leq C\|u\|_{L^q(B)}^{1-\theta}\|u\|_{W^{1,2}(B)}^\theta.
    \end{align*}
\end{lemma}
\begin{proof}
    This result can be proved similarly to Theorem C.2 and Theorem C.3 of \cite[Appendix]{KT25_6} replacing $s$ with $1$ as we are considering $u\in W^{1,2}(B)$. For this purpose, we need to consider Sobolev embedding result (Theorem \ref{T2.7}). Moreover, note that, in that context we need an extension result for Sobolev spaces in Heisenberg group which can be found in the thesis of Nhieu \cite{Nhi96_50}.
\end{proof}

\begin{theorem}(Parabolic Sobolev Inequality)\label{T2.10}
    Let $\xi_0\in\hn$, $0<T1<T_2$, $R>0$ and $I=(T_1,T_2)$. Then for every $$u\in L^2_{\textup{loc}}(I;W^{1,2}_{\textup{loc}}(B_R(\xi_0))\cap L^\infty_{\textup{loc}}(I;L^2_{\textup{loc}}(B_R(\xi_0))),$$ there exists a positive constant $C=C(Q)$ such that
    \begin{align*}
        \int_I\int_{B_R(\xi_0)}|u(\xi,t)|^{2\left(1+\frac{2}{Q}\right)}d\xi dt\leq C&\left(\int_I\int_{B_R(\xi_0)}|\nabla_{\hn}u(\xi,t)|^2\,d\xi dt+R^{-2}\int_I\int_{B_R(\xi_0)}|u(\xi,t)|^2\,d\xi dt\right)\\&\quad \times\left(\esssup_{t\in I}\int_{B_R(\xi_0)}|u(\xi,t)|^2\,d\xi\right)^\frac{2}{Q}.
    \end{align*}
\end{theorem}
\begin{proof}
    We omit this proof as it is similar to the proof of \cite[Lemma 2.3]{FSZ22_51}.
\end{proof}

\begin{remark}\label{R2.11}
    It holds that for two positive numbers $r$ and $m$ satisfying
   \begin{align}\label{eq2.6}
       \frac{1}{r}+\frac{1}{m}\left(\frac{2}{Q}+\frac{2}{q}-1\right)=\frac{1}{q}
   \end{align}
    with admissible range 
    \begin{align*}
        r\in \left[q,\frac{2Q}{Q-2}\right],\,\, m\in[2,\infty).
    \end{align*}
    Then, there exists a constant $C=C(Q,r,m)$ such that
    \begin{align}\label{eq2.7}
       \|f\|^{\frac{Qrm}{2r+Qm}}_{L^{r,m}(B_R(\xi_0)\times I)} \leq C&\Bigg( \int_I[f(\cdot,t)]^2_{W^{1,2}(B_R(\xi_0))}\,dt + R^{-2} \|f\|^2_{L^2(I;L^2(B_R(\xi_0))} \nonumber\\
        &\quad+\sup_{t\in I} \|u(\cdot,t)\|^q_{L^q(B_R(\xi_0))}\Bigg).
    \end{align}
\end{remark}
\begin{proof}
    This result follows the outline of the proof of \cite[Theorem 1.10]{KT25_6}, which is valid for fractional Sobolev spaces. Since $Q>2$, then for $(r,m)\in \left(\frac{2Q}{Q-2},2\right)$, (\ref{eq2.7}) follows from Theorem \ref{T2.7}. Otherwise, if we take $\theta=\frac{2}{m}$, then (\ref{eq2.5}) holds due to (\ref{eq2.6}). Then proceeding in a similar manner to \cite[Theorem 1.10]{KT25_6}, we obtain the required inequality. 
\end{proof}
\section{Energy Estimate}\label{S3}
\begin{lemma} \textup{(Caccioppoli Inequality I)}\label{L3.1}
    Let $(\xi_0,t_0)\in\Omega_T$ and consider some positive real numbers $l,r,R$ satisfying $0<l<r<R$. Let $u$ be local weak sub(super)-solution to the equation (\ref{eq1.1}) in $\Omega_T$ and take ball $B_r(\xi_0)$ such that $\bar{B}_r(\xi_0)\subset \Omega$. Consider cut-off functions $\nu\in C^\infty(\mathbb{R})$ with $\nu(t)=0$ on $t\leq t_0-\theta_2$ with $\theta_2>0$ satisfying $(t_0-\theta_2,t_0)\subset (t_0-R^2,t_0)\subset (0,T)$ and  $\psi\in C^\infty_0(B_\frac{r+l}{2}(\xi_0))$ with $0\leq \psi\leq 1$ in $\hn$. Then there exists a constant c=c(\textup{\texttt{data}}) such that, 
   \begin{align}\label{C'ppoli-I}
       &\esssup_{t_0-\theta_2<t<t_0}\int_{B_r}(w^2_\pm\psi^2\nu^2)(\xi,t)d\xi+\int_{t_0-\theta_2}^{t_0}\int_{B_r}(w_\pm\psi^2\nu^2)(\xi,t)\Bigg\{\int_{\hn}\frac{w_\mp(\eta,t)}{|\eta^{-1}\circ \xi|^{Q+2s}_{{\hn}}}d\eta\Bigg\}d\xi dt\nonumber\\
    &\quad+\int_{t_0-\theta_2}^{t_0}\iint_{B_r\times B_r}\frac{|(w_\pm\psi\nu)(\xi,t)-(w_\pm\psi\nu)(\eta,t)|^2}{|\eta^{-1}\circ \xi|^{Q+2s}_{\mathbb{H}^N}}d\xi d\eta dt\nonumber\\
   & \quad\quad + \int_{t_1}^{t_2}\int_{B_r(\xi_0)}(|\nabla_{\hn}w_\pm|^2\psi^2\nu^2)(\xi,t)\;d\xi\,dt\nonumber\\
    &\quad\quad\quad \leq c\Bigg(\frac{r^{1-s}}{r-l}\Bigg)^2||w_\pm||^2_{L^2(t_0-\theta_2,t_0;L^2(B_r(\xi_o)))}+c\int_{t_0-\theta_2}^{t_0}\int_{B_r(\xi_0)}(w_\pm^2(\nabla_{\hn}\psi)^2\nu^2)(\xi,t)\;d\xi\;dt\tag{C'ppoli-I}\\
    &\quad\quad\quad\quad +\int_{t_0-\theta_2}^{t_0}\int_{B_r}(w_\pm\psi^2\nu^2)(\xi,t)\Bigg\{\int_{\mathbb{H}^N\setminus B_r}\frac{w_{\pm}(\eta,t)}{|\eta^{-1}\circ \xi|^{Q+2s}_{{\hn}}}dy\Bigg\}d\xi dt\nonumber\\
    &\quad\qquad\qquad\mp\mathfrak{C}\int_{t_0-\theta_2}^{t_0}\int_{B_r}R^{-2}\textup{Tail}(u_\pm(t);R,\xi_0)(w_\pm\psi^2\nu^2)(\xi,t)d\xi dt\nonumber\\
    &\quad\qquad\qquad\qquad +\int_{t_0-\theta_2}^{t_0}\int_{B_r}(w^2_\pm\psi^2\nu)(\xi,t)\partial_t\nu(t)d\xi dt,\nonumber
   \end{align}
   where $w:=u-g-k$ with a level $k\in\mathbb{R}$ and $g(t)=\mathfrak{C}R^{-2}\int_{t_0-R^{2}}^t\textup{Tail }(u_\pm(s);R,\xi_0)\,ds.$
\end{lemma}
\begin{proof}
    Most of the proof of this result is done in the respective part of \cite[Lemma 3.1]{KT25_6} by changing $p=2$. The only part which we need to estimate rigorously is the third term of the left hand side of (\ref{wf}). For this reason we only present all the estimations of the concerned part only. Moreover, we prove this result only for the subsolution case since all the processes for a supersolution are same.  Before proceeding, we need to borrow some notations and constructions from the proof of \cite[Lemma 3.1]{KT25_6}. These are as follows
    \begin{itemize}
        \item Let us take $t_1=t_0-\theta_2$ and take an arbitrary $t_2$ in $(t_0-\theta_2,t_0)$.
        \item For every $t\in (0,T)$, choose $\phi=\phi(x,t)$ whose spacial support contained in $B_\frac{r+l}{2}(\xi_0)$. 
    \end{itemize}
Using $\phi^\varepsilon(\xi,t)$ as a test function in the weak subsolution formula (\ref{wf})
\begin{align}\label{eq3.1}
    &\int_{B_r(\xi_0)}u(\xi,t_2)\phi^\varepsilon(\xi,t_2)d\xi-\int_{B_r(\xi_0)}u(\xi,t_1)\phi^\varepsilon(\xi,t_1)d\xi\nonumber\\
    &\quad\int_{t_1}^{t_2}\int_{B_r(\xi_0)}u(\xi,t)\phi_t^\varepsilon(\xi,t)d\xi\,dt+\underbrace{\int_{t_1}^{t_2}\int_{B_r(\xi_0)}\nabla_{\hn}u(\xi,t)\nabla_{\hn}\phi^\varepsilon(\xi,t)\,d\xi\,dt}_{I^\varepsilon} \nonumber\\
             &\qquad+\int_{t_1}^{t_2}\iint_{\hn\times\hn}\frac{1}{2}(u(\xi,t)-u(\eta,t))(\phi^\varepsilon(\xi,t)-\phi^\varepsilon(\eta,t))K(\xi,\eta,t)\;d\eta\,d\xi \,dt\leq 0.
\end{align}
As we discussed, we will estimate only $I^\varepsilon$. Consider the smooth cut-off function $\psi$ on $B_r(\xi_0)$ such that $0\leq \psi\leq 1$ in $\hn$ and 
\begin{align*}
    \begin{cases}
        & \mbox{supp }\psi= B_\frac{r+l}{2}(\xi_0),\\
        &\psi\equiv 1\;\;\;\mbox{ in } B_l(\xi_0),\\
        & \nabla_{\hn}\psi\leq \frac{C}{r-l}.
    \end{cases}
\end{align*}
Next, we construct text function $\phi(\xi,t)$ as
$$ \phi(\xi,t)=w^\varepsilon_+(\xi,t)\psi^2(\xi)\nu^2(t)=w^\varepsilon_+\psi^2\nu^2(\xi,t),$$
where $ w^\varepsilon_+(\xi,t):=(u^\varepsilon(\xi,t)-g^\varepsilon(t)-k)_+,\;\;w_+(\xi,t):=(u(\xi,t)-g(t)-k)_+$ and $g$ is denoted by $g(t):=\mathfrak{C}R^{-2}\int_{t_0-R^{2}}^t\textup{Tail}(u_+(\varsigma);R,\xi_0)d\varsigma$. Further, for simplicity we denote
\begin{align*}
    \psi^i(\xi)\nu^j(t)\;\mbox{ as }\,\psi^i\nu^j(\xi,t)\,\mbox{ for some non-negative real numbers }i,j.
\end{align*}
Using this test function, $I^\varepsilon$ becomes
\begin{align}\label{eq3.2}
    I&=\int_{t_1}^{t_2}\int_{B_r(\xi_0)}\nabla_{\hn}u(\xi,t)\nabla_{\hn}(w^\varepsilon_+\psi^2\nu^2)^\varepsilon(\xi,t)\,d\xi\,dt\nonumber\\
    &=\int_{t_1}^{t_2}\int_{B_r(\xi_0)}\nabla_{\hn}u(\xi,t)((\nabla_{\hn}w^\varepsilon_+)\psi^2\nu^2)^\varepsilon(\xi,t)+2\int_{t_1}^{t_2}\int_{B_r(\xi_0)}\nabla_{\hn}u(\xi,t)(w^\varepsilon_+\psi(\nabla_{\hn}\psi)\nu^2)^\varepsilon(\xi,t)\nonumber\\
    &\geq \underbrace{\int_{t_1}^{t_2}\int_{B_r(\xi_0)}\nabla_{\hn}u(\xi,t)((\nabla_{\hn}w^\varepsilon_+)\psi^2\nu^2)^\varepsilon(\xi,t)\;d\xi\,dt}_{D_1^\varepsilon}-2\underbrace{\int_{t_1}^{t_2}\int_{B_r(\xi_0)}\nabla_{\hn}u(\xi,t)(w^\varepsilon_+\psi(\nabla_{\hn}\psi)\nu^2)^\varepsilon(\xi,t)\;d\xi\,dt.}_{D_2^\varepsilon}
\end{align}
{\bf Passage to limit as $\varepsilon\rightarrow 0$ : } Next, we estimate $D_i^\varepsilon$ for $i=1,2$ successively. Begin with $D_1^\varepsilon$, using H\"older inequality,
\begin{align}\label{eq3.3}
    \int_{t_1}^{t_2}&\int_{B_r(\xi_0)}\nabla_{\hn}u(\xi,t)((\nabla_{\hn}w^\varepsilon_+)\psi^2\nu^2)^\varepsilon(\xi,t)\;d\xi\,dt- \int_{t_1}^{t_2}\int_{B_r(\xi_0)}\nabla_{\hn}u(\xi,t)((\nabla_{\hn}w_+)\psi^2\nu^2)(\xi,t)\;d\xi\,dt\nonumber\\
    &\leq \Bigg(\int_{t_1}^{t_2}\int_{B_r(\xi_0)}|\nabla_{\hn}u(\xi,t)|^2\,d\xi\,dt\Bigg)^{1/2}\Bigg(\int_{t_1}^{t_2}\int_{B_r(\xi_0)}|((\nabla_{\hn}w^\varepsilon_+)\psi^2\nu^2)^\varepsilon-((\nabla_{\hn}w_+)\psi^2\nu^2)|^2(\xi,t)\,d\xi\,dt\Bigg)^{1/2}\nonumber\\
   &\stackrel{Lemma 
   \,\ref{L2.1}}{\rightarrow} 0,\;\;\;\;\mbox{ as } \varepsilon\rightarrow 0
\end{align}

The second step makes sense since $\nabla_{\hn}u\in L^2(t_1,t_2;L^2(B_r(\xi_0))$. Next, by a similar argument, we have the following
\begin{align}\label{eq3.4}
    D_2^\varepsilon\rightarrow D_2:= \int_{t_1}^{t_2}\int_{B_r(\xi_0)}\nabla_{\hn}u(\xi,t)(w_+\psi(\nabla_{\hn}\psi)\nu^2)(\xi,t)\;d\xi\,dt.
\end{align}\\[2mm]
Combining results of (\ref{eq3.3}) and (\ref{eq3.4}) into (\ref{eq3.2}), we have for $\varepsilon\rightarrow 0$, 
\begin{align*}
    &\int_{t_1}^{t_2}\int_{B_r(\xi_0)}\nabla_{\hn}u(\xi,t)((\nabla_{\hn}w_+)\psi^2\nu^2)(\xi,t)\;d\xi\,dt\\
    &\qquad -2\int_{t_1}^{t_2}\int_{B_r(\xi_0)}\nabla_{\hn}u(\xi,t)(w_+\psi(\nabla_{\hn}\psi)\nu^2)(\xi,t)\;d\xi\,dt\leq I\longleftarrow I^\varepsilon
\end{align*}
Next, using the Young's inequality, we can rewrite last inequality as follows
\begin{align}\label{eq3.5}
    I&\geq \int_{t_1}^{t_2}\int_{B_r(\xi_0)}(|\nabla_{\hn}w_+|^2\psi^2\nu^2)(\xi,t)\;d\xi\,dt-\epsilon \int_{t_1}^{t_2}\int_{B_r(\xi_0)}(|\nabla_{\hn}w_+|^2\psi^2\nu^2)(\xi,t)\;d\xi\;dt\nonumber\\
    &\quad - C(\epsilon,2)\int_{t_1}^{t_2}\int_{B_r(\xi_0)}(w_+^2(\nabla_{\hn}\psi)^2\nu^2)(\xi,t)\;d\xi\;dt\nonumber\\
    & \geq \frac{1}{2} \int_{t_1}^{t_2}\int_{B_r(\xi_0)}(|\nabla_{\hn}w_+|^2\psi^2\nu^2)(\xi,t)\;d\xi\,dt-c\int_{t_1}^{t_2}\int_{B_r(\xi_0)}(w_+^2(\nabla_{\hn}\psi)^2\nu^2)(\xi,t)\;d\xi\;dt
\end{align}
with $\epsilon=\frac{1}{2}$. Hence we have the desired result.
\end{proof}
\begin{remark}\label{R3.2}
    A different form of the estimate in \ref{C'ppoli-I} is required in order to prove the Harnack inequality. This form can be obtained by some simple modifications to the proof of the Caccioppoli inequality in \cite{KT25_6}. We omit these modifications and simply state the required inequality.
    \begin{align}\label{eq3.6}
        &\int_{B_r}(w_\pm\psi\nu)^2(\xi,t)d\xi\Bigg|_{t_0-\theta_2}^{t_0}+\int_{t_0-\theta_2}^{t_0}\int_{B_r}(w_\pm\psi^2\nu^2)(\xi,t)\Bigg\{\int_{\hn}\frac{w_\mp(\eta,t)}{|\eta^{-1}\circ \xi|^{Q+2s}_{{\hn}}}d\eta\Bigg\}d\xi dt\nonumber\\
    &\quad+\int_{t_0-\theta_2}^{t_0}\iint_{B_r\times B_r}\frac{|(w_\pm\psi\nu)(\xi,t)-(w_\pm\psi\nu)(\eta,t)|^2}{|\eta^{-1}\circ \xi|^{Q+2s}_{\mathbb{H}^N}}d\xi d\eta dt\nonumber\\
   & \quad\quad + \int_{t_1}^{t_2}\int_{B_r(\xi_0)}(|\nabla_{\hn}w_\pm|^2\psi^2\nu^2)(\xi,t)\;d\xi\,dt\nonumber\\
    &\quad\quad\quad \leq C\int_{t_0-\theta_2}^{t_0}\iint_{B_r\times B_r}\max\{w_\pm^2(\xi,t),w^2_\pm(\eta,t)\}\frac{|\psi(\xi)-\psi(\eta)|^2}{|\eta^{-1}\circ \xi|^{Q+2s}_{\mathbb{H}^N}}\nu^2(t)\;d\xi d\eta dt\\
    &\quad\quad\quad\quad+C\int_{t_0-\theta_2}^{t_0}\int_{B_r(\xi_0)}(w_\pm^2(\nabla_{\hn}\psi)^2\nu^2)(\xi,t)\;d\xi dt\nonumber\\
    &\quad\quad\quad\quad\quad +\int_{t_0-\theta_2}^{t_0}\int_{B_r}(w_\pm\psi^2\nu^2)(\xi,t)\Bigg\{\int_{\mathbb{H}^N\setminus B_r}\frac{w_{\pm}(\eta,t)}{|\eta^{-1}\circ \xi|^{Q+2s}_{{\hn}}}dy\Bigg\}d\xi dt\nonumber\\
    &\qquad\qquad\qquad\mp\mathfrak{C}\int_{t_0-\theta_2}^{t_0}\int_{B_r}R^{-2}\textup{Tail}(u_\pm(t);R,\xi_0)(w_\pm\psi^2\nu^2)(\xi,t)d\xi dt\nonumber\\
    &\quad\qquad\qquad\qquad +\int_{t_0-\theta_2}^{t_0}\int_{B_r}(w^2_\pm\psi^2\nu)(\xi,t)\partial_t\nu(t)d\xi dt\nonumber
    \end{align}
\end{remark}\vskip 2mm
Next, we do further calculation on the tail term on the right hand side of the previous estimate (\ref{C'ppoli-I}). Mainly we split it into the {\em middle distance} term and {\em faraway} term. We omit the proof as it has been done in detail in Lemma 3.2 of \cite{KT25_6}. The only change is the additional terms on both sides due to the presence of the local operator.  
\begin{lemma} \textup{(Caccioppoli Inequality II)}\label{L3.3}
    Consider $(\xi_0,t_0)\in\Omega_T$. Let consider positive numbers $\sigma,r,l,R$ such that $0<\sigma R\leq l<\frac{l+r}{2}\leq\frac{(3+\sigma)R}{4}\leq R$. Consider   $B_r(\xi_0)$ such that $\bar{B}_r(\xi_0)\subset \Omega$ and take two non-negative functions $\psi\in C^\infty_0(B_\frac{r+l}{2}(\xi_0))$ satisfying $0\leq \psi\leq 1$ in $\hn$ and $\nu\in C^\infty (\mathbb{R})$ such that $\nu\equiv 0$ on $t\leq t_0-\theta_2$ where $\theta_2>0$ satisfies $(t_0-\theta_2,t_0)\subset (t_0-R^2,t_0)\subset (0,T)$. Let $u$ be a local super(sub)-solution to the problem (\ref{eq1.1}). Then there exists constants $c=c(\textup{\texttt{data}})$ and $\mathfrak{C}=\mathfrak{C}(\textup{\texttt{data,}},\sigma)$ such that
    \begin{align}\label{C'ppoli-II}
        &\esssup_{t_0-\theta_2<t<t_0}\int_{B_r(\xi_0)}(w_\pm\psi\nu)^2(\xi,t)d\xi + \int_{t_0-\theta_2}^{t_0}\int_{B_r(\xi_0)}(|\nabla_{\hn}w_\pm|\psi\nu)^2(\xi,t)\;d\xi\,dt\nonumber\\
			&\quad+\int_{t_0-\theta_2}^{t_0}\iint_{B_r(\xi_0)\times B_r(\xi_0)}\frac{|(w_\pm\psi\nu)(\xi,t)-(w_\pm\psi\nu)(\eta,t)|^2}{|\eta^{-1}\circ \xi|^{Q+2s}_{\hn}}\; d\xi\,d\eta\,dt\nonumber\\
			&\quad\quad	\leq c\Bigg(\frac{r^{1-s}}{r-l}\Bigg)^2\|w_\pm\|^2_{L^2(t_0-\theta_2,t_0;L^2(B(\xi_0,r)))}\tag{C'ppoli-II}\\
            &\quad\quad\quad +c\int_{t_0-\theta_2}^{t_0}\int_{B_r(\xi_0)}(w_\pm(\nabla_{\hn}\psi)\nu)^2(\xi,t)\;d\xi\;dt+\int_{t_0-\theta_2}^{t_0}\int_{B_r(\xi_0)}(w_\pm^2\psi^2\nu)(\xi,t)\partial_t\nu(t)\;d\xi\,dt\nonumber\\
			&\quad\quad\quad\quad+cR^{-2s} \left(\frac{R}{r-l}\right)^{Q+2s}\int_{t_0-\theta_2}^{t_0}\int_{B_r(\xi_0)}w_\pm(\xi,t)\psi^2(\xi)\nu^2(t)\Bigg(\fint_{B_R(\xi_0)} u_\pm(\eta,t)\,d\eta\Bigg)d\xi\,dt.\nonumber
    \end{align}
    where $w$ is defined as in (\ref{C'ppoli-I}).
\end{lemma}
\section{Local Boundedness Result} \label{S4}
This section is devoted to the proof of the local boundedness. For this section, we took the inspiration from the respective section of \cite{KT25_6}. In the concerned paper, authors show the local boundedness result for $p>\frac{2Q}{Q+2s}$. As we are considering $p=2$, we can rely on the methods as described in \cite{KT25_6}.  
\subsection{Parameters}Our choice of parameters is as same as in \cite{KT25_6}. Only change we want to mention is the replacement of $p$ and $s$ with $2$ and $1$ respectively. \vskip 1mm 
Define 
\begin{align}
    \kappa=\frac{1}{Q},
\end{align}
and 
\begin{align}
    \hat{q}=2(1+\kappa).
\end{align}\vskip 2mm
Before moving to the next lemma, we want to introduce a notation of a truncated set as follows,
\begin{align}
    \mathcal{A}_\pm^g(t;\xi_0,\rho,k)=\{\xi\in B_\rho(\xi_0):(u(\xi,t)-g(t)-k)_\pm>0\},
\end{align}
where $\rho>0,k\geq 0,\xi_0\in\Omega,t\in (0,T]$ and $B_\rho(\xi_0)\subset\Omega$.
\begin{lemma}\label{L4.1}
    Consider $(\xi_0,t_0)\in\Omega
    _T$. Let consider positive numbers $\sigma,r,l,R$ such that $0<\sigma R\leq l<\frac{l+r}{2}<\frac{(3+\sigma)R}{4}\leq R$. With $(\xi_0,t_0)+\mathcal{Q}_l^\ominus\subset(\xi_0,t_0)+\mathcal{Q}_r^\ominus\subset [t_0-R^2,t_0]\times B_R(\xi_0)\Subset \Omega_T$ and $\sigma\in (0,1)$. Consider   $B_r(\xi_0)$ such that $\bar{B}_r(\xi_0)\subset \Omega$ and take a non-negative functions  $\psi\in C^\infty_0(B_\frac{r+l}{2}(\xi_0))$ satisfying $0\leq \psi\leq 1$ in $\hn$ with $|\nabla_{\hn}\psi|\leq\frac{C}{r-l}$ and $\nu$ supported in $(t_0-\theta_2,t_0]$ such that $\partial_t\nu\leq \frac{C}{\theta_2}$ for some $\theta_2>0$. Let, $u$ be a weak subsolution to the equation (\ref{eq1.1}) and recall $w_+=(u-g(t)-k)_+$ where $g(t)=\mathfrak{C}R^{-2}\int_{t_0-R^2}^{t_0}\textup{Tail }(u_+(s);R,\xi_0)ds$ with $\mathfrak{C}=\mathfrak{C}(\textup{\texttt{data}},\sigma)$ (as in Lemma \ref{L3.1}) then there exists a constant $C=C(\textup{\texttt{data}})$ such that
    \begin{align*}
        &\int_{t_0-l^2}^{t_0}\int_{B_r(\xi_0)}(|\nabla_{\hn}w_\pm|^2\psi^2\nu^2)(\xi,t)\;d\xi\,dt+\esssup_{t\in (t_0-l^{2},t_0)}\int_{B_l(\xi_0)} w^2_+(\xi,t)d\xi\\
			&\qquad\leq C\Bigg[\frac{r^{2(1-s)}}{(r-l)^2}+\frac{2}{(r-l)^2}\Bigg]\iint_{z_0+\mathcal{Q}^\ominus_r}w^2_+(\xi,t)d\xi dt\\
			&\qquad\qquad+Ck^2\Big(\frac{R}{r-l}\Big)^{2(Q+2s)}r^{-1}\Bigg(\int_{t_0-r^{2}}^{t_0}|\mathcal{A}^g_+(t;\xi_0,r,k)|^2 dt\Bigg)^\frac{1}{2},
    \end{align*}
    where $k$ satisfies
    $$k\geq \Big[ \fint_{t_0-R^{2}}^{t_0} \left( \fint_{B_R(\xi_0)} u_+(\eta,t) \,d\eta\right)^2\,dt \Big]^\frac{1}{2}.$$
\end{lemma}
\begin{proof}
    This lemma is essentially an extension of the Lemma 4.1 of \cite{KT25_6} to the mixed local-nonlocal case. As most of the calculations are similar, we shall only mention the changes.\vskip 1mm
    Assume 
    \begin{align*}
       k\geq \Big[ \fint_{t_0-R^{2}}^{t_0} \left( \fint_{B_R(\xi_0)} u_+(\eta,t) \,d\eta\right)^2\,dt \Big]^\frac{1}{2},
    \end{align*}
    and insert $\theta_2=r^{2}$ into the energy inequality (\ref{C'ppoli-II})  to get
\begin{align*}
			&\esssup_{t_0-r^{2s}<t<t_0}\int_{B_r(\xi_0)}(w_+^2\psi^p\nu^2)(\xi,t)d\xi+\int_{t_0-\theta_2}^{t_0}\int_{B_r(\xi_0)}(|\nabla_{\hn}w_\pm|^2\psi^2\nu^2)(\xi,t)\;d\xi\,dt\\
			&\qquad	\leq C\Bigg(\frac{r^{1-s}}{r-l}\Bigg)^2\|w_+\|^2_{L^2(t_0-r^{2},t_0;L^2(B_r(\xi_0)))}+	\underbrace{\int_{t_0-r^{2}}^{t_0}\int_{B_r(\xi_0)}(w_+^2\psi^2)(\xi,t)\nu(t)\partial_t\nu(t)\;d\xi\,dt}_{I}\\
			&\quad\quad\quad+\underbrace{CR^{-2s}\left(\frac{R}{r-l}\right)^{Q+2s}\int_{t_0-r^{2}}^{t_0}\int_{B_r(\xi_0)}w_+(\xi,t)\psi^2(\xi)\nu^2(t)\Bigg(\fint_{B_R(\xi_0)} u_+(\eta,t)\,d\eta\Bigg)d\xi\,dt}_{II}\\
            & \quad\quad\quad\quad+C\underbrace{\int_{t_0-r^{2}}^{t_0}\int_{B_r(\xi_0)}(w_\pm^2(\nabla_{\hn}\psi)^2\nu^2)(\xi,t)\;d\xi\;dt.}_{III}
\end{align*}
The estimation of $I$ follows from the fact that $\partial_t\nu<\frac{C}{r^2}$ that is
\begin{align}
    I\leq \frac{C}{r^2}\int_{t_0-r^{2}}^{t_0}\int_{B_r(\xi_0)}w^2_+(\xi,t)\,d\xi \,dt.
\end{align} Next, using the fact $|\nabla_{\hn}\psi|\leq \frac{C}{r-l}$, we have in $III$,
\begin{align}
    III\leq \frac{C}{(r-l)^2}\int_{t_0-r^{2}}^{t_0}\int_{B_r(\xi_0)}w^2_+(\xi,t)\,d\xi \,dt.
\end{align}
Now, we try to estimate $II$. This estimation is inspired from the corresponding steps of  lemma 4.1 in \cite{KT25_6}.\vskip 1mm
{\bf Estimation of II: } We start this estimation by noting that
\begin{align*}
    \mathcal{G}(t):=\fint_{B_R(\xi_0)}u_+(\eta,t)\,d\eta\in L^{\infty,2}_{\mbox{loc}}(\Omega_T).
\end{align*}
Now, using the Jensen, H\"older and Youngs' inequality, $II$ implies,
\begin{align*}
    II&=CR^{-2s}\left(\frac{R}{r-l}\right)^{Q+2s}\int_{t_0-r^2}^{t_0}\int_{B_r(\xi_0)}\mathcal{G}(t)(w_+\psi^2\nu^2)(\xi,t)\,d\xi\\
    &\leq CR^{-2s}\left(\frac{R}{r-l}\right)^{Q+2s}\|\mathcal{G}\|_{L^{\infty,2}(z_0+\mathcal{Q}^\ominus_r)}\|w_+\psi\nu\|_{L^{1,2}(z_0+\mathcal{Q}^\ominus_r)}\\
    &\leq C \left(\frac{R}{r-l}\right)^{Q+2s}\|kR^{-s}w_+\psi\nu\|_{L^{1,2}(z_0+\mathcal{Q}^\ominus_r)}\\
    &\leq \frac{Ck^2}{\epsilon}\left(\frac{R}{r-l}\right)^{2(Q+2s)}\frac{1}{R^s}\|\chi_{\{u>g(.)+k\}}\|_{L^{1,2}(z_0+\mathcal{Q}^\ominus_r)}+\frac{\epsilon}{R^s}\|\left(w_+\psi\nu\right)^2\|_{L^{1,2}(z_0+\mathcal{Q}^\ominus_r)}\\
    &\leq \frac{Ck^2}{\epsilon}\left(\frac{R}{r-l}\right)^{2(Q+2s)}r^{-1}\left(\int_{t_0-r^2}^{t_0}|\mathcal{A}^g_+(t;\xi_0,r,k)|^2\right)^\frac{1}{2}+\epsilon\underbrace{\frac{1}{R^s}\|\left(w_+\psi\nu\right)^2\|_{L^{1,2}(z_0+\mathcal{Q}^\ominus_r)}.}_{IV}
\end{align*}
In the sequence of calculations, we use the fact that $\frac{r}{R}<1$ and $r<1$ at the last inequality. Once again using these facts and employing the H\"older inequality in $IV$ to obtain the following
\begin{align*}
    IV&\leq Cr^Q\left(\frac{r}{R^s}\right)\left(\fint_{t_0-r^2}^{t_0}\left(\fint_{B_r}\left(w_+\psi\nu\right)^2\right)^2\right)^\frac{1}{2}\\
    &\leq Cr^Q\left(\frac{r}{R}\right)^s\left(\fint_{t_0-r^2}^{t_0}\left(\fint_{B_r}\left(w_+\psi\nu\right)^{\hat{q}}\right)^\frac{2}{1+\kappa}\right)^\frac{1}{2}\\
    &\leq Cr^Q\left(\frac{r}{R}\right)^s\left(\fint_{t_0-r^2}^{t_0}\left(\fint_{B_r}\left(w_+\psi\nu\right)^{\hat{q}}\right)^2\right)^\frac{2}{2\hat{q}}\\
    &\leq  C\|w_+\psi\nu\|^2_{L^{\hat{q},2\hat{q}}(z_0+\mathcal{Q}^\ominus_r)}.
\end{align*}
Next, we use Theorem \ref{T2.10} such that
\begin{align*}
    IV\leq C\Bigg(&\underbrace{\iint_{z_0+\mathcal{Q}^\ominus_r}|\nabla_{\hn}(w_+\psi\nu)|^2\,d\xi dt}_{V}+r^{-2}\iint_{z_0+\mathcal{Q}^\ominus_r}\left(w_+\psi\nu\right)^2(\xi,t)\,d\xi dt\\
    & +\esssup_{t_0-r^2<t<t_0}\int_{B_r(\xi_0)}\left(w_+\psi\nu\right)^2(\xi,t)\,d\xi\Bigg).
\end{align*}
Next, using the fact that $|\nabla_{\hn}\psi|\leq \frac{C}{r-l}$ in $V$, we have the following,
\begin{align*}
    V\leq \iint_{z_0+\mathcal{Q}^\ominus_r}\left(|\nabla_{\hn}w_+|\psi\nu\right)^2\,d\xi dt+\frac{C}{(r-l)^2}\iint_{z_0+\mathcal{Q}^\ominus_r}w_+^2(\xi,t)\,d\xi dt.
\end{align*}
Then, $II$ will be as follows,
\begin{align*}
    II&\leq Ck^2\left(\frac{R}{r-l}\right)^{2(Q+2s)}R^{2(1-s)}r^{-1}\left(\int_{t_0-r^2}^{t_0}|\mathcal{A}^g_+(t;\xi_0,r,k)|^2\right)^\frac{1}{2}\\
    &\quad\quad +\underbrace{\iint_{z_0+\mathcal{Q}^\ominus_r}\left(|\nabla_{\hn}w_+|\psi\nu\right)^2\,d\xi dt}_{VI}+\frac{C}{(r-l)^2}\iint_{z_0+\mathcal{Q}^\ominus_r}w_+^2(\xi,t)\,d\xi dt\\
    &\quad\quad\quad\quad+\underbrace{r^{-2}\iint_{z_0+\mathcal{Q}^\ominus_r}\left(w_+\psi\nu\right)^2(\xi,t)\,d\xi dt}_{VII}+\underbrace{\esssup_{t_0-r^2<t<t_0}\int_{B_r(\xi_0)}\left(w_+\psi\nu\right)^2(\xi,t)\,d\xi}_{VIII}.
\end{align*}
Merging $VII$ with the estimate $I$, sending $VI,VIII$ to the left, using the facts $(r-l)^2\leq (r^2-l^2)$ and combining all the estimates, we have our result. 
\end{proof}
\subsection{Proof of Local Boundedness Result} Without any loss of generality, we take $(\xi_0,t_0)=(0,0)\in\Omega_T$. This proof follows the idea of \cite{KT25_6}. Moreover, set $B_R(0)=B_R$ such that $\bar{B}_R\subset \Omega$, denote $(0,0)+\mathcal{Q}^\ominus_r=\mathcal{Q}^\ominus_r$. Define
\begin{align*}
    \hat{k}=\left[\fint_{-R^2}^0\left(\fint_{B_R}u_+(\eta,t)\,d\eta\right)^2dt\right]^\frac{1}{2}+1.
\end{align*}
Next, define iterative parameter as follows,
\begin{align*}
    \begin{cases}
        & \rho_j=\left(\sigma+\frac{1-\sigma}{2^{j+1}}\right)R,\\
        &k_j=d+d\left(1-\frac{1}{2^j}\right),
    \end{cases}
\end{align*}
for $\sigma\in (0,1)$, $d>2\hat{k}$ where $d$ will be determined later. Then it is obvious that 
\begin{align*}
    \begin{cases}
        &\rho_j-\rho_{j+1}=\left(\frac{1-\sigma}{2^{j+2}}\right)R ,\\
        & \frac{1}{2}\rho_j\leq \rho_{j+1}\leq \rho_j,\\
        & k_{j+1}-k_j=\frac{d}{2^{j+1}}.
    \end{cases}
\end{align*}
By the definitions of the parameter, it can be verified that 
\begin{align*}
    \frac{\rho_j+\rho_{j+1}}{2}=\left[\sigma+\frac{3(1-\sigma)}{2^{i+3}}\right]R<\left[\frac{3+\sigma}{4}\right]R.
\end{align*}
Moreover, for iterative inequalities, we set
\begin{align}\label{eq4.6}
    \begin{cases}
        &\mathcal{Y}_j=\frac{1}{d^2}\fiint_{\mathcal{Q}^\ominus_{\rho_j}}(u-g(t)-k_j)_+^2\,d\xi dt,\\
        &\mathcal{Z}_j=\left(\fint_{-\rho_j^2}^0\left(\frac{|\mathcal{A}^g_+(t;0,\rho_j,k_{j+1}|}{|B_{\rho_j}|}\right)^2dt\right)^\frac{1}{2(1+\kappa)},\\
        &\Pi_j=\fint_{-\rho^2_j}^0\frac{|\mathcal{A}^g_+(t;0,\rho_j,k_{j+1}|}{|B_{\rho_j}|}\,dt.
    \end{cases}
\end{align}
From (\ref{eq4.6}), we have 
\begin{align}\label{eq4.7}
    \Pi_j\leq (k_{j+1}-k_j)^{-2}d^2\mathcal{Y}_j\leq C2^{2j}\mathcal{Y}_j.
\end{align}
Considering Lemma \ref{L4.1} with $z_0=(0,0), l=\rho_{j+1}, r=\rho_j$ and $k=k_{j+1}$, we obtain
\begin{align}\label{eq4.8}
    \left[\left(u-g-k_{j+1}\right)_+\right]^2_{V^{2,2}\left(\mathcal{Q}^\ominus_{\rho_{j+1}}\right)}&\leq C\left[\frac{\rho_j^{2(1-s)}}{(\rho_j-\rho_{j+1})^2}+\frac{2}{(\rho_j-\rho_{j+1})^2}\right]\iint_{\mathcal{Q}^\ominus_{\rho_j}}\left(u-g-k_{j+1}\right)^2_+\;d\xi dt\nonumber\\
    &\quad\quad+ C\hat{k}^2\left(\frac{R}{\rho_j-\rho_{j+1}}\right)^{2(Q+2s)}{\rho_j}^{-1}\Bigg(\int_{t_0-{\rho_j}^{2}}^{t_0}|\mathcal{A}^g_+(t;0,\rho_j,k_{j+1})|^2 dt\Bigg)^\frac{1}{2}\\
    & =:I+II+III.\nonumber
\end{align}
Next, we estimate $I,II$ and $III$ successively,\vskip 1mm
{\bf Estimation of I :} Using definition in (\ref{eq4.6}), $\rho_j<1$ and $k_{j+1}>k_j$, we have 
\begin{align*}
    I&= C\frac{\rho_j^2}{(\rho_j-\rho_{j+1})^2}\frac{1}{\rho_j^{2s}}\iint_{\mathcal{Q}^\ominus_{\rho_j}}\left(u-g-k_{j+1}\right)_+^2\,d\xi dt\\
    &\leq C\frac{2^{2j}}{(1-\sigma)^2}\frac{1}{\rho_j^2}\iint_{\mathcal{Q}^\ominus_{\rho_j}}\left(u-g-k_j\right)_+^2\,d\xi dt\\
    & \leq C\frac{2^{2j}}{(1-\sigma)^2}\frac{d^2}{\rho_j^2}\mathcal{Y}_j|\mathcal{Q}^\ominus_{\rho_j}|\\
    & \leq C\frac{2^{2j}}{(1-\sigma)^2}\frac{d^2}{\rho_j^2}\mathcal{Y}_j|\mathcal{Q}^\ominus_{\rho_j}|.
\end{align*}\vskip 1mm
{\bf Estimation of II: } Considering the definition in (\ref{eq4.6}),
\begin{align*}
    II&\leq C\frac{2^{2j}}{R^2(1-\sigma)^2}d^2\mathcal{Y}_j|\mathcal{Q}^\ominus_{\rho_j}|\\
    &=C\frac{2^{2j}}{(1-\sigma)^2}\frac{\left[\sigma+\frac{1-\sigma}{2^{j+1}}\right]^2}{\rho_j^2}d^2\mathcal{Y}_j|\mathcal{Q}^\ominus_{\rho_j}|\\
    &\leq C\frac{2^{2j}}{(1-\sigma)^2}\frac{d^2}{\rho_j^2}\mathcal{Y}_j|\mathcal{Q}^\ominus_{\rho_j}|.
\end{align*}
In the sequence of calculations, we have used the fact $\sigma<1$. \vskip 1mm
{\bf Estimation of III: } Using (\ref{eq4.6}), we have 
\begin{align*}
    III&\leq Cd^2\left(\frac{2^j}{1-\sigma}\right)^{2(Q+2s)}\frac{1}{\rho_j}\rho_j^{Q+1}\mathcal{Z}_j^{1+\kappa}\\
    &=C\left(\frac{2^j}{1-\sigma}\right)^{2(Q+2s)}\frac{d^2}{\rho_j^2}|\mathcal{Q}^\ominus_{\rho_j}|\mathcal{Z}_j^{1+\kappa}.
\end{align*}
Now we are ready to estimate $\mathcal{Y}_j$ and $\mathcal{Z}_j$. \vskip 1mm
{\bf Estimation of $\mathcal{Y}_j$: }Next, we estimate $\mathcal{Y}_{j+1}$. In the course of these calculations, we successively use H\"older's inequality, Theorem \ref{T2.10}, Young's inequality and definitions in (\ref{eq4.6}),
\begin{align}\label{eq4.9}
    \mathcal{Y}_{j+1}&=d^{-2}\fiint_{\mathcal{Q}^\ominus_{\rho_{j+1}}}\left(u-g-k_{j+1}\right)_+^2\,d\xi dt\nonumber\\
    &\leq d^{-2}\Pi_j^\frac{2}{2+Q}\left(\fiint_{\mathcal{Q}^\ominus_{\rho_{j+1}}}\left(u-g-k_{j+1}\right)_+^{2\left(1+\frac{2}{Q}\right)}\right)^\frac{Q}{Q+2}\nonumber\\
    &\leq Cd^{-2}\Pi_j^\frac{2}{2+Q}\left|\mathcal{Q}^\ominus_{\rho_j}\right|^{-\frac{Q}{Q+2}}\left[\left(u-g-k_{j+1}\right)_+\right]^2_{V^{2,2}(\mathcal{Q}^\ominus_{\rho_{j+1}})}\\
    &\quad\quad+Cd^{-2}\Pi_j^\frac{2}{2+Q}\left|\mathcal{Q}^\ominus_{\rho_j}\right|^{-\frac{Q}{Q+2}}\rho_{j+1}^{-2}\|(u-g-k_{j+1})_+\|^2_{L^2(\mathcal{Q}^\ominus_{\rho_{j+1}})}\nonumber\\
    &\leq Cd^{-2}\Pi_j^\frac{2}{2+Q}\left|\mathcal{Q}^\ominus_{\rho_j}\right|^{-\frac{Q}{Q+2}}\left[\left(u-g-k_{j+1}\right)_+\right]^2_{V^{2,2}(\mathcal{Q}^\ominus_{\rho_{j+1}})}\nonumber\\
    &\quad\quad+Cd^{-2}\Pi_j^\frac{2}{2+Q}\left|\mathcal{Q}^\ominus_{\rho_j}\right|^{-\frac{Q}{Q+2}}\frac{1}{(\rho_j-\rho_{j+1})^2}\|(u-g-k_{j})_+\|^2_{L^2(\mathcal{Q}^\ominus_{\rho_{j}})}.\nonumber
\end{align}
At first, estimate second integral on the right hand side of (\ref{eq4.9}) like the estimation of $II$. Then combine with (\ref{eq4.8}) and the estimates $I,II$ and $III$ to get the following
\begin{align}\label{eq4.10}
    \mathcal{Y}_{j+1}\leq C\left[\frac{2^{4j}}{(1-\sigma)^2}\mathcal{Y}_j^{1+\frac{2}{2+Q}}+\frac{2^{2j(3+Q)}}{(1-\sigma)^{2(Q+2)}}\mathcal{Y}_j^\frac{2}{2+Q}\mathcal{Z}_j^{1+\kappa}\right].
\end{align}\vskip 1mm
{\bf Estimation of $\mathcal{Z}_j$: } Using Theorem \ref{T2.10}, we get
\begin{align}\label{eq4.11}
    \mathcal{Z}_{j+1}&=\left(\fint_{-\rho_{j+1}^2}^0\left(\frac{|\mathcal{A}^g_+(t;0,\rho_{j+1,k_{j+2}}|}{|B_{\rho_{j+1}}|}\right)^2dt\right)^\frac{1}{2(1+\kappa)}\nonumber\\
    &\leq \left(\rho_{j+1}\right)^{-\frac{Q+1}{1+\kappa}}\left(\int_{-\rho_{j+1}^2}^0\left(\int_{B_{\rho_{j+1}}}\left(\frac{u-g-k_{j+1}}{k_{j+2}-k_{j+1}}\right)^{\hat{q}}d\xi\right)^2dt\right)^\frac{2}{4(1+\kappa)}\\
    &=\left(\rho_{j+1}\right)^{-Q}\left(k_{j+2}-k_{j+1}\right)^{-2}\|(u-g-k_{j+1})_+\|_{L^{\hat{q},2\hat{q}}(\mathcal{Q}^\ominus_{\rho_{j+1}})}\nonumber\\
    &\leq \frac{2^{2j}}{d^2}\left|\mathcal{Q}^\ominus_{\rho_j}\right|^{-\frac{Q}{Q+2}} \left[\left[(u-g-k_{j+1})_+\right]_{V^{2,2}(\mathcal{Q}^\ominus_{\rho_{j+1}})}+\rho^{-2}_{j+1}\|u-g-k_{j+1}\|^2_{L^2(\mathcal{Q}^2_{\rho_{j+1}})}\right].\nonumber
\end{align}
Now proceeding as (\ref{eq4.10}), we obtain
\begin{align}\label{eq4.12}
    \mathcal{Z}_{j+1}\leq C\frac{2^{4j}}{(1-\sigma)^2}\mathcal{Y}_j+\frac{2^{2j(3+Q)}}{(1-\sigma)^{2(Q+2)}}\mathcal{Z}_j^{1+\kappa}.
\end{align}
Further simplifying (\ref{eq4.10}) and (\ref{eq4.12}) as 
\begin{align*}
    \begin{cases}
        &\mathcal{Y}_{j+1}\leq \frac{C}{(1-\sigma)^L}2^{Lj}\left(\mathcal{Y}_j^{1+\frac{2}{2+Q}}+\mathcal{Z}_j^{1+\kappa}\mathcal{Y}_j^\frac{2}{2+Q}\right),\\
        &\mathcal{Z}_{j+1}\leq \frac{C}{(1-\sigma)^L}2^{Lj}\left(\mathcal{Y}_j+\mathcal{Z}_j^{1+\kappa}\right),
    \end{cases}
\end{align*}
for some constant $C=C(\texttt{data})$ and $L=10+2Q$.\vskip 1mm
To use the iteration  lemma \ref{L2.4}, we need to estimate $\mathcal{Y}_0$ and $\mathcal{Z}_0$.\vskip 1mm
{\bf Estimation of $\mathcal{Y}_0$ : } From the definition of (\ref{eq4.6}), we have 
\begin{align*}
    \mathcal{Y}_0&=\frac{C}{d^2}\fiint_{\mathcal{Q}^\ominus_{\rho_0}}(u-g-k_0)^2_+\,d\xi dt\\
    &\leq \frac{C}{d^2}\fiint_{\mathcal{Q}^\ominus_R}(u-g)^2_+\,d\xi dt.
\end{align*}\vskip 1mm
{\bf Estimation of $\mathcal{Z}_0$ : } Following (\ref{eq4.11}), we have 
\begin{align}\label{eq4.13}
    \mathcal{Z}_0&\leq CR^{-Q}\left(\int_{-\rho_0^2}^0\left(\int_{B_{\rho_0}}\left(\frac{u-g-\hat{k}}{d-\hat{k}}\right)^{\hat{q}}d\xi\right)^2dt\right)^\frac{2}{4(1+\kappa)}\nonumber\\
    &\leq CR^{-Q}\left(\frac{1}{d-\hat{k}}\right)^2\left[\left[(u-g-\hat{k})_+\right]^2_{V^{2,2}(\mathcal{Q_{\rho_0}})}+\|(u-g-\hat{k})_+\|^2_{L^2(\mathcal{Q_{\rho_0}})}\right].
\end{align}
In the similar manner, estimate $I$, we have 
\begin{align}\label{eq4.14}
    \frac{1}{(R-\rho_0)^2}\iint_{\mathcal{Q}_R^\ominus}(u-g-\hat{k})^2_+\,d\xi dt\leq \frac{C}{(1-\sigma)^2}R^Q\fiint_{\mathcal{Q}_R^\ominus}(u-g)^2\,d\xi dt.
\end{align}
Similarly for $II$ and $III$, we get respectively
\begin{align}\label{eq4.15}
    \frac{2}{(R-\rho_0)^2}\iint_{\mathcal{Q}^\ominus_R}(u-g-k_0)^2\,d\xi dt\leq \frac{C}{(1-\sigma)^2}R^Q\fiint_{\mathcal{Q}^\ominus_R}(u-g)^2\,d\xi dt,
\end{align}
\begin{align}\label{eq4.16}
    C\hat{k}^2\left(\frac{1}{1-\sigma}\right)^{2(Q+2)}\frac{1}{\rho_0}\left(\int_{-\rho_0}^0\left|\mathcal{A}^g_+(t;0,\rho_0,k_1)\right|^2dt\right)^\frac{1}{2}\leq C\hat{k}^2\left(\frac{1}{1-\sigma}\right)^{2(Q+2s)}R^Q.  
\end{align}
Next, in (\ref{eq4.12}) we choose $k_{j+1}=\hat{k},\rho_{j+1}=\rho_0$ and $\rho_j=R$ which satisfy $\frac{R+\rho_0}{2}=\frac{3+\sigma}{2}R$. Then we have after combining (\ref{eq4.13})-(\ref{eq4.16}). 
\begin{align}\label{eq4.17}
    \mathcal{Z}_0\leq \frac{C}{(1-\sigma)^L}\frac{1}{(d-\hat{k})^2}\left(\fiint_{\mathcal{Q}_R^\ominus}(u-g)_+^2\;d\xi dt+\hat{k}^2\right).
\end{align}
Now, considering $d$ sufficiently large such that
\begin{align*}
    d=\left(\frac{4C}{(1-\sigma)^{2L}}\right)^\frac{1+\kappa}{2\Upsilon}2^{L\frac{1+\kappa}{2\Upsilon}+1}\left[\hat{k}+\left(\fiint_{\mathcal{Q}_R^\ominus}(u-g)^2_+\,d\xi dt\right)^\frac{1}{2}\right],
\end{align*}
where
\begin{align*}
    \Upsilon:=\min\left\{\kappa,\frac{2}{Q+2}\right\}.
\end{align*}
Hence, we get
\begin{align*}
    \mathcal{Y}_0+\mathcal{Z}_0^{1+\kappa}\leq \left(\frac{2C}{(1-\sigma)^L}\right)^\frac{1+\kappa}{\Upsilon}2^{-L\frac{1+\kappa}{\Upsilon^2}}.
\end{align*}
Using Lemma \ref{L2.4} and scaling back, we get, for $C_1=2L\left(\frac{1+\kappa}{2\Upsilon}+1\right)$ and translating back to $z_0=(\xi_0,t_0)$, we have
\begin{align}\label{eq4.18}
    \sup_{z_0+\mathcal{Q}_{\sigma R}^\ominus}u(\xi,t)&\leq \frac{C}{(1-\sigma)^{C_1}}\left[\left(\fiint_{z_0+\mathcal{Q}_{R}^\ominus}u_+^2\,d\xi  dt\right)^\frac{1}{2}+\left(\fint_{t_0-R^2}^{t_0}\left(\fint_{B_R(\xi_0)}u_+(\eta,t)\,d\eta\right)^2\,dt\right)^\frac{1}{2}+1\right]\nonumber\\
    &\quad\quad\quad\quad +\mathfrak{C}R^{-2}\int_{t_0-R^2}^{t_0}\mbox{Tail}(u_+(t);R,\xi_0)dt.
\end{align}

\section{Preliminary Results}\label{S5}
In this section, we are going to discuss some preliminary results. Proofs of these results have used the ideas in \cite{SZ22_8} and \cite{Kar26_9}. Unlike \cite{SZ22_8}, we prove these lemmas with a time dependent level. We derive all these lemmas using the Caccioppoli inequality in (\ref{eq3.6}). 
Before proceeding to the lemmas, we want to mention some notations which we will use throughout the lemmas extensively.
\begin{align*}
    \begin{cases}
        \mathcal{Q}:=B_R(\xi_0)\times(T_1,T_2]\subset\Omega_T,\\
        \mu^+\geq \esssup_{\mathcal{Q}}u,\,\,\mu^-\leq \essinf_{\mathcal{Q}}u,\\
        \omega\geq \mu^+-\mu^-.
    \end{cases}
\end{align*}
Moreover, throughout this paper we use $|B_{mr}|=\mathcal{C}|B_r|$, where $\mathcal{C}$ depends on $Q,\,m(>0)$ as well as the homogeneous norm on the Heisenberg Group.\vskip 1mm
Next, we discuss a class of lemmas starting with De Giorgi Lemma. Here we derive a pointwise estimate by considering measure theoretical information over the parabolic cylinder. 
\begin{lemma}(De Giorgi Lemma)\label{L5.1}
    Let for $0<\rho<R/2$ and consider $\lambda,\delta\in (0,1)$ , assume $z_0+\mathcal{Q}^\ominus_\rho(\delta)\subset B_R(\xi_0)\times (T_1,T_2]=\mathcal{Q}$. Then for locally bounded, local weak sub(super)solution $u$ to (\ref{eq1.1}), there exist $\tilde{\gamma}=\tilde{\gamma}(\textup{\texttt{data}})(>1)$ and $\iota\in(0,1)$ depending on \textup{\texttt{data}} and $\delta$, such that if
     \begin{align*}
         |\{\pm(\mu^\pm-u)\leq \lambda\omega\}\cap (z_0+\mathcal{Q}_\rho^\ominus(\delta))|\leq \iota\,|z_0+\mathcal{Q}_\rho^\ominus(\delta)|.
     \end{align*}
     then either
     \begin{align*}
         \tilde{\gamma}R^{-2}\int_{T_1}^{T_2}\textup{Tail }[(u-\mu^\pm)_\pm(t);R,\xi_0]\,dt>\lambda\omega,
     \end{align*}
     or
     \begin{align*}
         \pm(\mu^\pm-u)\geq \frac{1}{4}\lambda\omega\,\,\,\mbox{ a.e. in }z_0+\mathcal{Q}^\ominus_{\rho/2}(\delta).
     \end{align*}
\end{lemma}
\begin{proof}
    For the sake of simplicity, we assume $(\xi_0,t_0)=(0,0)$. We proceed this with the supersolution. The sub counterpart follows the similar steps. For $n=0,1,2,...$, we define $\rho_n,\tilde{\rho}_n,B_n,\tilde{B}_n,\mathcal{Q}_n^\ominus(\delta),\tilde{\mathcal{Q}}_n^\ominus(\delta)$ as follows
    \begin{align*}
        \begin{cases}
            & \rho_n=\frac{\rho}{2}+\frac{\rho}{2^{n+1}},\;\;\tilde{\rho}_n=\frac{\rho_n+\rho_{n+1}}{2},\\
            & B_n=B_{\rho_n}(0),\,\,\tilde{B}_n=B_{\tilde{\rho}_n}(0),\\
            &\mathcal{Q}_n^\ominus(\delta)=B_n\times(-\delta\rho_n^2,0],\,\,\tilde{\mathcal{Q}}_n^\ominus(\delta)=\tilde{B}_n\times(-\delta\tilde{\rho}_n^2,0].
        \end{cases}
    \end{align*}
    Here we consider the Caccioppoli estimate in Remark \ref{R3.2} with domain $B_n$ and parabolic cylinder $\mathcal{Q}^\ominus_n(\delta)$ to $w_+(\xi,t)=(u-g-k_n)_+$ with the time dependent level $g(t)+k_n$ where for $n=0,1,2,...$
    \begin{align*}
        k_n=\mu^+-\lambda_n\omega,\,\,\,\mbox{where } \lambda_n=\frac{\lambda}{2}+\frac{\lambda}{2^{n+1}},
    \end{align*}
    and $$g(t)=\mathfrak{C}\int_{-R^2}^t\int_{\hn\setminus B_R}\frac{|u_+(\eta,\tau)|}{|\eta^{-1}|_{\hn}^{Q+2s}}\,d\eta d\tau,$$ where positive constant $\mathfrak{C}$ to be chosen later.  Moreover two cut-off functions $\psi$ and $\nu$ in $B_n$ and $(-\delta\rho_n^2,0]$ respectively
    \begin{align*}
        \psi=\begin{cases}
            & 1\,\,\,\textup{in } B_{n+1}\\
            & 0\,\,\,\textup{in } \tilde{B}^c_n
        \end{cases}
        \,\,\,\,\,\,\,\,\,\,\textup{such that } |\nabla_{\hn}\psi|\leq \gamma\frac{2^{(n+2)}}{\rho}.
    \end{align*}
    and
    \begin{align*}
        \nu =\begin{cases}
            & 1\,\,\,\textup{in } (-\delta\rho_{n+1}^2,0]\\
            & 0\,\,\,\textup{in } (-\delta\tilde{\rho}_n^2,0]^c
        \end{cases}
        \,\,\,\,\,\,\,\,\,\,\textup{such that } |\partial_t\nu|\leq \gamma\frac{2^{2(n+2)}}{\delta\rho^2}.
    \end{align*}
    Considering all these entities, (\ref{eq3.6}) implies
    \begin{align}\label{eq5.1}
        &\esssup_{-\delta\rho^2_{n+1}<t<0}\int_{B_{n+1}}w^2_+(\xi,t)\,d\xi+\iint_{\mathcal{Q}_{n+1}^\ominus(\delta)}|\nabla_{\hn}w_+|^2\,d\xi dt\nonumber\\
        &\quad +\int_{-\delta\rho^2_{n+1}}\iint_{B_{n+1}\times B_{n+1}}\frac{|w_+(\xi,t)-w_+(\eta,t)|}{\dist[\xi]{\eta}^{Q+2s}}\,d\xi d\eta dt\nonumber\\
        &\quad\quad\leq\gamma \int_{-\delta\rho_n^2}^0\iint_{B_n\times B_n}\max\left\{w^2_+(\xi,t),w^2_+(\eta,t)\right\}\frac{|\psi(\xi)-\psi(\eta)|^2}{\dist[\xi]{\eta}^{Q+2s}}\nu^2(t)\,d\xi d\eta dt\nonumber\\
        &\quad\quad\quad +\gamma\iint_{\mathcal{Q}^\ominus_n(\delta)}\left(w_+\nabla_{\hn}\psi\,\nu\right)^2(\xi,t)\,d\xi dt+ \gamma\iint_{\mathcal{Q}^\ominus_n(\delta)}\left(w_+^2\psi^2\nu\,\partial_t\nu\right)(\xi,t)\,d\xi dt\\
        &\quad\quad\quad\quad+\iint_{\mathcal{Q}^\ominus_n(\delta)}(w_+\psi^2\nu^2)(\xi,t)\left\{\int_{\hn\setminus B_n}\frac{w_+(\eta,t)}{\dist[\xi]{\eta}^{Q+2s}}\,d\eta\right\}d\xi dt\nonumber\\
        &\quad\quad\quad\quad\quad-\mathfrak{C}\iint_{\mathcal{Q}^\ominus_n(\delta)}g'(t)(w_+\psi^2\nu^2)(\xi,t)\,d\xi dt\nonumber\\
        &\quad\quad\quad\quad\quad\quad =: I_1+I_2+I_3+I_4+I_5.\nonumber
    \end{align}
    For the terms $I_2$ and $I_3$, we have the following by using the definition $\mathcal{A}_n:=\{(\xi,t)\in\mathcal{Q}_n^\ominus(\delta):u(\xi,t)-g(t)>k_n\}$
    \begin{align*}
        I_2\leq \gamma\frac{2^{2n}}{\rho^2}(\lambda\omega)^2|\mathcal{A}_n|,
    \end{align*}
    and
    \begin{align*}
        I_3\leq \gamma\frac{2^{2n}}{\delta\rho^2}(\lambda\omega)^2|\mathcal{A}_n|.
    \end{align*}
    In the case of $I_1$, using assumption of our lemma
    \begin{align*}
        I_1&\leq \gamma \frac{2^{2n}}{\rho^2}\,2(\lambda\omega)^2 \int_{\delta\rho_n^2}^0\iint_{B_n\times B_n}\frac{\chi_{\{u(\xi,t)-g(t)>k_n\}}}{\dist[\xi]{\eta}^{Q+2(s-1)}}\,d\xi d\eta dt\\
        &=\gamma \frac{2^{2n}}{\rho^2}\,2(\lambda\omega)^2\iint_{\mathcal{Q}_n^\ominus(\delta)}\chi_{\{u(\xi,t)-g(t)>k_n\}}\,d\xi\left[\int_{B_n}\frac{1}{\dist[\xi]{\eta}^{Q+2(s-1)}}\right]\\
        &\leq \gamma\frac{2^{2n}}{\rho^2}(\lambda\omega)^2|\mathcal{A}_n|.
    \end{align*}
    In the last line we use the fact that $\frac{1}{\rho^{2s}}\leq \frac{1}{\rho^{2}}$ due to $\rho<1$. Next, we want to estimate $I_4$. Prior to this estimate, we note the following observations
    \begin{itemize}
        \item If $|\xi|<\rho_{n+1}$ and $|\eta|>\rho_n$, then
        \begin{align*}
            \frac{\dist[\xi]{\eta}}{|\eta|_{\hn}}\geq 1-\frac{\rho_{n+1}}{\rho_n}\geq \frac{1}{2^{n+2}}.
        \end{align*}
        \item If $|\xi|\leq \rho$ and $|\eta|>R$, then 
        \begin{align*}
            \frac{\dist[\xi]{\eta}}{|\eta|_{\hn}}\geq \frac{1}{2},\,\,\,\,\mbox{provided }\rho\leq\frac{R}{2} .
        \end{align*}
    \end{itemize}
Considering all the entities with the fact that $u\geq 0$ a.e. in $\mathcal{Q}$, $I_4$ implies
\begin{align}\label{eq5.2}
    I_4&=\iint_{\mathcal{Q}^\ominus_n(\delta)}\left(w_+\psi^2\nu^2\right)(\xi,t)\left\{\int_{\hn\setminus B_R}\frac{w_+(\eta,t)}{\dist[\xi]{\eta}^{Q+2s}}\,d\eta+\int_{\hn\setminus B_n}\frac{w_+(\eta,t)}{\dist[\xi]{\eta}^{Q+2s}}\,d\eta\right\}d\xi dt\nonumber\\
    &\leq \gamma 2^{(Q+2s)n}\iint_{\mathcal{Q}^\ominus_n(\delta)}\left(w_+\psi^2\nu^2\right)(\xi,t)\,d\xi\left\{\int_{B_R\setminus B_n}\frac{w_+(\eta,t)}{|\eta|_{\hn}^{Q+2s}}\,d\eta\right\}dt\nonumber\\
    &\quad\quad+\gamma\iint_{\mathcal{Q}^\ominus_n(\delta)}\left(w_+\psi^2\nu^2\right)(\xi,t)\,d\xi\left\{\int_{\hn\setminus B_R}\frac{w_+(\eta,t)}{|\eta|_{\hn}^{Q+2s}}\,d\eta\right\}dt\nonumber\\
    &\leq \gamma 2^{(Q+2s)n}\frac{\lambda\omega}{\rho^{2s}}\iint_{\mathcal{Q}_n^\ominus(\delta)}w_+(\xi,t)\,d\xi+\gamma \iint_{\mathcal{Q}_n^\ominus(\delta)}g'(t)\left(w_+\psi^2\nu^2\right)\,d\xi dt\nonumber\\
    &\leq \gamma 2^{(Q+2s)n}\frac{(\lambda\omega)^2}{\rho^{2}}|\mathcal{A}_n|+\gamma \int_{\mathcal{Q}_n^\ominus(\delta)}g'(t)\left(w_+\psi^2\nu^2\right)\, d\xi dt.
\end{align}
In the second to last step, we employ Lemma \ref{L2.3} and in the last step, we use definition of $\mathcal{A}_n$. Next, changing the arbitrary constant from $\gamma$ to $\mathfrak{C}$ in the last line of (\ref{eq5.2}), we can cancel out $I_5$. Hence, combining all the estimate into (\ref{eq5.1}), we have 
\begin{align}\label{eq5.3}
    \esssup_{-\delta\rho^2_{n+1}<t<0}\int_{B_{n+1}}&w^2_+\,d\xi+\iint_{\mathcal{Q}_{n+1}^\ominus(\delta)}|\nabla_{\hn}w_+|^2\,d\xi dt\nonumber\\
    &\leq \gamma \frac{2^{n(Q+2s+2)}}{\delta\rho^2}(\lambda\omega)^2|\mathcal{A}_n|.
\end{align}
Consequently, using H\"older inequality, (\ref{eq5.3}) and parabolic Sobolev inequality
\begin{align*}
    \frac{\lambda\omega}{2^{n+2}}|\mathcal{A}_{n+1}|&\leq\iint_{\mathcal{Q}_n^\ominus(\delta)}w_+(\xi,t)\,d\xi dt\\
    & \leq \left(\iint_{\mathcal{Q}^\ominus_n(\delta)}w_+^{2\kappa}\,d\xi dt\right)^\frac{1}{2\kappa}|\mathcal{A}_n|^{1-\frac{1}{2\kappa}}\\
    &\leq \gamma \left(\rho^2\iint_{\mathcal{Q}^\ominus_n(\delta)}|\nabla_{\hn}u(\xi,t)|^2\,d\xi dt+\iint_{\mathcal{Q}^\ominus_n(\delta)}|u(\xi,t)|^2\,d\xi dt\right)^\frac{1}{2\kappa}\\
    &\quad \quad\times\left(\esssup_{-\delta\rho_n^2<t<0}\fint_{B_n}w_+^2\,d\xi\right)^\frac{1}{Q\kappa}|\mathcal{A}_n|^{1-\frac{1}{2\kappa}}\\
    &\leq \gamma b^n\delta^{-\left[\frac{1}{\kappa}+\frac{2}{Q\kappa}\right]}\rho^{-(Q+2)\frac{1}{Q\kappa}}(\lambda\omega)^{2\left[\frac{1}{Q\kappa}+\frac{1}{2\kappa}\right]}|\mathcal{A}_n|^{1+\frac{1}{\delta\kappa}},
\end{align*}
where $b=b(Q)$. In the course of these calculations, we have denoted $\frac{Q}{Q-2}$ as $\kappa$. Now, the previous estimate leads to a recursive inequality by $Y_n=\frac{|\mathcal{A}_n|}{|\mathcal{Q}^\ominus_n(\delta)|}$, viz, there is some constant (depending only on \texttt{data}) such that 
\begin{align}
    Y_{n+1}\leq \gamma(2b)^n\delta^{-\left[\frac{1}{\kappa}+\frac{2}{Q\kappa}\right]}Y_n^{1+\frac{1}{\delta\kappa}}.
\end{align}
Hence by using the convergence Lemma \ref{L2.4}, taking $g(T_2)\leq \frac{1}{4}\lambda\omega$ and redefining $4\mathfrak{C}$ as $\mathfrak{C}$ (as done in \cite[Lemma 3.1]{Kar26_9}), we have the result.
\end{proof}
Next, we prove the De Giorgi lemma with respect to the forward-in-time version. 
\begin{lemma}\label{L5.2}
    Let $\lambda\in (0,1)$ and $u$ be a locally bounded weak sub(super)solution of the equation (\ref{eq1.1}) in $\Omega_T$. There exist $\tilde{\gamma}(>1)$ and $\iota_0\in(0,1)$, depending only on the \textup{\texttt{data}} and independent of $\lambda$, such that $B_{\rho/2}(\xi_0)\times(t_0,t_0+\iota_0\rho^{2}]\subset\mathcal{Q}$ and if 
    $$\pm(\mu^\pm-u(\cdot,t_0))\geq \lambda\omega\;\;\;\;\;\mbox{ a.e. in } B_\rho(\xi_0),$$
    then either $$\tilde{\gamma}R^{-2}\int_{T_1}^{T_2}\textup{Tail }[(u-\mu^\pm)_\pm(t);R,\xi_0]\,dt>\lambda\omega,$$
    or
    $$\pm(\mu^\pm-u)\geq \frac{1}{4}\lambda\omega\;\;\;\;\;\mbox{ a.e. in }B_{\rho/2}(\xi_0)\times(t_0,t_0+\iota_0\rho^2].$$
\end{lemma}
\begin{proof}
    Without loss of any generality, we assume $(\xi_0,t_0)=(0,0)$. Unlike the previous lemma, this relies on point wise initial data. We want to prove this lemma using the energy estimate (\ref{eq3.6}) but on the forward-in-time domain. Moreover, as far as cut-off function is concerned, we only consider time-independent cut-off function i.e. $\partial_t\nu\equiv 0$.  \vskip 1mm
    Consider $w_+$ and $g(t)$ ass in the previous lemma. Let us introduce $\rho_n,\tilde{\rho}_n,B_n,\tilde{B}_n,\mathcal{Q}^\oplus_n(\delta),\tilde{\mathcal{Q}}^\oplus_n(\delta),k_n,\lambda_n$ as follows
    \begin{align*}
        \begin{cases}
            \rho_n=\frac{\rho}{2}+\frac{\rho}{2^{n+1}},\,\,\tilde{\rho}_n=\frac{\rho_n+\rho_{n+1}}{2},\\
            B_n=B_{\rho_n}(0),\,\,\tilde{B}_n=B_{\tilde{\rho}_n}(0),\\
            \mathcal{Q}^\oplus_n(\delta)=B_n\times (0,\delta\rho_n^2],\,\,\tilde{\mathcal{Q}}^\oplus_n(\delta)=\tilde{B}_n\times (0,\delta\tilde{\rho}_n^2],\\
            \lambda_n=\frac{\lambda}{2}+\frac{\lambda}{2^{n+1}},\,\,k_n=\mu^+-\lambda_n\omega.
        \end{cases}
    \end{align*}
    Moreover, cut-off function $\psi(x)$ in $B_n$ is defined as
    \begin{align*}
        \psi\equiv \begin{cases}
            1\,\,\,\,\mbox{on } B_{n+1}\\
            0\,\,\,\,\mbox{in } \tilde{B}^c_n
        \end{cases}
        \,\,\,\,\,\,\mbox{such that }|\nabla_{\hn}\psi|\leq \frac{\gamma 2^{n+4}}{\rho}.
    \end{align*}
    As a result, we have 
    \begin{align*}
        \esssup_{0<t<\delta\rho_{n+1}^2}&\int_{B_{n+1}}w^2_+\,d\xi+\iint_{\mathcal{Q}_n^\oplus(\delta)}|\nabla_{\hn}w_+|^2\,d\xi dt\\
        &\leq \gamma\int_0^{\delta\rho_n^2}\iint_{B_n\times B_n}\max\left\{w^2_+(\xi,t),w_+^2(\eta,t)\right\}\frac{|\psi(\xi)-\psi(\eta)|^2}{\dist[\xi]{\eta}^{Q+2s}}\,d\xi d\eta dt\\
        &\quad\quad+\gamma\iint_{\mathcal{Q}^\oplus_n(\delta)}\left(w_+\nabla_{\hn}\psi\,\nu\right)^2\,d\xi dt+\iint_{\mathcal{Q}^\oplus_n(\delta)}\left(w_+\psi^2\nu^2\right)\,d\xi\left\{\int_{\hn\setminus B_n}\frac{w_+(\eta,t)}{\dist[\xi]{\eta}^{Q+2s}}\,d\eta\right\}dt\\
        &\quad\quad\quad\quad -\mathfrak{C}\iint_{\mathcal{Q}^\oplus_n(\delta)}g'(t)(w_+\psi^2\nu^2)(\xi,t)\,d\xi dt.
    \end{align*}
    Proceeding as in the previous lemma with suitable choice of $\mathfrak{C}$ and $\gamma$, we conclude that the left hand side of the previous inequality is dominated by
    \begin{align*}
        \gamma\frac{2^{(Q+2s+2)n}}{\delta\rho^2}(\lambda\omega)^2|\mathcal{A}_n|.
    \end{align*}
    where we use the definition $\mathcal{A}_n:=\{(\xi,t)\in\mathcal{Q}_n^\oplus(\delta):u(\xi,t)-g(t)>k_n\}$. After running the De Giorgi iteration (as done in Lemma \ref{L5.1}) on the last estimation, there exists a constant $\iota_0$ (depending only on data), such that with the choice of $\delta=\iota_0$, we have
\begin{align*}
 u(\xi,t)-g(t)\leq \mu^+-\frac{1}{2}\xi\omega\,\,\,\textup{ a.e. in } B_{\frac{\rho}{2}}\times (0,\iota_0\rho^2).
\end{align*}
Consequently, if we impose  $g(T_2)\leq \frac{\xi\omega}{4}$ after redefining $\mathfrak{C}=4\mathfrak{C}$, then
\begin{align*}
    u\leq -\frac{\xi\omega}{4}+\mu^+\,\,\,\textup{ a.e. in } B_{\frac{\rho}{2}}\times (0,\iota_0\rho^2).
\end{align*}
\end{proof}
The following lemma propagates the measure theoretical information from a time slice to a time interval. 
\begin{lemma}\label{L5.3}
    Let $\lambda\in(0,1)$ and $\alpha\in(0,1]$. For locally bounded, local weak sub(super) solution $u$ of (\ref{eq1.1}) in $\Omega_T$ there exists $\delta,\varepsilon\in (0,1)$ depending only on the \textup{\texttt{data}} and $\alpha$ such that if
    $$|\{\pm(\mu^\pm-u(\cdot,t_0))\geq \lambda\omega\}\cap B_\rho(\xi_0)|\geq \alpha|B_\rho(\xi_0)|,$$
    then either
    $$\frac{1}{\delta}R^{-2}\int_{T_1}^{T_2}\textup{Tail }[(u-\mu^\pm)_\pm(t);R,\xi_0]\,dt>\lambda\omega,$$
    or
    $$|\{\pm(\mu_\pm-u(\cdot,t)\geq \varepsilon\lambda\omega\}\cap B_\rho(\xi_0)|\geq\frac{\alpha}{2}|B_\rho(\xi_0)|,\;\;\;\;\;\textup{ for all } t\in (t_0,t_0+\delta\rho^{2}] $$
    provided $B_\rho(\xi_0)\times(t_0,t_0+\delta\rho^{2}]\subset\mathcal{Q}$. Moreover, we have $\delta\approx\alpha^{Q+3}$ and $\varepsilon\approx\alpha$.
\end{lemma}
\begin{proof}
    For the sake of simplicity, we assume $(\xi_0,t_0)=(0,0)$. We prove this lemma for the subsolution case. For convenience, let us denote $aM:=-a\lambda\omega+\mu^+$ for any real number $a$, and for any $t>0$
    \begin{align*}
        \mathcal{A}_{aM,\rho}(t):=\{u(\cdot,t)>aM\}\cap B_\rho,
    \end{align*}
    then it is evident that
    \begin{align}\label{eq5.5}
        |\mathcal{A}_{aM,\rho}(0)|\leq (1-\alpha)|B_\rho|.
    \end{align}
    Moreover, we set a time-independent cut-off function $\psi$ on $B_\rho$ as follows
    \begin{align*}
    \psi\equiv
        \begin{cases}
           1\,\,\,\,\mbox{ on } B_{(1-\sigma)\rho},\\
           0\,\,\,\,\mbox{ on } B^c_{(1-\sigma/2)\rho},
        \end{cases}
        \,\,\,\,\mbox{ such that }|\nabla_{\hn}\psi|\leq \frac{\gamma}{\sigma\rho}.
    \end{align*}
    Next, consider the truncated function $w_+=(u-M)_+$ and examine the energy estimate (\ref{eq3.6}) over the domain $B_\rho\times (0,\delta\rho^2]$. Then, for all $t\in (0,\delta\rho^2]$, we have
    \begin{align}\label{eq5.6}
        \int_{B(1-\sigma)\rho}w^2_+(\xi,t)\,d\xi&\leq \int_{B_\rho}w^2_+(\xi,0)\,d\xi+\gamma\int_0^{\delta\rho^2}\int_{B_\rho}\left(w^2_+\,\nabla_{\hn}\psi\right)^2(\xi,t)\,d\xi\nonumber\\
        &\quad+\gamma\int_0^{\delta\rho^2}\max\{w^2_+(\xi,t),w^2_+(\eta,t)\}\frac{|\psi(\xi)-\psi(\eta)|}{\dist[\xi]{\eta}^{Q+2s}}\,d\xi d\eta dt\\
        &\quad\quad+\gamma\int_0^{\delta\rho^2}\int_{B_\rho}w_+(\xi,t)\left\{\int_{\hn\setminus B_\rho}\frac{w_+(\eta,t)}{\dist[\xi]{\eta}^{Q+2s}}\,d\eta\right\}\,d\xi dt\nonumber\\
        &\quad\quad\quad =:I_1+I_2+I_3+I_4,\nonumber
    \end{align}
    where $\gamma=\gamma(\texttt{data})$. Now, we estimate the terms on the right hand side of previous inequality one by one. Consider $I_3$ and $I_2$ successively.
    \begin{align*}
        I_3&\leq \gamma\delta\rho^2\left(\frac{\lambda\omega}{\sigma\rho}\right)^2\iint_{B_\rho\times B_\rho}\frac{1}{\dist[\xi]{\eta}^{Q+2s-2}}\,d\xi dt\\
        &=\gamma\delta \left(\frac{\lambda\omega}{\sigma}\right)^2|B_\rho|,
    \end{align*}
    and
    \begin{align*}
        I_2&\leq \gamma\left(\frac{\lambda\omega}{\sigma\rho}\right)^2\delta\rho^2|B_\rho|\\
        &=\gamma\delta\left(\frac{\lambda\omega}{\sigma}\right)^2|B_\rho|.
    \end{align*}
    Before proceeding towards $I_4$, we observe that for $\xi\in \mbox{ supp }\psi$ i.e. $|\xi|\leq (1-\frac{\sigma}{2})\rho$ and $|\eta|\geq \rho$, we have 
    \begin{align*}
        \frac{\dist[\xi]{\eta}}{|\eta|_{\hn}}\geq 1-\frac{|\xi|_{\hn}}{|\eta|_{\hn}}\geq \frac{\sigma}{2}.
    \end{align*}
     Also, when $|\eta|_{\hn}\geq R$ and $|\xi|_{\hn}\leq \rho$, we have
    \begin{align*}
        \frac{\dist[\xi]{\eta}}{|\eta|_{\hn}}\geq\,\,\,\textup{ provided }\,\rho\leq \frac{R}{2}.
    \end{align*}
    So, $I_4$ implies with the enforcement of 
    \begin{align*}
        \frac{1}{\delta}R^{-2}\int_{T_1}^{T_2}\textup{Tail }[(u-\mu^\pm)_\pm(t);R,\xi_0]\leq\lambda\omega,
    \end{align*}
    we have
    \begin{align*}
        I_4&\leq 2^{Q+2s}\frac{\lambda\omega}{\sigma^{Q+2s}}|B_\rho|\left\{\int_0^{\delta\rho^2}\int_{B_R\setminus B_\rho} \frac{w_+(\eta,t)}{|\eta|_{\hn}^{Q+2s}}\,d\eta dt+\int_0^{\delta\rho^2}\int_{\hn\setminus B_R} \frac{w_+(\eta,t)}{|\eta|_{\hn}^{Q+2s}}\,d\eta dt\right\}\\
        &\leq \gamma \frac{\lambda\omega}{\sigma^{Q+2}}|B_\rho|\left\{\gamma\delta\lambda\omega+\int_0^{\delta\rho^2}\int_{\hn\setminus B_R} \frac{w_+(\eta,t)}{|\eta|_{\hn}^{Q+2s}}\,d\eta dt\right\}\\
        &\leq \gamma \frac{\delta(\lambda\omega)^2}{\sigma^{Q+2}}|B_\rho|.
    \end{align*}
    Combining all the estimates into (\ref{eq5.6}), we have
\begin{align*}
    \int_{B(1-\sigma)\rho}w^2_+(\xi,t)\,d\xi&\leq \int_{B_\rho}w^2_+(\xi,0)\,d\xi+\gamma\frac{\delta(\lambda\omega)^2}{\sigma^{Q+2}}|B_\rho|\\
    &\leq \left[(1-\alpha)+\frac{\gamma\delta}{\sigma^{Q+2}}\right](\lambda\omega)^2|B_\rho|.
\end{align*}
for all $t\in(0,\delta\rho^2]$. In the last step of the above estimates, we used (\ref{eq5.5}). The left-hand side of the above inequality can be estimated for $\varepsilon\in(0,1)$, 
\begin{align*}
    \int_{B(1-\sigma)\rho}w^2_+(\xi,t)\,d\xi&\geq \int_{B_{(1-\sigma)\rho}\cap\{u(\cdot,t)>\varepsilon M\}}w^2_+(\xi,t)\,d\xi\\
    &\geq (1-\varepsilon)^2(\lambda\omega)^2|\mathcal{A}_{\varepsilon M,(1-\sigma)\rho}(t)|.
\end{align*}
Also
\begin{align*}
  |\mathcal{A}_{\varepsilon M,\rho}(t)|&=\left|\mathcal{A}_{\varepsilon M,(1-\sigma)\rho}(t)\cup\left(\mathcal{A}_{\varepsilon M,\rho}(t)-\mathcal{A}_{\varepsilon M,(1-\sigma)\rho}(t)\right)\right|\\
  &\leq |\mathcal{A}_{\varepsilon M,(1-\sigma)\rho}(t)|+|B_\rho-B_{(1-\sigma)\rho}|\\
  & \leq |\mathcal{A}_{\varepsilon M,(1-\sigma)\rho}(t)|+Q\sigma|B_\rho|.
\end{align*}
Now, proceeding as in \cite[Lemma 3.3]{Kar26_9}, with a suitable choice of parameters $\sigma=\frac{\alpha}{8Q},\,\delta=\frac{1}{8\gamma}\sigma^{Q+2}\alpha,\,(1-\varepsilon)\geq \sqrt{\frac{1-\frac{3}{4}\alpha}{1-\frac{1}{2}\alpha}}$, we have our result. 
\end{proof}
In the purely non-local case, in proving {\em expansion of positivity}, we are unable to use {\em De Giorgi Isoperimetric Inequality} due to the presence of possible jumps in the function of $W^{s,2}$ space. In the mixed local-nonlocal case, things are different. Unlike the nonlocal case, we work with the function space $W^{1,2}$. This allow us to use De Giorgi isoperimetric inequality to expand the positivity from a certain time level to a time interval. 
\begin{lemma}\label{L5.4}
    Consider some positive constants $\lambda\in (0,1)$ and $\alpha\in(0,1]$. For locally bounded, local weak sub(super)solution to (\ref{eq1.1}) in $\Omega_T$, there exists constants $\delta,\eta\in (0,1)$ and $\hat{\gamma}>1$ depending upon \textup{\texttt{data}} and $\alpha$ such that if $B_{8\rho}(\xi_0
    )\times (-\delta\rho^2,\delta\rho^2]\subset\mathcal{Q}$ and 
    \begin{align*}
        |\{\pm(\mu^\pm-u(\cdot,t_0)\geq \lambda\omega\}\cap B_\rho(\xi_0)|\geq \alpha|B_\rho(\xi_0)|.
    \end{align*}
    Then either
    \begin{align*}
       \hat{\gamma}R^{-2}\int_{T_1}^{T_2}\textup{Tail }[(u-\mu^\pm)_\pm(t);R,\xi_0]\,dt>\lambda\omega,
    \end{align*}
    or
    \begin{align*}
        \pm(\mu^\pm-u)\geq \varkappa\lambda\omega,\,\,\,\,\,\,\textup{ a.e. in }\,\, B_{2\rho}(\xi_0)\times \left(t_0+\frac{1}{2}\delta\rho^2,t_0+\delta\rho^2\right].
    \end{align*}
    Moreover, $\delta\approx \alpha^{Q+3}$ and $\varkappa\approx \alpha$.
\end{lemma}
\begin{proof}
    For convenience, we will assume $(\xi_0,t_0)=(0,0)$. Consider the larger cylinder 
    \begin{align*}
        \mathcal{Q}_{8\rho}(\delta)=\mathcal{Q}^\oplus_{8\rho}(\delta)\cup \mathcal{Q}^\ominus_{8\rho}(\delta)=B_{8\rho}(0)\times(-\delta\rho^2,\delta\rho^2]
    \end{align*}
    and the truncated function
    \begin{align*}
        w_+=(u-g-M_n)_+,\,\,\,\,\,\textup{ where } M_n:=-\frac{\varepsilon\lambda\omega}{2^n}+\mu^+,\,g(t)=\mathfrak{C}\int_{-R^2}^t\int_{\hn\setminus B_R}\frac{|u_+(\eta,\tau)|}{|\eta^{-1}|_{\hn}^{Q+2s}}\,d\eta d\tau,
    \end{align*}
    for $n=0,1,2,3,...$ and $\varepsilon$ is as in previous Lemma \ref{L5.3}. The nonnegative, piecewise smooth cut-off functions $\psi$ and $\nu$ are defined as follows
    \begin{align*}
        \psi\equiv\begin{cases}
            1\,\,\,\mbox{ on } B_{4\rho},\\
            0\,\,\,\mbox{ on } B^c_{8\rho},
        \end{cases}
        \,\,\,\,\,\,\mbox{ such that }|\nabla_{\hn}\psi|\leq \frac{\gamma}{\rho},
    \end{align*}
    and 
    \begin{align*}
        \nu\equiv\begin{cases}
            1\,\,\,\mbox{ on }\, t\geq 0,\\
            0\,\,\,\mbox{ on }\,t\leq -\delta\rho^2,
        \end{cases}
        \,\,\,\,\,\,\mbox{ such that }|\partial_t\nu|\leq \frac{\gamma}{\delta\rho^2}.
    \end{align*}
    Using all this entities in (\ref{eq3.6}), we have 
    \begin{align}\label{eq5.7}
        \iint_{\mathcal{Q}_{4\rho}^\oplus(\delta)}|\nabla_{\hn}w_+|^2\,d\xi dt&\leq \gamma\int_{-\delta\rho^2}^{\delta\rho^2}\iint_{B_{8\rho}\times B_{8\rho}}\max\left\{w^2_+(\xi,t),w^2_+(\eta,t)\right\}\frac{|\psi(\xi)-\psi(\eta)|}{\dist[\xi]{\eta}^{Q+2s}}\,d\xi d\eta dt\nonumber\\
        &\quad+ \gamma\iint_{\mathcal{Q}_{8\rho}(\delta)}\left(w_+\,\nabla_{\hn}\psi\,\nu\right)^2(\xi,t)\,d\xi dt+\iint_{\mathcal{Q}_{8\rho}(\delta)}\left(w^2_+\psi^2\nu\,\partial_t\nu\right)(\xi,t)\,d\xi dt\nonumber\\
        &\quad\quad+\iint_{\mathcal{Q}_{8\rho}(\delta)}\left(w_+\psi^2\nu^2\right)(\xi,t)\,d\xi\left\{\int_{\hn\setminus B_{8\rho}}\frac{w_+(\eta,t)}{\dist[\xi]{\eta}^{Q+2s}}\,d\eta\right\}\,dt\\
        &\quad\quad\quad-\mathfrak{C}\iint_{\mathcal{Q}^\ominus_n}g'(t)(w_+\psi^2\nu^2)(\xi,t)\,d\xi dt.\nonumber
    \end{align}
    Note that we remove the first, second and third terms on the left hand side of (\ref{eq3.6}) due to choice of cut-off function and non-negativeness of the terms  . Next, we estimate each term on the right hand side of (\ref{eq5.7}) one by one. The first term on the right hand side of (\ref{eq5.7}) is estimated as
    \begin{align*}
        \int_{-\delta\rho^2}^{\delta\rho^2}\iint_{B_{8\rho}\times B_{8\rho}}&\max\left\{w^2_+(\xi,t),w^2_+(\eta,t)\right\}\frac{|\psi(\xi)-\psi(\eta)|}{\dist[\xi]{\eta}^{Q+2s}}\,d\xi d\eta dt
        \\&\leq \frac{\gamma}{\rho^2}\left(\frac{\varepsilon\lambda\omega}{2^n}\right)^2\,(2\delta\rho^2)\iint_{B_{8\rho}\times B_{8\rho}}\frac{\chi_{\{u(\xi,t)-g(t)>M_n\}}}{\dist[\xi]{\eta}^{Q+2(s-1)}}\,d\xi d\eta\\
        &\leq \gamma \delta\left(\frac{\varepsilon\lambda\omega}{2^n}\right)^2|B_{4\rho}|.
    \end{align*}
    In these course of calculations, we have used the measure theoretical assumptions of our lemma and $|B_{8\rho}|=\mathcal{C}|B_{4\rho}|$. The second and third terms of the right hand side of (\ref{eq5.7}) can be dealt with as follows with the help of the measure theoretical information, $|B_{8\rho}|=\mathcal{C}|B_{4\rho}|$ and properties of cut-off functions 
    \begin{align*}
        \iint_{\mathcal{Q}_{8\rho}(\delta)}\left(w_+\,\nabla_{\hn}\psi\,\nu\right)^2(\xi,t)\,d\xi dt&\leq \frac{\gamma}{\rho^2}\left(\frac{\varepsilon\lambda\omega}{2^n}\right)^2|\mathcal{Q}_{8\rho}(\delta)|\\
        &\leq\gamma\delta\left(\frac{\varepsilon\lambda\omega}{2^n}\right)^2|B_{4\rho}|,
    \end{align*}
    and 
    \begin{align*}
       \iint_{\mathcal{Q}_{8\rho}(\delta)}\left(w^2_+\psi^2\,\partial_t\nu\right)(\xi,t)\,d\xi dt&\leq \frac{\gamma}{\delta\rho^2}\left(\frac{\varepsilon\lambda\omega}{2^n}\right)^2|\mathcal{Q}_{8\rho}(\delta)|\\
       &\leq \gamma \left(\frac{\varepsilon\lambda\omega}{2^n}\right)^2|B_{4\rho}|.
    \end{align*}
    The second to last term on the right hand side of (\ref{eq5.7}) is estimated by employing {\em tail estimation}. Prior to the estimation we want to mention the following observations
    \begin{itemize}
        \item If $|\xi|_{\hn}\leq 4\rho$, $|\eta|_{\hn}\geq 8\rho$, then 
        \begin{align*}
            \frac{\dist[\xi]{\eta}}{|\eta|_{\hn}}\geq 1-\frac{|\xi|_{\hn}}{|\eta|_{\hn}}\geq 1-\frac{4\rho}{8\rho}=\frac{1}{2}.
        \end{align*}
        \item If $|\xi|\leq 4\rho$, $|\eta|\geq R$ and suppose $\rho\leq \frac{R}{8}$, then
        \begin{align*}
            \frac{\dist[\xi]{\eta}}{|\eta|_{\hn}}\geq 1-\frac{|\xi|_{\hn}}{|\eta|_{\hn}}\geq 1-\frac{4\rho}{R}\geq\frac{1}{2}.
        \end{align*}
    \end{itemize}
    We have the following estimate.
    \begin{align*}
        \iint_{\mathcal{Q}_{8\rho}(\delta)}&\left(w_+\psi^2\nu^2\right)(\xi,t)\,d\xi\left\{\int_{\hn\setminus B_{8\rho}}\frac{w_+(\eta,t)}{\dist[\xi]{\eta}^{Q+2s}}\,d\eta\right\}\,dt\\
        &=\iint_{\mathcal{Q}_{8\rho}(\delta)}\left(w_+\psi^2\nu^2\right)(\xi,t)\,d\xi\left\{\int_{\hn\setminus B_R}\frac{w_+(\eta,t)}{\dist[\xi]{\eta}^{Q+2s}}\,d\eta+\int_{B_R\setminus B_{8\rho}}\frac{w_+(\eta,t)}{\dist[\xi]{\eta}^{Q+2s}}\,d\eta\right\}\,dt\\
        &\leq \gamma \iint_{\mathcal{Q}_{8\rho}(\delta)}\left(w_+\psi^2\nu^2\right)(\xi,t)\,d\xi\left\{\int_{\hn\setminus B_R}\frac{w_+(\eta,t)}{|\eta|_{\hn}^{Q+2s}}\,d\eta+\int_{B_R\setminus B_{8\rho}}\frac{w_+(\eta,t)}{|\eta|_{\hn}^{Q+2s}}\,d\eta\right\}\,dt\\
        &\leq \gamma \iint_{\mathcal{Q}_{8\rho}(\delta)}\left(w_+\psi^2\nu^2\right)(\xi,t)\,d\xi\int_{\hn\setminus B_R}\frac{u_+(\eta,t)}{|\eta|_{\hn}^{Q+2s}}\,d\eta dt\\
        &\quad \quad +\gamma \iint_{\mathcal{Q}_{8\rho}(\delta)}\left(w_+\psi^2\nu^2\right)(\xi,t)\,d\xi\int_{B_R\setminus B_{8\rho}}\frac{w_+(\eta,t)}{|\eta|_{\hn}^{Q+2s}}\,d\eta dt\\
        &\leq \gamma\delta\left(\frac{\varepsilon\lambda\omega}{2^n}\right)^2|B_{4\rho}|+\gamma \iint_{\mathcal{Q}_{8\rho}(\delta)}\left(w_+\psi^2\nu^2\right)(\xi,t)\,d\xi\int_{\hn\setminus B_R}\frac{u_+(\eta,t)}{|\eta|_{\hn}^{Q+2s}}\,d\eta dt.
    \end{align*}
    In this series of calculations, we have used Lemma \ref{L2.3} and the property $|B_{8\rho}|=\mathcal{C}|B_{4\rho}|$. Moreover changing $\gamma$ to $\mathfrak{C}$ in the last integral of the last step to cancel out the last term on the right hand side of (\ref{eq5.7}). Now, combining all the estimates into (\ref{eq5.7}) and considering $\delta<1$, we have
    \begin{align}
        \iint_{\mathcal{Q}_{4\rho}^\oplus(\delta)}|\nabla_{\hn}w_+|^2\,d\xi dt\leq \gamma\left(\frac{\varepsilon\lambda\omega}{2^n}\right)^2|B_{4\rho}|,
    \end{align}
   where $\gamma$ depends upon only \texttt{data}.  Next, we recall De Giorgi isoperimetric inequality for the levels
    \begin{align*}
        l= M_{n+1}=-\frac{\varepsilon\lambda\omega}{2^{n+1}}+\mu^+,\,\,\textup{ and } \,\,k=M_n=-\frac{\varepsilon\lambda\omega}{2^{n}}+\mu^+\,\,\textup{ for } n=0,1,2,3,...
    \end{align*}
    that is 
    \begin{align}\label{eq5.9}
        \frac{\varepsilon\lambda\omega}{2^{n+1}}|\mathcal{A}_{M_{n+1},4\rho}(t)|\leq \gamma\frac{\mathcal{C}\rho}{\alpha}\int_{B_{4\rho}\cap\{M_n<u(\cdot,t)-g(t)<M_{n+1}\}}|\nabla_{\hn}u|\,d\xi.
    \end{align}
    In this calculation, we have made use of the inequality
    \begin{align*}
        \left|\{u(\cdot,t)-g(t)\leq -\varepsilon\lambda\omega+\mu^+\}\cap B_{4\rho}\right|\geq \left|\{u(\cdot,t)\leq -\varepsilon\lambda\omega+\mu^+\}\cap B_{4\rho}\right|\geq \frac{\alpha}{\mathcal{C}}|B_{4\rho}|.
    \end{align*}
    This can be deduced from the measure theoretical information in the hypothesis, Lemma \ref{L5.3} and the property $|B_{4\rho}|=\mathcal{C}|B_\rho|$. \vskip 1mm
    Now, integrate (\ref{eq5.9}) with respect to time over $(0,\delta\rho^2)$ and set
    \begin{align*}
        |\mathcal{A}_n|=|\{u-g>M_n\}\cap\mathcal{Q}_{4\rho}^\oplus(\delta)|=\int_0^{\delta\rho^2}|\mathcal{A}_{M_n,4\rho}(\tau)|\,d\tau.
    \end{align*}
    Then, (\ref{eq5.9}) yields with help of H\"older inequality,
    \begin{align*}
        \frac{\varepsilon\lambda\omega}{2^{n+1}}|\mathcal{A}_{n+1}|&\leq \tilde{\gamma}\rho\iint_{\mathcal{Q}_{4\rho}^\oplus(\delta)\cap\{M_n<u(\cdot,t)-g(t)<M_{n+1}\}}|\nabla_{\hn}u|\,d\xi dt\\
        &\leq \tilde{\gamma}\rho\left(\iint_{\mathcal{Q}_{4\rho}^\oplus(\delta)}|\nabla_{\hn}w_+|^2\,d\xi dt\right)^\frac{1}{2}|\mathcal{A}_n-\mathcal{A}_{n+1}|^\frac{1}{2}\\
        &\leq 2\tilde{\gamma}\,\frac{\varepsilon\lambda\omega}{2^{n+1}}\sqrt{|\mathcal{Q}_{4\rho}^\oplus(\delta)|}\left(|\mathcal{A}_n|-|\mathcal{A}_{n+1}|\right)^\frac{1}{2},\,\,\,\,\,\textup{ for } n=0,1,2,3,...,
    \end{align*}
 where $\tilde{\gamma}:=\frac{\mathcal{C}\gamma}{\alpha}$ is depending upon \texttt{data} and $\alpha$. Squaring both side of the above inequality, we have 
 \begin{align*}
    |\mathcal{A}_{n+1}|^2\leq (2\tilde{\gamma})^2|\mathcal{Q}_{4\rho}^\oplus(\delta)|^2\left(|\mathcal{A}_n|-|\mathcal{A}_{n+1}|\right).
 \end{align*}Next, adding all these inequalities for $n=0,1,2,3,...,n^*-1$, where $n^*$ is a positive integer (to be determined later). Since $\mathcal{A}_{n^*}\subset \mathcal{A}_i$ for any positive integer $i<n^*$. So the estimate yields
 \begin{align*}
     n^*|\mathcal{A}_{n^*}|^2\leq \sum_{n=0}^{n^*-1}|\mathcal{A}_{n+1}|^2&\leq (2\tilde{\gamma)}^2|\mathcal{Q}_{4\rho}^\oplus(\delta)|\sum_{n=0}^{\infty}\left(|\mathcal{A}_n|-|\mathcal{A}_{n+1}|\right)\\
     &\leq (2\tilde{\gamma})^2|\mathcal{Q}_{4\rho}^\oplus(\delta)|^2.
 \end{align*}
 Hence
 \begin{align}\label{eq5.10}
     |\mathcal{A}_{n^*}|^2\leq \frac{(2\tilde{\gamma})^2}{n^*}|\mathcal{Q}_{4\rho}^\oplus(\delta)|^2.
 \end{align}
 Thus, considering fixed $\iota\in (0,1)$, one can choose $n^*$ so large that 
 \begin{align}\label{eq5.11}
     \frac{|\{u-g>-\varepsilon_\iota\lambda\omega+\mu^+\}\cap\mathcal{Q}_{4\rho}^\oplus(\delta)|}{|\mathcal{Q}_{4\rho}^\oplus(\delta)|}<\iota,\,\,\,\,\textup{ for }\,\,\, \frac{2\tilde{\gamma}}{\sqrt{n^*}}\leq\iota,\,\,\,\textup{ and }\,\,\varepsilon_\iota=\frac{\varepsilon}{2^{n^*}}.
 \end{align}
 Next, consider the De Giorgi Lemma \ref{L5.1} over forward parabolic cylinder $\mathcal{Q}_{4\rho}^\oplus(\delta)$ with replacement of $\lambda\omega$ with $\varepsilon_\iota\lambda\omega$. In comparison with Lemma \ref{L5.1}, both $\iota$ depends only on \texttt{data} and $\delta$ (that is fixed term $\alpha$). When such $\iota$ is determined and fixed in terms of \texttt{data}, we calculate $n^*$ by the indicated procedure. The conditions for (\ref{eq5.11}) are satisfied in that case as well. Also the conditions of Lemma \ref{L5.1} hold. Then, from the conclusion of Lemma \ref{L5.1}, we have 
 \begin{align*}
     u(\xi,t)-g(t)\leq -\frac{1}{2}\varepsilon_\iota\lambda\omega+\mu^+\,\,\,\,\,\textup{ a.e. in }\,\, B_{4\rho}\times (\frac{1}{2}\delta\rho^2,\delta\rho^2].
 \end{align*}
 Now considering $\hat{\gamma}:=\max\{4\mathfrak{C},\frac{1}{\delta},\frac{4}{\varepsilon_\iota}\}$ and considering $g(T_2)\leq \frac{1}{4}\varepsilon_\iota\lambda\omega$, we have
 \begin{align*}
     u\leq -\varkappa\lambda\omega+\mu^+\,\,\,\,\,\textup{ a.e. in }\,\, B_{4\rho}\times (\frac{1}{2}\delta\rho^2,\delta\rho^2],
 \end{align*}
 where $\varkappa=\frac{1}{4}\varepsilon_\iota$.
\end{proof}
The next proposition improves the result in Lemma \ref{L5.4}. This type of result is inspired from the result in \cite[Prop. 4.2]{Lia21_14} or \cite[Prop. 4.5]{Nak23_3}. In contrast to the previous result (Lemma \ref{L5.4}), it shows that the measure theoretical information of positivity at certain time slice can be propagated further in time as much as we want. This will be useful to draw the conclusion of Proposition \ref{P6.2}. 
\begin{proposition}\label{P5.5}
    Let $L,K>0$ be arbitrary numbers and $u$ be a locally bounded, weak sub (super) solution to the problem (\ref{eq1.1}) in $\Omega_T$. Let $\rho>0$ and assume that $B_{2\rho}(\xi_0)\times (t_0,t_0+K\rho^2]\subset\mathcal{Q}\Subset\Omega_T$. Suppose that there exists $\alpha\in (0,1]$ such that 
    \begin{align*}
        |\{\pm(\mu^\pm-u(\cdot,t_0)\geq L\}\cap B_\rho(\xi_0)|\geq \alpha|B_\rho(\xi_0)|.
    \end{align*}
    Then there exists $\varkappa\in (0,1)$ depending only on \textup{\texttt{data}} and $\alpha$ such that for all time $t_0<t\leq t_0+K\rho^2$
    \begin{align*}
        \pm(\mu^\pm-u(\cdot,\xi))\geq \varkappa L-R^{-2}\int_{T_1}^{T_2}\textup{ Tail}[(u-\mu^\pm)_\pm(t);R,\xi_0]\,dt\,\,\,\,\textup{ a.e. in }\, B_\rho(\xi_0).
    \end{align*}
\end{proposition}
\begin{proof}
    This proof will follow the steps of \cite[Prop. 4.5]{Nak23_3}. Considering lemmas \ref{L5.3} and \ref{L5.4}, note that it is sufficient to show that there exists some $\bar{\varepsilon}>0$ depending on $\texttt{data}, \alpha$ and $K$, such that
\begin{align}\label{eq5.12}
    \left|\{\pm(\mu^\pm-u+g(t))\geq \bar{\varepsilon}M\}\cap B_\rho(\xi_0)\right|\geq \alpha|B_\rho(\xi_0)|, 
\end{align}
where $g(t)$ as in Lemma \ref{L5.1}. We prove this lemma for the subsolution case and choose
\begin{align*}
    aM:=-aL+\mu^+,\,\,\,\,\textup{ for any } a.
\end{align*}
For the supersolution case, similar steps will be followed with some required sign changes.\\[1mm]
{\bf Step I: }\vskip 1mm
Consider $\mathcal{A}_{aM,\rho}(t)=\left\{u(\cdot,t)-g(t)>aM\right\}\cap B_\rho$. Choose truncated function $w_+:=(u-g(t)-M)_+$, where $g(t)$ as in Lemma \ref{L5.1}. Proceeding as in (\ref{eq5.6}) in Lemma \ref{L5.3}, for $t\in (0,\delta\rho^2]$, we have

\begin{align}\label{eq5.13}
    \int_{B(1-\sigma)\rho}w^2_+(\xi,t)\,d\xi&\leq \gamma\int_{B_\rho}w^2_+(\xi,0)\,d\xi+\gamma\int_0^{\delta\rho^2}\int_{B_\rho}\left(w^2_+\,\nabla_{\hn}\psi\right)^2(\xi,t)\,d\xi\nonumber\\
        &\quad+\gamma\int_0^{\delta\rho^2}\max\{w^2_+(\xi,t),w^2_+(\eta,t)\}\frac{|\psi(\xi)-\psi(\eta)|}{\dist[\xi]{\eta}^{Q+2s}}\,d\xi d\eta dt\nonumber\\
        &\quad\quad+\gamma\int_0^{\delta\rho^2}\int_{B_\rho}w_+(\xi,t)\left\{\int_{\hn\setminus B_\rho}\frac{w_+(\eta,t)}{\dist[\xi]{\eta}^{Q+2s}}\,d\eta\right\}\,d\xi dt\\
        &\quad\quad\quad -\mathfrak{C}\iint_{\mathcal{Q}_\rho^\oplus(\delta)}g'(t)(w_+\psi^2)(\xi,t)\,d\xi dt\nonumber\\
        &\quad\quad\quad\quad =:I_1+I_2+I_3+I_4+I_5,\nonumber
\end{align}
where we set a time-independent cut-off function $\psi$ on $B_\rho$ as follows
    \begin{align*}
    \psi\equiv
        \begin{cases}
           1\,\,\,\,\mbox{ on } B_{(1-\sigma)\rho},\\
           0\,\,\,\,\mbox{ on } B^c_{(1-\sigma/2)\rho},
        \end{cases}
        \,\,\,\,\mbox{ such that }|\nabla_{\hn}\psi|\leq \frac{\gamma}{\sigma\rho}.
    \end{align*}
    Now, we estimate the terms on the right-hand side of previous inequality one by one. Consider $I_3$ and $I_2$ successively,
    \begin{align*}
        I_3&\leq \gamma\left(\frac{L}{\sigma\rho}\right)^2\int_0^{\delta\rho^2}\iint_{B_\rho\times B_\rho}\frac{\chi_{\mathcal{A}_{M,\rho}(t)}}{\dist[\xi]{\eta}^{Q+2s-2}}\,d\xi dt\\
        &=\gamma\delta \left(\frac{L}{\sigma}\right)^2\frac{\left|\left\{u-g>M\right\}\cap \mathcal{Q}_\rho^\oplus(\delta)\right|}{|\mathcal{Q}_\rho^\oplus(\delta)|}|B_\rho|,
    \end{align*}
    and similarly 
    \begin{align*}
        I_2&\leq \gamma\delta \left(\frac{L}{\sigma}\right)^2\frac{\left|\left\{u-g>M\right\}\cap \mathcal{Q}_\rho^\oplus(\delta)\right|}{|\mathcal{Q}_\rho^\oplus(\delta)|}|B_\rho|.
    \end{align*}
    Before proceeding towards $I_4$, we observe that for $\xi\in \mbox{ supp }\psi$ i.e. $|\xi|\leq (1-\frac{\sigma}{2})\rho$ and $|\eta|\geq \rho$, we have 
    \begin{align*}
        \frac{\dist[\xi]{\eta}}{|\eta|_{\hn}}\geq 1-\frac{|\xi|_{\hn}}{|\eta|_{\hn}}\geq \frac{\sigma}{2}.
    \end{align*}
    Also, when $|\eta|_{\hn}\geq R$ and $|\xi|_{\hn}\leq \rho$, we have
    \begin{align*}
        \frac{\dist[\xi]{\eta}}{|\eta|_{\hn}}\geq\frac{1}{2}\,\,\,\textup{ provided }\,\rho\leq \frac{R}{2}.
    \end{align*}
    So, $I_4$ implies 
    \begin{align*}
        I_4&\leq 2^{Q+2s}\frac{1}{\sigma^{Q+2s}}\int_0^{\delta\rho^2}\int_{B_\rho}w_+\,d\xi\left\{\int_{B_R\setminus B_\rho} \frac{w_+(\eta,t)}{|\eta|_{\hn}^{Q+2s}}\,d\eta +\int_{\hn\setminus B_R} \frac{w_+(\eta,t)}{|\eta|_{\hn}^{Q+2s}}\,d\eta \right\}dt.
    \end{align*}
    Similar to the respective part of Lemma \ref{L5.1}, we can cancel out the second integral on the right hand of previous inequality with $I_5$ by suitable choice of $\mathfrak{C}$ and $\gamma$. Then rest part on the right hand side of the above integral can be dominated by 
    \begin{align*}
        \gamma\delta\frac{L^2}{\sigma^{Q+2}}\frac{\left|\left\{u-g>M\right\}\cap \mathcal{Q}_\rho^\oplus(\delta)\right|}{|\mathcal{Q}_\rho^\oplus(\delta)|}|B_\rho|.
    \end{align*}
    The left hand side of (\ref{eq5.13}) has the following estimation for $\varepsilon\in  (0,1)$ 
    \begin{align*}
       \int_{B_{(1-\sigma)\rho}}w^2_+(\xi,t)\,d\xi \geq C(1-\varepsilon)^2L^2|\mathcal{A}_{\varepsilon M,(1-\sigma)\rho}(t)|,
    \end{align*}
    where $C$ is some positive constant. Combining all the estimates and proceeding as in Lemma \ref{L5.3}, we have
    \begin{align}\label{eq5.14}
        \left|\mathcal{A}_{\varepsilon M,\rho}(t)\right|\leq \frac{1}{(1-\varepsilon)^2}\left[(1-\alpha)+\frac{\gamma\delta}{\sigma^{Q+2}}\frac{\left|\left\{u-g>M\right\}\cap \mathcal{Q}_\rho^\oplus(\delta)\right|}{|\mathcal{Q}_\rho^\oplus(\delta)|}+Q\sigma\right]|B_\rho|
    \end{align}
    Choose $\sigma\in (0,1)$, so that
    \begin{align*}
        \sigma=\delta^\frac{1}{Q+3}\left(\frac{\left|\left\{u-g>M\right\}\cap \mathcal{Q}_\rho^\oplus(\delta)\right|}{|\mathcal{Q}_\rho^\oplus(\delta)|}\right)^\frac{1}{Q+3}.
    \end{align*}
    Then, (\ref{eq5.14}) implies, 
    \begin{align*}
        \left|\mathcal{A}_{\varepsilon M,\rho}(t)\right|\leq \frac{1}{(1-\varepsilon)^2}\left[(1-\alpha)+\gamma\delta^\frac{1}{Q+3}\left(\frac{\left|\left\{u-g>M\right\}\cap \mathcal{Q}_\rho^\oplus(\delta)\right|}{|\mathcal{Q}_\rho^\oplus(\delta)|}\right)^\frac{1}{Q+3}\right]|B_\rho|.
    \end{align*}
    Further, choose 
    \begin{align}\label{eq5.15}
        \begin{cases}
            \delta=\left(\frac{\alpha}{8\gamma}\right)^{Q+3},\\
            \frac{1}{(1-\epsilon)^2}\leq \frac{1-\frac{3}{4}\alpha}{1-\frac{7}{8}\alpha},
        \end{cases}
    \end{align}
    then
    \begin{align}\label{eq5.16}
        \left|\mathcal{A}_{\varepsilon M,\rho}(t)\right|\leq\left(1-\frac{3}{4}\alpha\right)|B_\rho|,
    \end{align}
    for any $0<t\leq \delta\rho^2=:t_1$.\\[2mm]
    {\bf Step II: }\vskip 1mm
    Set  
    \begin{align*}
         \begin{cases}
             L_1:=\varepsilon L,\,\,\varepsilon_1:=\frac{1}{2^{n_1+i_1+1}},\,\,M_1:=\varepsilon M=-L_1+\mu^+,\,\,\mathcal{M}_1:=-\frac{L_1}{2^{i_1+1}}+\mu^+,\\
             \mathcal{Q}_1:=B_\rho\times [t_1,t_2]\,\,\textup{ with }\,t_2:=t_1+\delta\rho^2,
         \end{cases}
    \end{align*}
    where both natural numbers $i_1,n_1$ will be calculated later. Next, write (\ref{eq5.13}) with these parameters and truncated function $w_+:=(u-g(t)-\mathcal{M}_1)_+$ over the cylinder $\mathcal{Q}_1$ to have for any $t\in (t_1,t_2)$,
    \begin{align}\label{eq5.17}
        \int_{B(1-\sigma_1)\rho}w^2_+(\xi,t)\,d\xi&\leq \int_{B_\rho}w^2_+(\xi,t_1)\,d\xi+\gamma\int_{t_1}^{t_2}\int_{B_\rho}\left(w^2_+\,\nabla_{\hn}\psi\right)^2(\xi,t)\,d\xi\nonumber\\
        &\quad+\gamma\int_{t_1}^{t_2}\max\{w^2_+(\xi,t),w^2_+(\eta,t)\}\frac{|\psi(\xi)-\psi(\eta)|}{\dist[\xi]{\eta}^{Q+2s}}\,d\xi d\eta dt\nonumber\\
        &\quad\quad+\gamma\int_{t_1}^{t_2}\int_{B_\rho}w_+(\xi,t)\left\{\int_{\hn\setminus B_\rho}\frac{w_+(\eta,t)}{\dist[\xi]{\eta}^{Q+2s}}\,d\eta\right\}\,d\xi dt\\
        &\quad\quad\quad -\mathfrak{C}\iint_{\mathcal{Q}_1}g'(t)(w_+\psi^2)(\xi,t)\,d\xi dt\nonumber\\
        &\quad\quad\quad\quad =:I_1+I_2+I_3+I_4+I_5\nonumber,
    \end{align}
    for some $\sigma_1\in (0,1)$ to be determined later. Recalling the measure theoretical information (\ref{eq5.16}) at time level $t=t_1$, we have 
    \begin{align*}
        I_1\leq \left(\frac{L_1}{2^{i_1+1}}\right)^2\left|\{u(\cdot,t_1)-g(t_1)>\mathcal{M}_1\}\cap B_\rho \right|&\leq \left(\frac{L_1}{2^{i_1+1}}\right)^2\left|\{u(\cdot,t_1)-g(t_1)> M_1\}\cap B_\rho \right|\\
        &\leq \left(1-\frac{3}{4}\alpha\right) \left(\frac{L_1}{2^{i_1+1}}\right)^2|B_\rho|.
    \end{align*}
    Rest of the right hand side of (\ref{eq5.17}) can be estimated as in {\bf Step I} to have 
    \begin{align*}
        I_2+I_3+I_4+I_5\leq \frac{\gamma(\texttt{data})\delta}{\sigma_1^{Q+2}}\left(\frac{L_1}{2^{i_1+1}}\right)^2\frac{\left|\left\{u-g>\mathcal{M}_1\right\}\cap \mathcal{Q}_1\right|}{|\mathcal{Q}_1|}|B_\rho|.
    \end{align*}
    Combining all the estimates into (\ref{eq5.17}), we have
    \begin{align}\label{eq5.18}
        \int_{B(1-\sigma_1)\rho}w^2_+(\xi,t)\,d\xi\leq \left[\left(1-\frac{3}{4}\alpha\right)+\frac{\gamma(\texttt{data})\delta}{\sigma_1^{Q+2}}\frac{\left|\left\{u-g>\mathcal{M}_1\right\}\cap \mathcal{Q}_1\right|}{|\mathcal{Q}_1|}\right] \left(\frac{L_1}{2^{i_1+1}}\right)^2|B_\rho|.
    \end{align}
    Similar to the {\bf Step I}, the left hand side of (\ref{eq5.18}) can be estimated , 
    \begin{align*}
        \int_{B(1-\sigma_1)\rho}w^2_+(\xi,t)\,d\xi
        &\geq \int_{\mathcal{A}_{\varepsilon_1M_1,(1-\sigma_1)\rho}(t)}\left[(\varepsilon_1 M_1-\mathcal{M}_1)\right]\,d\xi\\
        & \geq (1-2^{-n_1})^2\left(\frac{L_1}{2^{i_1+1}}\right)^2\left|\mathcal{A}_{\varepsilon_1M_1,(1-\sigma_1)\rho}(t)\right|,
    \end{align*}
    for some constant $C$. Therefore inserting this into (\ref{eq5.18}), we get , for every $t_1\leq t\leq t_2$
    \begin{align}\label{eq5.19}
        \left|\mathcal{A}_{\varepsilon_1M_1,\rho}(t)\right|\leq \frac{1}{(1-2^{-n_1})^2}\left[\left(1-\frac{3}{4}\alpha\right)+\frac{\gamma(\texttt{data})\delta}{\sigma_1^{Q+2}}\frac{\left|\left\{u-g>M_1\right\}\cap \mathcal{Q}_1\right|}{|\mathcal{Q}_1|}+Q\sigma_1\right]|B_\rho|.
    \end{align}
    Choose $\sigma_1$ as,
    \begin{align*}
        \sigma_1=\delta^{\frac{1}{Q+3}}\left(\frac{\left|\left\{u-g>\mathcal{M}_1\right\}\cap \mathcal{Q}_1\right|}{|\mathcal{Q}_1|}\right)^{\frac{1}{Q+3}}.
    \end{align*}
    Then (\ref{eq5.19}) implies,
    \begin{align*}
         \left|\mathcal{A}_{\varepsilon_1M_1,\rho}(t)\right|\leq \frac{1}{(1-2^{-n_1})^2}\left[\left(1-\frac{3}{4}\alpha\right)+\gamma\delta^{\frac{1}{Q+3}}\left(\frac{\left|\left\{u-g>\mathcal{M}_1\right\}\cap \mathcal{Q}_1\right|}{|\mathcal{Q}_1|}\right)^{\frac{1}{Q+3}}\right]|B_\rho|.
    \end{align*}
    At this point, keeping the parameters $L_1,\varepsilon_1$ in hand, we recall the arguments and steps that lead to (\ref{eq5.10}) in Lemma \ref{L5.4}, we get
    \begin{align*}
        \frac{\left|\left\{u-g>M_1\right\}\cap \mathcal{Q}_1\right|}{|\mathcal{Q}_1|}\leq \frac{\gamma}{\alpha\sqrt{i_1}}\leq \frac{\gamma}{\alpha\delta\sqrt{i_1}},
    \end{align*}
    for some constant $\gamma=\gamma(\texttt{data})$. Therefore, the previous estimate looks like
    \begin{align*}
        \left|\mathcal{A}_{\varepsilon_1M_1,\rho}(t)\right|\leq \frac{1}{(1-2^{-n_1})^2}\left[\left(1-\frac{3}{4}\alpha\right)+\frac{\gamma(\texttt{data)}}{\alpha^{\frac{1}{Q+3}}i_1^{\frac{1}{2(Q+3)}}}\right]|B_\rho|.
    \end{align*}
    For given $K$, we choose $i_1,n_1\in \mathbb{N}$ in such a way that
    \begin{align*}
        \frac{\gamma(\texttt{data)}}{\alpha^{\frac{1}{Q+3}}i_1^{\frac{1}{2(Q+3)}}(1-2^{-n_1})^2}\leq \frac{\alpha\delta}{8K},\,\,\,\,\,\,\frac{\left(1-\frac{3}{4}\alpha\right)}{(1-2^{-n_1})^2}\leq 1-\frac{3}{4}\alpha+ \frac{\alpha\delta}{8K},
    \end{align*}
    to obtain, for all $t_1\leq t\leq t_2$
    \begin{align*}
         \left|\mathcal{A}_{\varepsilon_1M_1,\rho}(t)\right|\leq\left[1-\frac{3}{4}\alpha+ \frac{\alpha\delta}{4K}\right]|B_\rho|.
    \end{align*}\\[2mm]
    {\bf Step III: }\vskip 1mm
    In this step, we begin the induction process. Fix a natural number $j(>2)\in\mathbb{N}$. In Step II, $i=1$ case has been done and we assume that procedure has been done inductively upto $(j-1)-$th steps; in turn, three sequences $\{\varepsilon_i\}_{i=1}^{j-1}$, $\{L_i\}_{i=1}^{j-1}$ and $\{t_i\}_{i=1}^{j-1}$ have been selected up to the $(j-1)-$th and assume that
    \begin{align}\label{eq5.20}
        \left|\mathcal{A}_{\varepsilon_{j-1}M_{j-1},\rho}(t)\right|\leq\left[1-\frac{3}{4}\alpha+ (j-1)\frac{\alpha\delta}{4K}\right]|B_\rho|,\,\,\,\,\,\textup{ for all } t\in [t_{j-1},t_j],
    \end{align}
    where $t_j:=t_{j-1}+\delta\rho^2$ and $\varepsilon_{j-1}:=\frac{1}{2^{n_{j-1}+i_{j-1}+1}}$ with $i_{j-1},n_{j-1}\in\mathbb{N}$ being large enough. Also note that the parameters $\delta$ and $\alpha$ have been determined in {\em Step I}. Now set 
    \begin{align*}
       \begin{cases}
           L_j:=\varepsilon L_{j-1},\,\varepsilon_j:=\frac{1}{2^{n_j+i_j+1}},\,\,M_j:=\varepsilon^j M=-L_j+\mu^+,\,\,\mathcal{M}_j:=-\frac{L_j}{2^{i_j+1}}+\mu^+,\\
           \mathcal{Q}_j:=B_\rho\times [t_j,t_{j+1}]\,\,\textup{ with } t_{j+1}:=t_j+\delta\rho^2,
       \end{cases}
    \end{align*}
  where both natural numbers $n_j,i_j$ are large enough so that 
  \begin{align*}
      n_j\geq n_{j-1}\,\,\textup{ and }\,\,i_j\geq n_{j-1}+i_{j-1}.
  \end{align*}
  Write (\ref{eq5.13}) with these parameters and truncated function $w_+:=(u-g(t)-\mathcal{M}_j)_+$ over the cylinder $\mathcal{Q}_j$ and estimate similar to {\em Step II}, we get
  \begin{align}\label{eq5.21}
      \int_{B_{(1-\sigma_j)\rho}}w^2_+(\xi,t)\,d\xi\leq \int_{B_\rho}w^2_+(\xi,t_j)\,d\xi+ \frac{\gamma(\texttt{data})\delta}{\sigma_j^{Q+2}}\left(\frac{L_j}{2^{i_j+1}}\right)^2\frac{\left|\left\{u-g>\mathcal{M}_j\right\}\cap \mathcal{Q}_j\right|}{|\mathcal{Q}_j|}|B_\rho|,
  \end{align}
  where $\sigma_j\in (0,1)$ to be determined later. Using the measure theoretical information (\ref{eq5.20}) at time level $t=t_j$, the first integral on the right hand side of (\ref{eq5.21}) can be estimated as 
  \begin{align*}
      \int_{B_\rho}w^2_+(\xi,t_j)\,d\xi&\leq \left(\frac{L_j}{2^{i_j+1}}\right)^2\left|\{u(\cdot,t_j)-g(t_j)>\mathcal{M}_j\}\cap B_\rho \right|\\
      & \leq \left(\frac{L_j}{2^{i_j+1}}\right)^2\left|\{u(\cdot,t_j)-g(t_j)>-\frac{L_{j-1}}{2^{n_{j-1}+i_{j-1}+1}}+\mu^+\}\cap B_\rho \right|\\
      &\leq \left(\frac{L_j}{2^{i_j+1}}\right)^2\left|\mathcal{A}_{\varepsilon_{j-1}M_{j-1},\rho}(t)\right|\\
      & \leq \left(\frac{L_j}{2^{i_j+1}}\right)^2\left[1-\frac{3}{4}\alpha+ (j-1)\frac{\alpha\delta}{4K}\right]|B_\rho|.
  \end{align*}
  Similar to what was done before, the left hand side of (\ref{eq5.21}) can be estimated as 
  \begin{align*}
      \int_{B_{(1-\sigma_j)\rho}}w^2_+(\xi,t)\,d\xi&\geq \int_{\mathcal{A}_{\varepsilon_j M_j,(1-\sigma_j)\rho}(t)}w^2_+(\xi,t)\,d\xi\\
      &\geq (\varepsilon_jM_j-\mathcal{M}_j)^2|\mathcal{A}_{\varepsilon_j M_j,(1-\sigma_j)\rho}(t)|\\
      &\geq (1-2^{-n_j})^2\left(\frac{L_j}{2^{i_j+1}}\right)^2\left[|\mathcal{A}_{\varepsilon_j M_j,\rho}(t)|-Q\sigma_j|B_\rho|\right].
  \end{align*}
  Combining all these estimates into (\ref{eq5.21}), we have
  \begin{align*}
      |\mathcal{A}_{\varepsilon_j M_j,\rho}(t)|\leq \frac{1}{(1-2^{-n_j})^2}\left[1-\frac{3}{4}\alpha+ (j-1)\frac{\alpha\delta}{4K}+\frac{\gamma\delta}{\sigma_j^{Q+2}}\frac{\left|\left\{u-g>\mathcal{M}_j\right\}\cap \mathcal{Q}_j\right|}{|\mathcal{Q}_j|}+Q\sigma_j\right]|B_\rho|.
  \end{align*}
   Choose $\sigma_j\in (0,1)$ as 
   \begin{align*}
       \sigma_j=\delta^{\frac{1}{Q+3}}\left(\frac{\left|\left\{u-g>\mathcal{M}_j\right\}\cap \mathcal{Q}_j\right|}{|\mathcal{Q}_j|}\right)^{\frac{1}{Q+3}}.
   \end{align*}
  Then it implies that
  \begin{align*}
      |\mathcal{A}_{\varepsilon_j M_j,\rho}(t)|\leq \frac{1}{(1-2^{-n_j})^2}\left[1-\frac{3}{4}\alpha+ (j-1)\frac{\alpha\delta}{4K}+\gamma\delta^{\frac{1}{Q+3}}\left(\frac{\left|\left\{u-g>\mathcal{M}_j\right\}\cap \mathcal{Q}_j\right|}{|\mathcal{Q}_j|}\right)^{\frac{1}{Q+3}}\right]|B_\rho|.
  \end{align*}
  Again similar to {\em Step II}, considering Lemma \ref{L5.4}, we have 
  \begin{align*}
      \frac{\left|\left\{u-g>\mathcal{M}_j\right\}\cap \mathcal{Q}_j\right|}{|\mathcal{Q}_j|}\leq \frac{\gamma}{\alpha\delta\sqrt{i_j}},
  \end{align*}
  for some constant $\gamma=\gamma(\texttt{data})$. Then, we get
  \begin{align*}
      \left|\mathcal{A}_{\varepsilon_jM_j,\rho}(t)\right|\leq \frac{1}{(1-2^{-n_j})^2}\left[1-\frac{3}{4}\alpha+(j-1)\frac{\alpha\delta}{4K}+\frac{\gamma}{\alpha^{\frac{1}{Q+3}}i_j^{\frac{1}{2(Q+3)}}}\right]|B_\rho|.
  \end{align*}
  Finally for given $K$, we choose $i_j,n_j\in\mathbb{N}$ in such a way that
  \begin{align*}
      \frac{\gamma}{\alpha^{\frac{1}{Q+3}}i_j^{\frac{1}{2(Q+3)}}(1-2^{-n_j})^2}\leq \frac{\alpha\delta}{8K}\,\,\textup{ and }\,\,\frac{1-\frac{3}{4}\alpha+(j-1)\frac{\alpha\delta}{4K}}{(1-2^{-n_j})^2}\leq 1-\frac{3}{4}\alpha+(j-1)\frac{\alpha\delta}{4K}+\frac{\alpha\delta}{8K}.
  \end{align*}
  Consequently, these all yield, for every $t\in[t_j,t_{j+1}]$,
  \begin{align}\label{eq5.22}
      \left|\mathcal{A}_{\varepsilon_jM_j,\rho}(t)\right|\leq\left[1-\frac{3}{4}\alpha+j\frac{\alpha\delta}{4K}\right]|B_\rho|.
  \end{align}
  Hence, induction procedure is complete. \\[2mm]
  {\bf Step IV: }\vskip 1mm
  Repeating the procedure which leads to (\ref{eq5.22}) up to $j\in\mathbb{N}$ satisfying
  \begin{align*}
      \frac{K}{\delta}-1<j\leq\frac{K}{\delta}
  \end{align*}
  and choose $\bar{\varepsilon}\in (0,1)$ in such a way that
  \begin{align*}
      \bar{\varepsilon}L=\varepsilon_jL_j=\frac{L_j}{2^{n_j+i_j+1}}.
  \end{align*}
  Therefore, we get 
  \begin{align*}
      |\mathcal{A}_{\bar{\varepsilon}M,\rho}(t)|\leq \left(1-\frac{\alpha}{2}\right)|B_\rho|,
  \end{align*}
  for all $0<t\leq \sum_{i=1}^{j+1}(t_i-t_{i-1})=(j+1)\delta\rho^2$ which eventually implies that
  \begin{align*}
      &|\{u(\cdot,t)-g(t)\leq -\bar{\varepsilon}L+\mu^+\}\cap B_\rho|\geq \frac{\alpha}{2}|B_\rho|,
  \end{align*}
  for all $0<t\leq K\rho^2$.
\end{proof}

\section{Harnack Inequality}\label{S6}
\subsection{Weak Harnack Inequality I}
We start this subsection with a lemma, known as {\em Clustering Lemma}, is a quantitative version of the Severini-Egorov Theorem \cite{Sev10_21,Ego11_20}. For Sobolev case, this theorem is discussed in \cite{DBGV06_22} and for fractional case, see \cite{DIV23_23}. The interesting part of this lemma is that unlike to the Euclidean space, we can not have the covering of cubes due to the geometrical structure of Heisenberg group. For this purpose, we relies on the covering result \cite[Lemma 3.2]{Lu94_16} and steps in \cite[Lemma 2.5]{KMMJP21_24},\cite[Lemma 5.1]{Nak23_3}. This lemma is used to refine the result in Lemma \ref{L5.4} to the power like dependence result in Proposition \ref{P6.2}. The important feature of this lemma is that it helps to prove $\delta$ independent Weak Harnack Inequality (Weak Harnack Inequality I).

Note that, we will consider $W^{1,2}(B_\rho(\xi_0))$ instead of $W^{s,2}(B_\rho(\xi_0))$ as the former one can be embedded into the later one (see Lemma \ref{L2.8}).
\begin{lemma}\label[lemma]{L6.1}(Clustering Lemma)
   Let, $\rho_0>\rho>0$ with $s\in (0,1)$. Consider $\xi_0\in\hn$ and $u\in W^{1,2}(B_\rho(\xi_0))$ satisfies
   \begin{align}\tag{C-I}\label{eqCI}
       \left|\left\{u(.)>c\right\}\cap B_\rho(\xi_0)\right|>\alpha|B_\rho(\xi_0)|,
   \end{align}
   and
   \begin{align}\tag{C-II}\label{eqCII}
       [u]_{W^{1,2}(B_\rho(\xi_0)}\leq B,
   \end{align}
   for some $\alpha\in (0,1)$ and $c, B>0$. Then, for every $\delta,\lambda\in (0,1)$ there exists $\xi_*\in B_\rho(\xi_0)$ and a parameter $k\equiv k(Q,s,\alpha, B, \delta, \lambda)\in (0,1)$ such that
   \begin{align*}
       \left|\left\{u(.)>\lambda c\right\}\cap B_{k\rho}(\xi_*)\right|>(1-\delta)|B_{k\rho}(\xi_*)|.
   \end{align*}
   \end{lemma}
\begin{proof}
    Let, $0<k<\frac{\rho_0-\rho}{2m\rho}$, for some constant $m>1$. Let $\left\{ B_{k\rho}(\xi)\right\}_{x\in B_\rho(\xi_0)}$ be a covering of $B_\rho(\xi_0)$, then there exists disjoints sequence $\left\{ B_{k\rho}(\xi_i)\right\}_{\xi_i\in B_\rho(\xi_0)}$ (see \cite{MS79_13}), for all $i\in I$, such that 
    \begin{align}\label{eq6.1}
        B_\rho(\xi_0)\subset\cup_{i\in I}B_{2k\rho}(\xi_0)\subset B_{\rho_0}(x_0).
    \end{align}
    It is obvious that $B_{2mk\rho}(\xi_i)\subset B_R(\xi_0)$. Moreover, the ball $B_{2mk\rho}(\xi_i)$ have bounded overlap with bound independent of $k$, due to the measure doubling property.  Next, let us define
    \begin{align*}
        \begin{cases}
            D^+:=\left\{i\in I: \left|\{u>c\}\cap B_{2k\rho}(\xi_0)\right|>\frac{\alpha}{2\mathcal{C}}|B_{2k\rho}(\xi_0)|\right\},\\
            D^-:=\left\{i\in I: \left|\{u>c\}\cap B_{2k\rho}(\xi_0)\right|\leq \frac{\alpha}{2\mathcal{C}}|B_{2k\rho}(\xi_0)|\right\}.
        \end{cases}
    \end{align*}
    By assumption, we get 
    \begin{align*}
        \alpha|B_\rho(\xi_0)|&\leq \left|\left\{u(.)> c\right\}\cap B_\rho(\xi_0)\right|\\
        &\leq \sum_{i\in D^+}\left|\left\{u(.)> c\right\}\cap B_{2k\rho}(\xi_0)\right|+\sum_{i\in D^-}\left|\left\{u(.)> c\right\}\cap B_{2k\rho}(\xi_0)\right|\\
        & \leq\underbrace{ \sum_{i\in D^+}\left|\left\{u(.)> c\right\}\cap B_{2k\rho}(\xi_0)\right|}_{A}+\frac{\alpha}{2\mathcal{C}}\sum_{i\in D^-}\left|B_{2k\rho}(\xi_i)\right|\\
        &\leq A+\frac{\alpha}{2}\sum_{i\in D^-}\left|B_{k\rho}(\xi_i)\right|\\
        & \leq A+\frac{\alpha}{2} \left|B_{(1+k)\rho}(\xi_0)\right|.
    \end{align*}
    The last inequality  follows by the fact that $\cup_{i\in I}B_{k\rho}(\xi_i)\subset B_{(1+k)\rho}(\xi_0)$. Consequently, we have 
    \begin{align}\label{eq6.2}
        \frac{\alpha}{2}\left\{|B_\rho(\xi_0)|-|B_{(1+k)\rho}(\xi_0)\setminus B_{k\rho}(\xi_0)|\right\}\leq A.
    \end{align}
   Whenever \ref{eqCI} and \ref{eqCII} are true, there might be two outcomes namely either
   \begin{align}\tag{I}\label{I}
       \left|\left\{u>\lambda c\right\}\cap B_{k\rho}(\xi_i)\right|>(1-\delta)|B_{k\rho}(\xi_i)|,
   \end{align}
   or
   \begin{align}\tag{II}\label{II}
       \left|\left\{u>\lambda c\right\}\cap B_{k\rho}(\xi_i)\right|\leq(1-\delta)|B_{k\rho}(\xi_i)|.
   \end{align}
   Now, if (I) holds, then let $\xi_*\in B_\rho(\xi_0)$ be the center of $B_{k\rho}$ and the result follows. If (II) is true for $i\in D^+$, then 
   \begin{align}\label{eq6.3}
        \left|\left\{u\leq\lambda c\right\}\cap B_{k\rho}(\xi_i)\right|>\delta|B_{k\rho}(\xi_i)|.
   \end{align}
   For every $\xi\in \left\{u(.)>c\right\}\cap B_{2k\rho}(\xi_i)$ and $\eta\in \left\{u(.)\leq\lambda c\right\}\cap B_{k\rho}(\xi_i)$, we have
   \begin{align*}
       u(\xi)-u(\eta)>c(1-\lambda)>0,
   \end{align*}
   and 
\begin{align*}
    \dist[\xi]{\eta}&\leq \dist[\xi]{\xi_i}+\dist[\xi_i]{\eta}\\
    &\leq 3k\rho.
\end{align*}
Therefore, summing up over $D^+$ and using the bounded overlapping property, assumptions \ref{eqCII}, (\ref{eq6.2}) and (\ref{eq6.3}), we have by Lemma \ref{L2.8}
\begin{align*}
    \frac{\alpha\delta}{2}(1-\lambda)^2&\left[|B_\rho(\xi_0)|-|B_{(1+k)\rho}(\xi_0)\setminus B_\rho(\xi_0)\right]\\
    &\leq C \sum_{i\in D^+}\int_{\{u(.)<\lambda c\}\cap B_{k\rho}(\xi_i)}\int_{\{u(.)> c\}\cap B_{2k\rho}(\xi_i)}\frac{(3k\rho)^{Q+2s}}{|B_{k\rho}(\xi_i)|}\frac{|u(\xi-u(\eta)|^2}{\dist[\xi]{\eta}^{Q+2s}}\, d\xi d\eta\\
    &\leq 3^{Q+2s}C(k\rho)^{2s}\iint_{B_\rho(\xi_0)\times B_\rho(\xi_0)} \frac{|u(\xi-u(\eta)|^2}{\dist[\xi]{\eta}^{Q+2s}}\, d\xi d\eta\\
    & \leq 3^{Q+2s}CB(k\rho_0)^{2s},
\end{align*}
where $C=C(Q,s)$ is a constant. Consequently, sending $k\rightarrow0$ in the last inequality renders that
\begin{align*}
    \frac{\alpha\delta}{2}(1-\lambda)^2|B_\rho(\xi_0)|=0,
\end{align*}
which contradicts that $|B_\rho(\xi_0)|>0$. 
\end{proof}
The following lemma is a refinement of the {\em Expansion of Positivity} Lemma \ref{L5.4}.  This lemma propagates the said result with a power like dependence. Unlike \cite[Lemma 4.5]{CN25_2}, we have this result using the clustering lemma (Lemma \ref{L6.1}). The idea of the proof of this lemma is taken from \cite[Prp.5.3]{Nak23_3}. In contrast to the result in \cite[Prp.5.3]{Nak23_3}, our respective result needs no restriction over $s$.  
\begin{proposition}\label{P6.2}
    Let $(\xi_0,t_0)\in\Omega_T$ and $u$ be a local supersolution of the problem (\ref{eq1.1}) in $\mathcal{Q}$. For any $M>0$, there exists $\alpha\in (0,1]$ such that
    \begin{align*}
       \left|\{u(t_0)\geq M\}\cap B_\rho(\xi_0)\right|\geq \alpha|B_\rho(\xi_0)|.
    \end{align*}
  Then, there exists $\delta_0,\tilde{\varkappa}\in (0,1)$ and $d>1$, all depending on \textup{\texttt{data}} and $\alpha$, yields
  \begin{align*}
      u(t)\geq \tilde{\varkappa}\alpha^dM-R^{-2}\int_{T_1}^{T_2}\textup{Tail}[u_-(t);R,\xi_0]\,\,\,\,\,\textup{ a.e. in } B_{2\rho}(\xi_0),
  \end{align*}
  for every $t_0+\frac{1}{2}\delta_0\rho^2\leq r\leq t_0+\delta_0\rho^2$. 
\end{proposition}
\begin{proof}
    We prove this result through several steps, inspired from the proof of the Proposition 5.3 in \cite{Nak23_3}. \\[1mm]
    {\em Step I}\vskip 2mm
    Due to the assumptions, Lemma \ref{L5.3}, there exists $\varepsilon,\delta\in (0,1)$ (depending upon \texttt{data}) such that 
    \begin{align}
        \left|\{u(t)\geq \varepsilon M\}\cap B_\rho(\xi_0)\right|\geq \frac{\alpha}{2}|B_\rho(\xi_0)|,
    \end{align}
    whenever
    \begin{align*}
        \frac{1}{\delta}R^{-2}\int_{T_1}^{T_2}\textup{Tail }[u_-(t);R,\xi_0]\,dt\leq M,
    \end{align*}
    for all $t\in [t_0,t_0+\delta\rho^2]$. Set two parabolic cylinders as follows
    \begin{align*}
        & \mathcal{Q}_1:= B_{2\rho}(\xi_0)\times (t_0+\frac{1}{2}\delta\rho^2,t_0+\delta\rho^2],\\
        & \mathcal{Q}_2:= B_{4\rho}(\xi_0)\times (t_0,t_0+\delta\rho^2].
    \end{align*}
    Taking into account the truncated function $(u(\xi,t)-M)_-$ in the energy estimate (\ref{eq3.6}) and proceeding as in Lemma \ref{L5.1}, we achieve the following 
    \begin{align}\label{eq6.5}
        \int_{\frac{1}{2}\delta\rho^2}^{\delta\rho^2}\int_{B_{2\rho}(\xi_0)}|\nabla_{\hn}(u-M)_-|^2\,d\xi dt\leq \gamma M^2|B_{2\rho}(\xi_0)|,
    \end{align}
    where $\gamma$ is a constant depending only on $\texttt{data}$. \\[2 mm]
    {\em Step II}\vskip 1mm
    We start this step with rescaled function
    \begin{align*}
        w(\xi,t):=\frac{1}{\varepsilon}\left(1-\frac{\left(u(\xi,t)-M\right)_-}{M}\right).
    \end{align*}
    Using this rescaled version in the previous inequality (\ref{eq6.5}), it follows that
    \begin{align*}
        \int_{\frac{1}{2}\delta\rho^2}^{\delta\rho^2}\fint_{B_{2\rho}(\xi_0)}|\nabla_{\hn}w(\xi,t)|^2\,d\xi dt\leq \frac{\gamma}{\varepsilon^2},
    \end{align*}
    and 
    \begin{align*}
        \left|\{w(t)\geq 1\}\cap B_\rho(\xi_0)\right|\geq \frac{\alpha}{2}\left|B_\rho(\xi_0)\right|,
    \end{align*}
    for all $t\in \left(t_0+\frac{1}{2}\delta\rho^2,t_0+\delta\rho^2\right]$. Then there exists a time level $t_1\in \left(t_0+\frac{1}{2}\delta\rho^2,t_0+\delta\rho^2\right]$ such that the previous inequality implies
    \begin{align*}
        \fint_{B_{2\rho}(\xi_0)}|\nabla_{\hn}w(\xi,t_1)|^2\,d\xi dt\leq \frac{\gamma}{\delta\rho^2\varepsilon^2},
    \end{align*}
    and
    \begin{align*}
        \left|\{w(t_1)\geq 1\}\cap B_\rho(\xi_0)\right|\geq \frac{\alpha}{2}\left|B_\rho(\xi_0)\right|.
    \end{align*}
    Next, we are in stage to apply Clustering Lemma \ref{L6.1} with $\delta=\lambda=1$ and $c=1$. This implies that there exists a point $\xi_*\in B_\rho(\xi_0)$ and a scale parameter $k=k(Q,s,\alpha,\varepsilon)\in (0,1)$ such that
    \begin{align*}
        &\left|\left\{w(t_1)>\frac{1}{2}\right\}\cap B_{k\rho}(\xi_*)\right|>\frac{|B_{k\rho}(\xi_*)}{2}>\frac{\alpha}{2}|B_{k\rho}(\xi_*)|\\
        \iff &\left|\left\{w(t_1)>\frac{1}{2}\varepsilon M\right\}\cap B_{k\rho}(\xi_*)\right|>\frac{\alpha}{2}|B_{k\rho}(\xi_*)| .
        \end{align*}
    It is useful to note that without loss of any generality, the parameter $k>0$ can be stipulated to satisfy $k<\frac{1}{4}$.\\[2mm]
    {\em Step III}\vskip 1mm
    Applying Lemma \ref{L5.4} with $\alpha=\frac{1}{2},\,(\xi_0,t_0)=(\xi_*,t_1)$ and $\lambda\omega$ replaced with $\frac{\varepsilon M}{2}$, then there exists $\varkappa_1,\delta\in (0,1)$, both depending only on \texttt{data}, such that
    \begin{align}\label{eq6.6}
        u(t_1)\geq \varkappa_1\frac{\varepsilon M}{2}\,\,\,\,\textup{ a.e. in } B_{2k\rho}(\xi_*)\times \left(t_1+\frac{\delta_1}{2}(k\rho)^2,t_1+\delta_1(k\rho)^2\right],
    \end{align}
    whenever
    \begin{align}\label{eq6.7}
        R^{-2}\int_{T_1}^{T_2}\textup{Tail }[u_-(t);R,\xi_0]\,dt\leq \varkappa_1\frac{\varepsilon M}{2}.
    \end{align}
    Then, like {Step II}, by (\ref{eq6.6}) we have, at the time level $t_2=t_1+\delta_1(k\rho)^2$, 
    \begin{align*}
        \left|\left\{u(t_2)\geq \varkappa_1\frac{\varepsilon M}{2}\right\}\cap B_{2k\rho}(\xi_*)\right|=|B_{2k\rho}(\xi_*)|.
    \end{align*}
    Recalling, Lemma \ref{L5.4} with $\varkappa_2, \delta_2\in (0,1)$, both depending only on \texttt{data}, such that 
    \begin{align}\label{eq6.8}
        u(t)\geq \varkappa_2\varkappa_1\frac{\varepsilon M}{2}\,\,\,\,\textup{ a.e. in } B_{4k\rho}(\xi_*)\times \left(t_2+\frac{\delta_2}{2}(2k\rho)^2, t_2+\delta_2(2k\rho)^2\right],
    \end{align}
    whenever
    \begin{align}\label{eq6.9}
       R^{-2}\int_{T_1}^{T_2}\textup{Tail }[u_-(t);R,\xi_0]\,dt\leq \varkappa_1\varkappa_2\frac{\varepsilon M}{2}.
    \end{align}
    Let, $\hat{n}\in \mathbb{N}\setminus \{1,2\}$ so that
    \begin{align*}
        2^{2-\hat{n}}<k\leq 2^{3-\hat{n}}\,\,\,\textup{and}\,\,\,B_{2\rho}(\xi_0)\subset B_{2^{\hat{n}}k\rho}(\xi_*).
    \end{align*}
    Continuing the same procedures which led to the inequality (\ref{eq6.8}) and (\ref{eq6.9}), we have the sequences of parameters viz. $\{\varkappa_i\}_{i=1}^{\hat{n}}$, $\{\delta_i\}_{i=1}^{\hat{n}}$ in $(0,1)$, depending on \texttt{data}, such that 
    \begin{align}\label{eq6.10}
        u(t)\geq \left(\Pi_{i=1}^{\hat{n}}\varkappa_i\right)\frac{\varepsilon}{2}M\,\,\,\,\textup{ a.e. in } B_{2^{\hat{n}}k\rho}(\xi_*)\times\left(t_{\hat{n}}+\frac{\delta_{\hat{n}}}{2}\left(2^{\hat{n}-1}k\rho\right)^2, t_{\hat{n}}+\delta_{\hat{n}}\left(2^{\hat{n}-1}k\rho\right)^2\right],
    \end{align}
    whenever
    \begin{align}\label{eq6.11}
        R^{-2}\int_{T_1}^{T_2}\textup{Tail }[u_-(t);R,\xi_0]\,dt\leq \left(\Pi_{i=1}^{\hat{n}}\varkappa_i\right)\frac{\varepsilon M}{2},
    \end{align}
    where $t_{\hat{n}}:=t_{\hat{n}-1}+\delta_{\hat{n}-1}\left(2^{\hat{n}-2}k\rho\right)^2$. Unlike, the both $\varkappa_1$ and $\delta_1$, the others are independent of $\alpha$. Moreover, conditions (\ref{eq6.7}) and (\ref{eq6.9}) can be verified from (\ref{eq6.11}). From the proofs of Lemma \ref{L5.3} and \ref{L5.4}, we can stipulate $\delta_1$ and $\varkappa_1$ to fulfill
    \begin{align*}
        \delta_1\leq \frac{\alpha^{Q+3}}{8\gamma\,(8Q)^{Q+2}},\,\,\,\,\,\,\,\,\,\,\,\varkappa_1\leq \frac{1}{2^{n_1}}\left[1-\left(\frac{1-\frac{3}{4}\alpha}{1-\frac{1}{2}\alpha}\right)^\frac{1}{2}\right],
    \end{align*}
    where $n_1\in \mathbb{N}$ is  chosen in such a way that $\frac{\gamma(\texttt{data},\alpha,\delta_1)}{\sqrt{n_1}}\leq \iota$ where $\iota=\iota(Q,\alpha)\in (0,1)$ comes from Lemma \ref{L5.1}. Additionally, note that $\delta_i$ and $\varkappa_i$ would be allowed to choose commonly as
    \begin{align*}
        \delta_i=\frac{1}{8\gamma.\,(8Q)^{Q+2}},\,\,\,\,\,\,\,\,\,\,\,\varkappa_i\leq\frac{1}{2^{n_2}}\left[1-\left(\frac{1}{8}\right)^\frac{1}{2}\right], 
    \end{align*}
    for all $i\in \{2,3,4,...,\hat{n}\}$ and $n_2$ satisfies $\frac{\gamma(\texttt{data},\alpha,\delta_1)}{\sqrt{n_2}}\leq \iota$. It is easy to see that, for all $i\in\{2,3,4,...,\hat{n}\}$, $\delta_1\leq \delta_i$ and $\varkappa_1\leq\varkappa_i$. Therefore redefining these parameters as,
    \begin{align*}
        \delta_i\equiv \delta_*:=\frac{\alpha^{Q+3}}{10^6\,(8Q)^{Q+2}},\,\,\,\,\,\,\,\,\,\,\,\varkappa_i\equiv\varkappa_*\leq \frac{1}{2^{n_0}}\left[1-\left(\frac{1-\frac{3}{4}\alpha}{1-\frac{1}{2}\alpha}\right)^\frac{1}{2}\right]\leq \frac{1}{8},
    \end{align*}
    for some large $n_0\in\mathbb{N}$. Next, recalling all the observations and defining recursively $\{t_i\}_{i=2}^{\hat{n}+1}$ by
    \begin{align*}
        t_i:=t_{i-1}+\delta_{\hat{n}-1}\left(2^{i-2}k\rho\right)^2,
    \end{align*}
    we finally conclude that if 
    \begin{align*}
        R^{-2}\int_{T_1}^{T_2}\textup{Tail}[u_-(t);R,\xi_0]\,dt\leq \frac{1}{2}\varkappa_*^{\hat{n}}\varepsilon M,
    \end{align*}
    then,
    \begin{align}\label{eq6.12}
        u(t)\geq \frac{1}{2}\varkappa_*^{\hat{n}}\varepsilon M\,\,\,\,\textup{a.e. in } B_{2^{\hat{n}}k\rho}(\xi_*)\times \left(t_{\hat{n}+1}(\xi_*)(2^{\hat{n}-1}k\rho)^2, t_{\hat{n}+1}\right].
    \end{align}
    By prior assumption $4\leq 2^{\hat{n}}k\leq 8$, 
    \begin{align*}
        \varkappa_*^{\hat{n}}=\left(2^{\hat{n}}\right)^{\log_2{\varkappa_*}}\begin{cases}
            >\left(\frac{k}{4}\right)^{\log_2(1/\varkappa_*)},\\
            \leq \left(\frac{k}{8}\right)^{\log_2(1/\varkappa_*)}.
        \end{cases}
    \end{align*}
    Define $d=d(\texttt{data},\alpha)$ and $\varkappa_0=\varkappa_0(\texttt{data},\alpha)\in (0,1)$ by 
    \begin{align*}
        d:=\log_\alpha\left[\left(\frac{k}{8}\right)^{\log_2(1/\varkappa_*)}\right],\,\,\,\,\,\,\varkappa_0:=\frac{\varepsilon}{2}.
    \end{align*}
    Then the result in (\ref{eq6.12}) along with the {\em tail} are estimated as 
    \begin{align*}
        \frac{1}{2}\varkappa_*^{\hat{n}}\varepsilon M\geq \varkappa_0\alpha^dM,
    \end{align*}
    and
    \begin{align*}
        R^{-2}\int_{T_1}^{T_2}\textup{Tail}[u_-(t);R,\xi_0]\,dt\leq \frac{1}{2}\varkappa_*^{\hat{n}}\varepsilon M&\leq \frac{1}{2}\left(\frac{k}{8}\right)^{\log_2(1/\varkappa_*)}\varepsilon M\\
        &=\varkappa_0\alpha^dM,
    \end{align*}
    respectively. As a result, we recall the fact $B_{2\rho}(\xi_0)\subset B_{2^{\hat{n}}k\rho}(\xi_*)$ and merge the previous results with display (\ref{eq6.12}) to have for every $t_{\hat{n}+1}-\frac{\delta_*}{2}\left(2^{\hat{n}-1}k\rho\right)^2<t\leq t_{\hat{n}+1}$,
    \begin{align}\label{eq6.13}
        u(t)\geq \varkappa_0\alpha^dM-R^{-2}\int_{T_1}^{T_2}\textup{Tail}[u_-(t);R,\xi_0]\,dt\,\,\,\,\,\,\textup{ a.e. in } B_{2\rho}(\xi_0). 
    \end{align}
    Here, note that time interval still depends on $\hat{n}$. So we need to find some subinterval that is independent of $\hat{n}$. For this purpose a direct computation can be done as follows
    \begin{align*}
        t_{\hat{n}+1}&=t_1+\delta_*\sum_{i=2}^{\hat{n}+1}\left(2^{i-2}k\rho\right)^2\\
        &=t_1+\delta_*\rho^2\frac{(2^{\hat{n}-1}k)^2-k^2}{2^2-1}.
    \end{align*}
    Without any loss of generality, by considering $k$ sufficiently small as required, we may assume that $\log_2(1/k)$ is an integer and choose $\hat{n}=3+\log_2(1/k)$. Then we have
    \begin{align*}
        t_{\hat{n}+1}=t_1+\left(\delta_*\frac{4^2-k^2}{3}\right)\rho^2,
    \end{align*}
    and
    \begin{align*}
        t_{\hat{n}+1}-\frac{\delta_*}{2}\left(2^{\hat{n}-1}k\rho\right)^2&=t_1+\delta_*\rho^2\left[\frac{4^2-k^2}{3}-\frac{1}{2}\left(2^{\hat{n}-1}k\right)^2\right]\\
        &=t_1++\delta_*\rho^2\left[\frac{4^2-k^2}{3}-\frac{1}{2}\left(4\right)^2\right]\\
        &<t_1+\delta_*\rho^2\frac{4^2-k^2}{3}.
    \end{align*}
    Take $\hat{\delta}_*:=\delta_*\frac{4^2-k^2}{3}$, which depends only on \texttt{data} and $\alpha$. Next, (\ref{eq6.13}) yields that for every $t_1+\frac{1}{2}\hat{\delta}_*\rho^2<t\leq t_1+\hat{\delta}_*\rho^2$ 
    \begin{align}\label{eq6.14}
        u(t)\geq \varkappa_0\alpha^dM-R^{-2}\int_{T_1}^{T_2}\textup{Tail}[u_-(t);R,\xi_0]\,dt\,\,\,\,\,\textup{ a.e. in } B_{2\rho}(\xi_0).
    \end{align}
    Moreover, note that the selection of $\delta_*$ assure that 
    \begin{align*}
        \hat{\delta}_*=\delta_*\frac{4^2-k^2}{3}=\frac{\alpha^{Q+3}}{10^6(8Q)^{Q+2}}\frac{4^2-k^2}{3}<1.
    \end{align*}
    \\[2mm]
    {\em Step IV}\vskip 1mm
    The result in (\ref{eq6.14}) is similar to our required result but in a different time interval. We need to shift the qualitative location of $t_1$ to the original time layer $t_0$ by the use of the result in Proposition \ref{P5.5}. From (\ref{eq6.14}), we can write
    \begin{align*}
        \left|\left\{u\left(t_1+\frac{1}{2}\hat{\delta}_*\rho^2\right)\geq\varkappa_0\alpha^dM\right\}\cap B_{2\rho}(\xi_0)\right|=|B_{2\rho}(\xi_0)|,
    \end{align*}
    therefore, from the result in Proposition \ref{P5.5} with $\alpha=1$ and for any $K>0$, there exists $\hat{\varkappa}_0\equiv\hat{\varkappa}_0(\texttt{data},K)\in (0,1)$ such that
    \begin{align*}
        u(t)\geq \hat{\varkappa}_0(\varkappa_0\alpha^dM)-R^{-2}\int_{T_1}^{T_2}\textup{Tail}[u_-(t);R,\xi_0]\,dt\,\,\,\,\textup{ a.e. in } B_{2\rho}(\xi_0),
    \end{align*}
    for every time $t_1+\frac{1}{2}\hat{\delta}_*\rho^2\leq t\leq t_1+\hat{\delta}_*\rho^2+K\rho^2$. Now, recalling the range of $t_1$ such that for all $\theta\in (1/2,1]$ such that $t_1=t_0+\theta\delta\rho^2$ and taking
    \begin{align*}
        K:=\theta\delta+\hat{\delta}_*=:\frac{\delta_0}{2},\,\,\,\,\tilde{\varkappa}:=\hat{\varkappa}_0\varkappa_0,
    \end{align*}
    we obtain our result. Also note that $\delta_0,\tilde{\varkappa}$ depend upon $\texttt{data}$ and $\alpha$.
 \end{proof}    
   \subsection{Proof of the Weak Harnack Inequality I} 
   Fix the parameters $\delta_0,\tilde{\varkappa}=\delta_0,\tilde{\varkappa}(\texttt{data},\alpha)$ and $d>1$ as in Proposition \ref{P6.2}. Set 
   \begin{align*}
       m:=\inf_{B_\rho(\xi_0)\times (t_0+\frac{1}{2}\delta_0\rho^2,t_0+\delta_0\rho^2]}u+R^{-2}\int_{T_1}^{T_2}\textup{Tail}[u_-(t);R,\xi_0].
   \end{align*}
    For some $0<q<\frac{1}{d}$ and $m>0$, we compute
    \begin{align}\label{eq6.15}
        \fint_{B_\rho(\xi_0)}u(t_0)^q\,d\xi&= q\int_0^\infty\frac{|\{u(t_0)\geq M\}\cap B_\rho(\xi_0)|}{|B_\rho|}M^{q-1}\, dM\nonumber\\
        &\leq q\int_0^m M^{q-1}\,dM+q \int_m^\infty\frac{|\{u(t_0)\geq M\}\cap B_\rho(\xi_0)|}{|B_\rho|}M^{q-1}\, dM\nonumber\\
        &= m^q+q \int_m^\infty\frac{|\{u(t_0)\geq M\}\cap B_\rho(\xi_0)|}{|B_\rho|}M^{q-1}\, dM.
    \end{align}
    Then with the help of the Proposition \ref{P6.2} with $\alpha:=\frac{|\{u(t_0)\geq L\}\cap B_\rho(\xi_0)|}{|B_\rho|}$, we have 
    \begin{align*}
        m\geq \tilde{\varkappa}M\left(\frac{|\{u(t_0)\geq M\}\cap B_\rho(\xi_0)|}{|B_\rho|}\right)^d,
    \end{align*}
    and the second integral on the right hand side of (\ref{eq6.15}) can be estimated as
    \begin{align*}
        q \int_m^\infty\frac{|\{u(t_0)\geq M\}\cap B_\rho(\xi_0)|}{|B_\rho|}M^{q-1}\, dM\leq \frac{q}{\tilde{\varkappa}^\frac{1}{d}}m^\frac{1}{d}\int_m^\infty M^{q-\frac{1}{d}-1}\,dM.
    \end{align*}
    Hence, (\ref{eq6.15}) can be rewritten as
    \begin{align*}
         \fint_{B_\rho(\xi_0)}u(t_0)^q\,d\xi&\leq m^q+\frac{q}{\tilde{\varkappa}^\frac{1}{d}}m^\frac{1}{d}\int_m^\infty M^{q-\frac{1}{d}-1}\,dM\\
         &=m^q+\frac{q}{\tilde{\varkappa}^\frac{1}{d}}m^\frac{1}{d}\left(\frac{1}{\frac{1}{d}-q}\,m^{q-\frac{1}{d}}\right)\\
         &=\left(1+\frac{q}{\tilde{\varkappa}^\frac{1}{d}\left(\frac{1}{d}-q\right)}\right)m^q\\
         &=\left(1+\frac{q}{\tilde{\varkappa}^\frac{1}{d}\left(\frac{1}{d}-q\right)}\right)\left[\inf_{B_\rho(\xi_0)\times (t_0+\frac{1}{2}\delta_0\rho^2,t_0+\delta_0\rho^2]}u+R^{-2}\int_{T_1}^{T_2}\textup{Tail}[u_-(t);R,\xi_0]\right]^q.
    \end{align*}
    In this series of calculations, we use the fact $0<q<\frac{1}{d}$ in the second line. Consequently, this leads to our result. 
    \subsection{Weak Harnack Inequality II} The next two lemmas concern measure theoretical information. Both use the {\em good term} to deduce the results.  
    \begin{lemma}
        Let $u$ be a local sub (super) solution of the problem (\ref{eq1.1}) in $\Omega_T$ and $u\geq 0$ in $\mathcal{Q}\Subset \Omega_T$. Let $M>0$ be an arbitrary number. Then either
        \begin{align*}
            R^{-2}\int_{T_1}^{T_2}\textup{Tail}[u_-(t);R,\xi_0]\,dt\leq M,
        \end{align*}
        or, there exists some constant $\gamma(\textup{\texttt{data}})$ such that
        \begin{align*}
            \inf_{t\in (t_0-\rho^2,t_0]}\left|\left\{u(\cdot,t)\leq M\right\}\cap B_\rho(\xi_0)\right|\leq \frac{\gamma M}{\left(\int_{z_0+\mathcal{Q}\rho^\ominus}\frac{u_+}{(2\rho)^{Q+2s}}\,d\xi dt\right)}|B_\rho(\xi_0)|.
        \end{align*}
    \end{lemma}
    \begin{proof}
        For simplicity, we assume $(\xi_0,t_0)=(0,0)$. Consider time independent cut-off function $\psi$ in $B_\rho$ such that
        \begin{align*}
            \psi\equiv\begin{cases}
                1\,\,\,\,\textup{ on } B_\rho,\\
                0\,\,\,\,\textup{ on } B_{\frac{3}{2}\rho},
            \end{cases}
            \,\,\,\,\textup{ such that }|\nabla_{\hn}\psi|\leq \frac{\gamma}{\rho},
        \end{align*}
        where $\gamma=\gamma(\texttt{data)}$. For any $M>0$, consider the truncated function $w_\pm:=(u-M)_\pm$ and use this in the energy estimate (\ref{eq3.6}) over the cylinder $\mathcal{Q}_\rho^\ominus$ where $\mathcal{Q}_\rho^\ominus=B_\rho\times(-\rho^2,0]$, we have
        \begin{align}\label{eq6.16}
            \int_{\mathcal{Q}_\rho^\ominus}&w_-(\xi,t)d\xi\left(\int_{B_\rho}\frac{w_+(\eta,t)}{\dist[\xi]{\eta}^{Q+2s}}d\eta\right)dt+\int_{\mathcal{Q}_\rho^\ominus}w_-(\xi,t)d\xi\left(\int_{\hn\setminus B_\rho}\frac{w_+(\eta,t)}{\dist[\xi]{\eta}^{Q+2s}}d\eta\right)dt\nonumber\\
            &\leq \gamma\iint_{\mathcal{Q}_\rho^\ominus}\left(w^2_-\nabla_{\hn}\psi\right)^2(\xi,t)\,d\xi dt+\gamma\int_{-\rho^2}^0\iint_{B_\rho\times B_\rho}\max\{w^2_-(\xi,t),w^2_-(\eta,t)\}\frac{|\psi(\xi)-\psi(\eta)|^2}{\dist[\xi]{\eta}^{Q+2s}}\,d\xi d\eta dt\nonumber\\
            &\quad+\gamma\int_{\mathcal{Q}_\rho^\ominus}w_-(\xi,t)d\xi\left(\int_{\hn\setminus B_\rho}\frac{w_-(\eta,t)}{\dist[\xi]{\eta}^{Q+2s}}\,d\eta\right)dt+\int_{B_\rho}w^2_-(\xi,-\rho^2)\, d\xi.
        \end{align}
        Estimating like its counterpart in Lemma \ref{L5.3}, we obtain that the right hand side of the previous inequality is dominated by 
        \begin{align*}
            \gamma M^2|B_\rho|,
        \end{align*}
        for some constant $\gamma=\gamma(\texttt{data})$. Recalling the non-negative property of $u$ in $\mathcal{Q}_\rho^\ominus$, then we have from the inequality (\ref{eq6.16}),
       \begin{align*}
            \gamma M^2|B_\rho|&\geq \int_{\mathcal{Q}_\rho^\ominus}w_-(\xi,t)d\xi\left(\int_{B_\rho}\frac{w_+(\eta,t)}{\dist[\xi]{\eta}^{Q+2s}}d\eta\right)dt\\
            &\geq \gamma\inf_{t\in (-\rho^2,0]}\int_{B_\rho}w_-(\xi,t)\,d\xi\left(\int_{\mathcal{Q}\rho^\ominus}\frac{u_+}{(2\rho)^{Q+2s}}\,d\xi dt-M\right)\\
            &=\gamma \left(\int_{\mathcal{Q}\rho^\ominus}\frac{u_+}{(2\rho)^{Q+2s}}\,d\xi dt\right)\inf_{t\in (-\rho^2,0]}\int_{\{u(\cdot,t)\leq M\}\cap B_\rho}w_-(\xi,t)\,d\xi-\gamma M^2|B_\rho|\\
            &\geq \gamma \left(\int_{\mathcal{Q}\rho^\ominus}\frac{u_+}{(2\rho)^{Q+2s}}\,d\xi dt\right) \inf_{t\in (-\rho^2,0]}\left|\left\{u(\cdot,t)\leq M\right\}\cap B_\rho\right|\,M-\gamma M^2|B_\rho|.
       \end{align*}
       In this sequence of calculations, we use the non-negativity of $\int_{\mathcal{Q}_\rho^\ominus}w_-(\xi,t)d\xi\left(\int_{\hn\setminus B_\rho}\frac{w_+(\eta,t)}{\dist[\xi]{\eta}^{Q+2s}}d\eta\right)dt$ in the first line and the second line follows by considering $\rho<1$ and  the simple observations 
       \begin{align*}
           u_+\leq \gamma\left((u-M)_++M\right)\,\,\,\textup{ and }\,\,\,\dist[\xi]{\eta}\leq 2\rho\,\,\, \textup{ for all }x,y\in B_\rho. 
       \end{align*}
       Finally, adjusting all the constants, we have our result from previous inequality. 
    \end{proof}
    \begin{lemma}\label{L6.4}
        Let $u$ be a local sub (super) solution of the problem (\ref{eq1.1}) in $\Omega_T$ and $u\geq 0$ in $\mathcal{Q}\Subset \Omega_T$. Let $M>0$ be an arbitrary number. Then either
        \begin{align*}
            R^{-2}\int_{T_1}^{T_2}\textup{Tail}[u_-(t);R,\xi_0]\,dt\leq M,
        \end{align*}
        or there exists some constant $\gamma(\textup{\texttt{data}})$ such that
        \begin{align*}
            \inf_{t\in (t_0-\rho^2,t_0]}\left|\left\{u(\cdot,t)\leq M\right\}\cap B_\rho(\xi_0)\right|\leq \frac{\gamma M}{\fint_{-\rho^2}^0\textup{Tail}[u_+(t);\rho,0]\, dt}|B_\rho(\xi_0)|.
        \end{align*}
    \end{lemma}
    \begin{proof}
        Without loss of any generality, consider $(\xi_0,t_0)=(0,0)$. We mainly use the inequality (\ref{eq6.16}) with the property that $\int_{\mathcal{Q}_\rho^\ominus}w_-(\xi,t)d\xi\left(\int_{B_\rho}\frac{w_+(\eta,t)}{\dist[\xi]{\eta}^{Q+2s}}d\eta\right)dt\geq 0$. Since $\dist[\xi]{\eta}\leq 2|\eta|_{\hn}$ for all $\xi\in B_\rho$ and $\eta\in \hn\setminus B_\rho$,upon using the estimate $u_+\leq \gamma((u-M)_++M)$, we deduce that
        \begin{align*}
            \gamma M^2|B_\rho|&\geq \int_{\mathcal{Q}_\rho^\ominus}w_-(\xi,t)d\xi\left(\int_{\hn\setminus B_\rho}\frac{w_+(\eta,t)}{\dist[\xi]{\eta}^{Q+2s}}d\eta\right)dt\\
            &\geq \frac{\gamma(\texttt{data})}{2^{Q+2s}} \inf_{t\in(-\rho^2,0]}\int_{B_\rho}w_-(\xi,t)\,d\xi\left(\int_{-\rho^2}^0\int_{\hn\setminus B_\rho}\frac{u_+(\eta,t)}{|\eta|^{Q+2s}}d\eta dt-M\right)\\
            & = \frac{\gamma(\texttt{data})}{2^{Q+2s}}\left[\inf_{t\in (-\rho^2,0]}\left|\left\{u(\cdot,t)\leq M\right\}\cap B_\rho\right|\, M\left(\fint_{-\rho^2}^0\textup{Tail}[u_+(t);\rho,0]\, dt\right)-M^2|B_\rho|\right].
        \end{align*}
        Consequently, adjusting all the relevant constants, we have
       \begin{align*}
            \inf_{t\in (-\rho^2,0]}\left|\left\{u(\cdot,t)\leq \frac{M}{2}\right\}\cap B_\rho\right|\leq \frac{\gamma M}{\fint_{-\rho^2}^0\textup{Tail}[u_+(t);\rho,0]\, dt}|B_\rho|.
       \end{align*}
    \end{proof}
    The proof of the following lemma is similar to the De Giorgi Lemma \ref{L5.1}. But the main difference with Lemma \ref{L5.1} is that this lemma is proven with weaker conditions as compare to the Lemma \ref{L5.1}. This lemma is proven by considering measure theoretical information in a spatial ball at a time slice whereas Lemma \ref{L5.1} relies on measure theoretical information on a parabolic cylinder. 
\begin{lemma}\label{L6.5}
    Let $(\xi_0,t_0)\in \Omega_T$ and $u$ be a locally bounded weak solution of the problem (\ref{eq1.1}) in $\Omega_T$. Suppose $B_{2\rho}(\xi_0)\times (t_0,t_0+\delta\rho^2]\subseteq\mathcal{Q}\Subset \Omega_T$. Let $0<\rho<\frac{R}{2},\lambda>0$, $\delta\in (0,1)$ and there exists $\iota_*\in (0,1)$, depending only on \textup{\texttt{data}} and $\delta$, such that, if 
    \begin{align*}
        \left|\{u(\cdot,t)<\lambda\}\cap B_{2\rho}(\xi_0)\right|\leq\iota_*|B_{2\rho}(\xi_0)|.
    \end{align*}
    Then either
    \begin{align*}
        R^{-2}\int_{T_1}^{T_2}\textup{Tail}[u_-(t);R,\xi_0]\,dt>\frac{\lambda}{8},
    \end{align*}
    or
    \begin{align*}
        u\geq \frac{\lambda}{8}\,\,\,\,\,\textup{a.e. in }\, B_\frac{\rho}{2}(\xi_0)\times \left(t_0+\frac{1}{2}\delta\rho^2,t_0+\delta\rho^2\right].
    \end{align*}
\end{lemma}
\begin{proof}
    This lemma mainly follow the steps of the De Giorgi Lemma \ref{L5.1}. For simplicity $(\xi_0,t_0)=(0,0)$. Let us introduce $\rho_n,\tilde{\rho}_n,B_n,\tilde{B}_n,\mathcal{Q}_n^\oplus,\tilde{\mathcal{Q}}_n^\oplus,\lambda_n$ as follows
    \begin{align*}
        \begin{cases}
            \rho_n=\rho+\frac{\rho}{2^{n+1}},\,\tilde{\rho}_n=\frac{\rho_n+\rho_{n+1}}{2},\\
            B_n=B_{\rho_n}(0),\,\tilde{B}_n=B_{\tilde{\rho}_n}(0),\\
            \mathcal{Q}_n^\oplus=B_n\times(0,\delta\rho_n^2],\,\tilde{\mathcal{Q}}_n^\oplus=\tilde{B}_n\times(0,\delta\tilde{\rho}_n^2],\\
            \lambda_n=\lambda+\frac{\lambda}{2^{n+1}}.
        \end{cases}
    \end{align*}
    Unlike Lemma \ref{L5.1}, we consider time independent cut-off function $\psi$ as follows
    \begin{align*}
        \psi\equiv\begin{cases}
            1\,\,\,\,\textup{ on } B_{n+1},\\
            0\,\,\,\,\textup{ on } B^c_n,
        \end{cases}
        \,\,\,\,\textup{ such that }|\nabla_{\hn}\psi|\leq \gamma\frac{2^{n+4}}{\rho}.
    \end{align*}
   Set truncated function
   \begin{align*}
       w_-:=\left(u+g(t)-\lambda_n\right)_-\,\,\,\,\textup{ where } \,g(t)=\mathfrak{C}\int_{R^2}^t\int_{\hn\setminus B_R}\frac{|u(\eta,t)|}{|\eta|^{Q+2s}_{\hn}}\,d\eta dt,
   \end{align*}
   where positive constant $\mathfrak{C}$ to be chosen later. Moreover, we define
   \begin{align*}
       &\mathcal{A}_n:=\left\{u(\xi,t)+g(t)<\lambda_n\right\}\cap\mathcal{Q}_n^\oplus,\\
       &\mathcal{A}_n(t):=\left\{u(\xi,t)+g(t)<\lambda_n\right\}\cap B_n\,\,\,\,\textup{ for }\, t\in (0,\delta\rho^2].
   \end{align*}
   Considering all these entities in (\ref{eq3.6}), we have
   \begin{align}\label{eq6.17}
       \esssup_{0<t\leq \delta\rho_{n+1}^2}&\int_{B_{n+1}}w^2_-(\xi,t)\,d\xi+\iint_{\mathcal{Q}_{n+1}}|\nabla_{\hn}w_-|^2\,d\xi dt\nonumber\\
       &\leq \gamma\int_0^{\delta\rho^2}\iint_{B_n\times B_n}\max\{w^2_-(\xi,t),w^2_-(\eta,t)\}\frac{|\psi(\xi)-\psi(\eta)|}{\dist[\xi]{\eta}^{Q+2s}}\,d\xi d\eta dt\nonumber\\
       &\quad+\gamma \iint_{\mathcal{Q}_n^\oplus}(w_-\nabla_{\hn}\psi\nu)^2\,d\xi dt+\iint_{\mathcal{Q}_n^\oplus}(w_-\psi^2\nu^2)\,d\xi\left\{\frac{w_-(\eta,t)}{\dist[\xi]{\eta}^{Q+2s}}\,d\eta\right\}dt\nonumber\\
       &\quad\quad-\mathfrak{C}\iint_{\mathcal{Q}_n^\oplus}g'(t)(w_-\psi^2\nu^2)(\xi,t)\,d\xi dt.
   \end{align}

   Next, considering the respective calculations in Lemma \ref{L5.2} with suitable choice of $\mathfrak{C}$ and $\gamma$, we conclude that the left hand side of the previous inequality is dominated by
   \begin{align*}
       \gamma\frac{2^{(Q+2s+2)n}}{\delta\rho^2}\lambda^2|\mathcal{A}_n|.
   \end{align*}
   After that by similar calculations as in Lemma \ref{L5.1} with the forward in time interval, we have the recursive inequality
   \begin{align}
       Y_{n+1}\leq \gamma (2b)^n\delta^{-\left[\frac{1}{\kappa}+\frac{2}{q\kappa}\right]}Y_n,
   \end{align}
   where $b=b(Q),\,\gamma=\gamma(\texttt{data}),\,\kappa=\frac{Q}{Q-2}$ and $Y_n=\frac{|\mathcal{A}_n|}{|\mathcal{Q}_n^\oplus|}.$\\
   At this stage, we consider two possibilities.\vskip 1mm
   {\bf Case I: } For some $n_0\in\mathbb{N}$, let 
   \begin{align*}
       \frac{|\mathcal{A}_{n_0}|}{|\mathcal{Q}_{n_0}^\oplus|}\leq \iota_*.
   \end{align*}
 This yields the following
 \begin{align*}
     &\left|\{u+g(t)\leq \lambda_{n_0}\}\cap \mathcal{Q}_{n_0}^\oplus\right|\leq \iota_*|\mathcal{Q}_{n_0}^\oplus|\\
     \implies&\left|\{u+g(t)\leq \lambda\}\cap (0,\delta\rho^2)+\mathcal{Q}_\rho^\ominus(\delta)\right|\leq \iota_*2^Q|(0,\delta\rho^2)+\mathcal{Q}_\rho^\ominus(\delta)|.
 \end{align*}
 We get the last line because $\rho< \rho_{n_0}\leq 2$ and $\lambda\leq \lambda_{n_0}$. We choose $\iota_*$ sufficiently small such that
 \begin{align}\label{eq6.19}
     \iota_*\leq 2^{-Q}\iota,
 \end{align}
 where $\iota$ is chosen as in Lemma \ref{L5.1}.  Then we have
 \begin{align}\label{eq6.20}
     \left|\{u+g(t)\leq \lambda\}\cap (0,\delta\rho^2)+\mathcal{Q}_\rho^\ominus(\delta)\right|\leq \iota|(0,\delta\rho^2)+\mathcal{Q}_\rho^\ominus(\delta)|.
 \end{align}
 At this point, by {\em tail} condition and (\ref{eq6.20}), we are allowed to apply Lemma \ref{L5.1} with the center $(0,\delta\rho^2)$ such that 
 \begin{align*}
     u+g(t)\geq \frac{\lambda}{4}\,\,\,\,\textup{a.e. in } B_\frac{\rho}{2}\times \left(\frac{3}{4}\delta\rho^2,\delta\rho^2\right].
 \end{align*}
 Considering the {\em tail} condition with $g(T_2)\leq \frac{1}{8}\lambda$, we have
 \begin{align}\label{eq6.21}
     u\geq \frac{\lambda}{8}\,\,\,\,\textup{a.e. in } B_\frac{\rho}{2}\times \left(\frac{3}{4}\delta\rho^2,\delta\rho^2\right].
 \end{align}
 \vskip 1mm
 {\bf Case II: } If for any $n\in\mathbb{N}$,
 \begin{align*}
     \frac{|\mathcal{A}_n|}{|\mathcal{Q}_n^\oplus|}> \iota_*.
 \end{align*}
 Then applying convergence result (Lemma \ref{L2.4}), we have $\iota\in (0,1)$ as in Lemma \ref{L5.1} such that 
 \begin{align*}
     Y_0\leq \iota,
 \end{align*}
 and $Y_n\rightarrow 0$ as $n\rightarrow 0$. Therefore, 
 \begin{align*}
     u+g(t)\geq \lambda\,\,\,\,\textup{a.e. in } B_\frac{\rho}{2}\times \left(0,\delta\rho^2\right].
 \end{align*}
 Considering the {\em tail} condition with $g(T_2)\leq \frac{1}{8}\lambda$, we have
 \begin{align}\label{eq6.22}
     u\geq \frac{7\lambda}{8}\,\,\,\,\textup{a.e. in } B_\frac{\rho}{2}\times \left(0,\delta\rho^2\right]
 \end{align}
Finally, combining both (\ref{eq6.21}) and (\ref{eq6.22}), we have our result. 
\end{proof}
\subsection{Proof of the Weak Harnack Inequality II}
\begin{proof}
Without loss of generality, let us consider $(\xi_0,t_0)=(0,0)$. Recalling the Lemma \ref{L6.4} with redefining $M$, such that
\begin{align*}
    \inf_{(-(2\rho)^2,0]}\left|\{u(\cdot,t)\leq M\}\cap B_{2\rho}\right|\leq \frac{\gamma(\textup{\texttt{data}}) M}{\fint_{-{(2\rho)}^2}^0\textup{Tail}[u_+(t);2\rho,0]\, dt}|B_{2\rho}|,
\end{align*}
provided that 
\begin{align}\label{eq6.23}
   R^{-2}\int_{T_1}^{T_2}\textup{Tail}[u_-(t);R,0]\, dt\leq M.
\end{align}
Take $\iota\in (0,1)$ as in Lemma \ref{L5.1}. Then choose $M>0$ in such a way that 
\begin{align*}
    &\frac{\gamma M}{\fint_{-{(2\rho)}^2}^0\textup{Tail}[u_+(t);2\rho,0]\, dt}\leq \iota\\
    \implies&\frac{\gamma M}{\left[\left(\fiint_{\mathcal{Q}_{2\rho}^\ominus}u_+^2\,d\xi  dt\right)^\frac{1}{2}+\left(\fint_{-{(2\rho)}^2}^{0}\left(\fint_{B_{2\rho}}u_+(\xi,t)\,d\xi\right)^2\,dt\right)^\frac{1}{2}+\fint_{-{(2\rho)}^2}^0\textup{Tail}[u_+(t);2\rho,0]\, dt+1\right]}\leq \iota.
\end{align*}
Set $M$ as 
\begin{align*}
    M=\frac{\iota}{4\gamma}\left[\left(\fiint_{\mathcal{Q}_{2\rho}^\ominus}u_+^2\,d\xi  dt\right)^\frac{1}{2}+\left(\fint_{-{(2\rho)}^2}^{0}\left(\fint_{B_{2\rho}}u_+(\xi,t)\,d\xi\right)^2\,dt\right)^\frac{1}{2}+\fint_{-{(2\rho)}^2}^0\textup{Tail}[u_+(t);2\rho,0]\, dt+1\right].
\end{align*}
Due to choice of this $M$, there exists $t^*\in (-(2\rho)^2,0]$ such that
\begin{align*}
    \inf_{(-(2\rho)^2,0]}\left|\{u(\cdot,t)\leq M\}\cap B_{2\rho}\right|\leq \iota|B_{2\rho}|.
\end{align*}
Therefore, applying the Lemma \ref{L6.5}, we have 
\begin{align}\label{eq6.24}
    u\geq \frac{1}{8}M\,\,\,\,\textup{ a.e in }\, B_{\frac{\rho}{2}}\times (t^*+\frac{3}{4}\delta\rho^2,t^*+\delta\rho^2],
\end{align}
where $\delta\in (0,1)$ depends upon \texttt{data} and $\alpha$ with $B_{2\rho}\times\left(-(2\rho)^2,\delta\rho^2\right]\subseteq \mathcal{Q}\Subset \Omega_T $ and 
\begin{align}\label{eq6.25}
     R^{-2}\int_{T_1}^{T_2}\textup{Tail}[u_-(t);R,0]\, dt\leq \frac{M}{8}.
\end{align}
Let, $\bar{t}\in (t^*+\frac{3}{4}\delta\rho^2,t^*+\delta\rho^2] $, then from (\ref{eq6.24}), we have 
\begin{align*}
    \left|\{u(\cdot,\bar{t})\geq \frac{1}{4} M\}\cap B_{2\rho}\right|=|B_{2\rho}|.
\end{align*}
Since (\ref{eq6.25}) automatically implies (\ref{eq6.23}), so we can proceed with the {\em tail} condition in (\ref{eq6.25}). Next, we can use Proposition \ref{P6.2}, i.e., the power-like expansion of positivity with $\alpha=1$. Thus, we have some $\tilde{\varkappa}\in (0,1)$ and $\delta_0$ depending only on \texttt{data}, such that
\begin{align}
    u\geq \tilde{\varkappa} M\,\,\,\,\textup{ a.e. in } B_\rho\times (t^*+\frac{3}{4}\delta\rho^2+\frac{1}{2}\delta_0\rho^2,t^*+\delta\rho^2+\delta_0\rho^2+K\rho^2],
\end{align}
where $K>0$ (see in Step IV of Proposition \ref{P6.2}) will be chosen later and provided
\begin{align}
     R^{-2}\int_{T_1}^{T_2}\textup{Tail}[u_-(t);R,0]\, dt\leq \tilde{\varkappa}M.
\end{align}
Take minimum between $\frac{1}{8}$ and $\tilde{\varkappa}$ and proceed with this {\em tail} condition. Take $K=32$, then for any $t^*\in(-(2\rho)^2,0] $, it is obvious to have
\begin{align*}
    \left(\frac{5}{4}\rho^2,(4\rho)^2\right]\subseteq (t^*+\frac{3}{4}\delta\rho^2+\frac{1}{2}\delta_0\rho^2,t^*+\delta\rho^2+\delta_0\rho^2+K\rho^2],
\end{align*}
provided $B_{2\rho}\times\left(-(2\rho)^2, 2(4\rho)^2\right]\subseteq\mathcal{Q}\Subset \Omega_T$. Therefore, due to the choice of $K$, we can conclude that 
\begin{align*}
    u\geq \frac{\iota\tilde{\varkappa}}{2\gamma}&\Bigg[\left(\fiint_{\mathcal{Q}_{2\rho}^\ominus}u_+^2\,d\xi  dt\right)^\frac{1}{2}+\left(\fint_{-{(2\rho)}^2}^{0}\left(\fint_{B_{2\rho}}u_+(\xi,t)\,d\xi\right)^2\,dt\right)^\frac{1}{2}\\&\quad+\fint_{-{(2\rho)}^2}^0\textup{Tail}[u_+(t);2\rho,0]\, dt+1\Bigg]
     - R^{-2}\int_{T_1}^{T_2}\textup{Tail}[u_-(t);R,0]\, dt,
\end{align*}
for a.e. in $B_\rho\times \left(\frac{5}{4}\rho^2,(4\rho)^2\right]$, provided that
\begin{align*}
    B_{2\rho}\times\left(-(2\rho)^2, 2(4\rho)^2\right]\subseteq\mathcal{Q}\Subset \Omega_T.
\end{align*}
Hence, we can conclude the result by taking $\varkappa=\frac{\iota\tilde{\varkappa}}{2\gamma}$. 
\end{proof}


\begin{thebibliography}{KMMJP12}

\bibitem[APT25]{APT25_38}
Karthik Adimurthi, Harsh Prasad, and Vivek Tewary.
\newblock Local {Hölder} regularity for nonlocal parabolic p-{Laplace}
  equations.
\newblock {\em Annali Scuola Normale Superiore - Classe Di Scienze}, page~18,
  May 2025.
\newblock URL:
  \url{https://journals.sns.it/index.php/annaliscienze/issue/view/forthcoming_articles},
  \href {https://doi.org/10.2422/2036-2145.202401_012}
  {\path{doi:10.2422/2036-2145.202401_012}}.

\bibitem[BdCN13]{BdCN13_29}
Daniel Blazevski and Diego del Castillo-Negrete.
\newblock Local and nonlocal anisotropic transport in reversed shear magnetic
  fields: {S}hearless {C}antori and nondiffusive transport.
\newblock {\em Physical Review E—Statistical, Nonlinear, and Soft Matter
  Physics}, 87(6):063106, 2013.
\newblock URL: \url{https://doi.org/10.1103/PhysRevE.87.063106}.

\bibitem[BFS18]{BFS18_28}
Zolt\'an~M. Balogh, Katrin F\"assler, and Hernando Sobrino.
\newblock Isometric embeddings into {H}eisenberg groups.
\newblock {\em Geom. Dedicata}, 195:163--192, 2018.
\newblock \href {https://doi.org/10.1007/s10711-017-0282-5}
  {\path{doi:10.1007/s10711-017-0282-5}}.

\bibitem[BGKL26]{BGKL26_52}
Souvik Bhowmick, Sekhar Ghosh, Vishvesh Kumar and R Lakshmi.
\newblock Harnack inequality for superposition operators of mixed fractional order.
\newblock {\em arXiv preprint arXiv:2606.01449}, 2026. 
\newblock URL: \url{https://doi.org/10.48550/arXiv.2606.01449}

\bibitem[BLS21]{BLS21_7}
Lorenzo Brasco, Erik Lindgren, and Martin Str\"omqvist.
\newblock Continuity of solutions to a nonlinear fractional diffusion equation.
\newblock {\em J. Evol. Equ.}, 21(4):4319--4381, 2021.
\newblock \href {https://doi.org/10.1007/s00028-021-00721-2}
  {\path{doi:10.1007/s00028-021-00721-2}}.

\bibitem[BLU07]{BLU07_25}
A.~Bonfiglioli, E.~Lanconelli, and F.~Uguzzoni.
\newblock {\em Stratified {L}ie groups and potential theory for their
  sub-{L}aplacians}.
\newblock Springer Monographs in Mathematics. Springer, Berlin, 2007.

\bibitem[CKSV12]{CKSV12_46}
Zhen-Qing Chen, Panki Kim, Renming Song, and Zoran Vondra{\v{c}}ek.
\newblock Boundary {H}arnack principle for {$\Delta$}+ {$\Delta^{\alpha/2}$}.
\newblock {\em Transactions of the American Mathematical Society},
  364(8):4169--4205, 2012.

\bibitem[CN25]{CN25_2}
Simone Ciani and Kenta Nakamura.
\newblock Nonlocal parabolic {D}e {G}iorgi classes.
\newblock {\em arXiv preprint arXiv:2508.16247}, 2025.
\newblock URL: \url{https://doi.org/10.48550/arXiv.2508.16247}.

\bibitem[CTR24]{CTR24_18}
Debajyoti Choudhuri, Leandro~S Tavares, and Du{\v{s}}an~D Repov{\v{s}}.
\newblock A multiphase eigenvalue problem on a stratified {L}ie group.
\newblock {\em Rendiconti del Circolo Matematico di Palermo Series 2},
  73(7):2533--2546, 2024.
\newblock URL: \url{https://doi.org/10.1007/s12215-024-01035-1}.

\bibitem[DB88]{DB88_34}
Emmanuele Di~Benedetto.
\newblock Intrinsic {H}arnack type inequalities for solutions of certain
  degenerate parabolic equations.
\newblock {\em Archive for Rational Mechanics and Analysis}, 100(2):129--147,
  1988.
\newblock URL: \url{https://doi.org/10.1007/BF00282201}.

\bibitem[DBGV08]{DBGV08_36}
Emmanuele Di~Benedetto, Ugo Gianazza, and Vincenzo Vespri.
\newblock Harnack estimates for quasi-linear degenerate parabolic differential
  equations.
\newblock {\em Acta Math}, 200:181--209, 2008.
\newblock URL: \url{https://doi.org/10.1007/s11511-008-0026-3}.

\bibitem[DCKP14]{DCKP14_35}
Agnese Di~Castro, Tuomo Kuusi, and Giampiero Palatucci.
\newblock Nonlocal {H}arnack inequalities.
\newblock {\em Journal of Functional Analysis}, 267(6):1807--1836, 2014.
\newblock URL: \url{https://doi.org/10.1016/j.jfa.2014.05.023}.

\bibitem[DGV06]{DBGV06_22}
Emmanuele DiBenedetto, Ugo Gianazza, and Vincenzo Vespri.
\newblock Local clustering of the non-zero set of functions in {$W^{1, 1}(E)$}.
\newblock {\em Atti Della Accademia Nazionale Dei Lincei. Rendiconti Lincei.
  Matematica E Applicazioni}, 17:223--225, 2006.
\newblock URL: \url{https://doi.org/10.1007/bf00382749}.

\bibitem[DiB93]{DiB93_10}
Emmanuele DiBenedetto.
\newblock {\em Degenerate parabolic equations}.
\newblock Universitext. Springer-Verlag, New York, 1993.
\newblock \href {https://doi.org/10.1007/978-1-4612-0895-2}
  {\path{doi:10.1007/978-1-4612-0895-2}}.

\bibitem[DIV25]{DIV23_23}
Fatma~Gamza D{\"u}zg{\"u}n, Antonio Iannizzotto, and Vincenzo Vespri.
\newblock A clustering theorem in fractional {S}obolev spaces.
\newblock {\em Annales Fennici Mathematici}, 50:243--252, 2025.
\newblock URL: \url{https://doi.org/10.54330/afm.161328}.

\bibitem[DLV22]{DLV22_30}
Serena Dipierro, Edoardo~Proietti Lippi, and Enrico Valdinoci.
\newblock (non)local logistic equations with {N}eumann conditions.
\newblock {\em Annales de l'Institut Henri Poincar{\'e} C}, 40(5):1093--1166,
  2022.
\newblock URL: \url{https://doi.org/10.4171/aihpc/57}.

\bibitem[Ego11]{Ego11_20}
D.~Egorov.
\newblock Sur les suites des fonctions meaurables.
\newblock {\em Comptes Rendus de Acad. des Sc. de Paris}, 152:244--246, 1911.

\bibitem[FK13]{FK13_42}
Matthieu Felsinger and Moritz Kassmann.
\newblock Local regularity for parabolic nonlocal operators.
\newblock {\em Communications in Partial Differential Equations},
  38(9):1539--1573, 2013.
\newblock URL: \url{https://doi.org/10.1080/03605302.2013.808211}.

\bibitem[Foo09]{Foo09_45}
Mohammud Foondun.
\newblock Heat kernel estimates and {H}arnack inequalities for some {D}irichlet
  forms with non-local part.
\newblock {\em Electron. J. Probab.}, 14(11):314--340, 2009.
\newblock URL: \url{http://www.math.washington.edu/~ejpecp/}.

\bibitem[FSZ22]{FSZ22_51}
Yuzhou Fang, Bin Shang, and Chao Zhang.
\newblock Regularity theory for mixed local and nonlocal parabolic p-laplace
  equations.
\newblock {\em The Journal of Geometric Analysis}, 32(1):22, 2022.

\bibitem[Gar25]{Gar25_19}
Prashanta Garain.
\newblock On mixed local-nonlocal {S}obolev-type inequalities and their
  connection with singular equations in the {H}eisenberg group.
\newblock {\em arXiv preprint arXiv:2512.11449}, 2025.
\newblock URL: \url{https://doi.org/10.48550/arXiv.2512.11449}.

\bibitem[GK22]{GK22_44}
Prashanta Garain and Juha Kinnunen.
\newblock On the regularity theory for mixed local and nonlocal quasilinear
  elliptic equations.
\newblock {\em Transactions of the American Mathematical Society},
  375(08):5393--5423, 2022.
\newblock URL: \url{https://doi.org/10.1090/tran/8621}.

\bibitem[HaK00]{HaK00_26}
Piotr Haj\l~asz and Pekka Koskela.
\newblock Sobolev met {P}oincar\'e.
\newblock {\em Mem. Amer. Math. Soc.}, 145(688):x+101, 2000.
\newblock \href {https://doi.org/10.1090/memo/0688}
  {\path{doi:10.1090/memo/0688}}.

\bibitem[HKST15]{HKST15_27}
Juha Heinonen, Pekka Koskela, Nageswari Shanmugalingam, and Jeremy~T. Tyson.
\newblock {\em Sobolev spaces on metric measure spaces}, volume~27 of {\em New
  Mathematical Monographs}.
\newblock Cambridge University Press, Cambridge, 2015.
\newblock An approach based on upper gradients.
\newblock \href {https://doi.org/10.1017/CBO9781316135914}
  {\path{doi:10.1017/CBO9781316135914}}.

\bibitem[Jer86]{Jer86_15}
David Jerison.
\newblock The {P}oincar{\'e} inequality for vector fields satisfying
  {H}{\"o}rmander’s condition.
\newblock {\em Duke Math. J.}, 53(2):503--523, 1986.
\newblock \href {https://doi.org/10.1215/S0012-7094-86-05329-9}
  {\path{doi:10.1215/S0012-7094-86-05329-9}}.

\bibitem[Kar26]{Kar26_9}
Debraj Kar.
\newblock Optimal in tail {H}\"older estimates for weak solutions of the
  nonlocal parabolic p-{L}aplace equations on the {H}eisenberg group.
\newblock 2026.
\newblock URL: \url{https://doi.org/10.48550/arXiv.2602.10612}.

\bibitem[KMMJP12]{KMMJP21_24}
Juha Kinnunen, Niko Marola, Michele Miranda~Jr, and Fabio Paronetto.
\newblock Harnack's inequality for parabolic {De} {Giorgi} classes in metric
  spaces.
\newblock {\em Adv. Differ. Equ.}, 17:801--832, 2012.
\newblock \href {https://doi.org/10.57262/ade/1355702923}
  {\path{doi:10.57262/ade/1355702923}}.

\bibitem[KS13]{KS13_39}
Moritz Kassmann and Russell~W Schwab.
\newblock Regularity results for nonlocal parabolic equations.
\newblock {\em arXiv preprint arXiv:1305.5418}, 2013.
\newblock URL: \url{https://doi.org/10.48550/arXiv.1305.5418}.

\bibitem[KT26]{KT25_6}
Debraj Kar and Vivek Tewary.
\newblock Nonlocal quasilinear parabolic equations in {H}eisenberg group:
  {L}ocal boundedness with an optimal tail.
\newblock {\em Discrete and Continuous Dynamical Systems}, 51:46--95, 2026.
\newblock \href {https://doi.org/10.3934/dcds.2026021}
  {\path{doi:10.3934/dcds.2026021}}.

\bibitem[KW24]{KW24_40}
Moritz Kassmann and Marvin Weidner.
\newblock The parabolic {H}arnack inequality for nonlocal equations.
\newblock {\em Duke Mathematical Journal}, 173(17):3413--3451, 2024.
\newblock URL: \url{https://doi.org/10.1215/00127094-2024-0008}.

\bibitem[Lia21]{Lia21_14}
Naian Liao.
\newblock Remarks on parabolic {D}e {G}iorgi classes.
\newblock {\em Annali di Matematica Pura ed Applicata (1923-)},
  200(6):2361--2384, 2021.
\newblock \href {https://doi.org/10.1007/s10231-021-01084-8}
  {\path{doi:10.1007/s10231-021-01084-8}}.

\bibitem[Lu94]{Lu94_16}
Guozhen Lu.
\newblock The sharp {P}oincar{\'e} inequality for free vector fields: an
  endpoint result.
\newblock {\em Revista Matem{\'a}tica Iberoamericana}, 10(2):453--466, 1994.
\newblock URL: \url{https://doi.org/10.4171/rmi/158}.

\bibitem[LW25]{LW25_49}
Naian Liao and Marvin Weidner.
\newblock Time-insensitive nonlocal parabolic {H}arnack estimates.
\newblock {\em Proceedings of the London Mathematical Society}, 130(5):e70051,
  2025.
\newblock URL: \url{https://doi.org/10.1112/plms.70051}.

\bibitem[Mos64]{Mos64_32}
J{\"u}rgen Moser.
\newblock A {H}arnack inequality for parabolic differential equations.
\newblock {\em Communications on pure and applied mathematics}, 17(1):101--134,
  1964.
\newblock URL: \url{https://doi.org/10.1002/cpa.3160170106}.

\bibitem[MPPP23]{MPPP23_1}
Maria Manfredini, Giampiero Palatucci, Mirco Piccinini, and Sergio Polidoro.
\newblock H\"older continuity and boundedness estimates for nonlinear
  fractional equations in the {H}eisenberg group.
\newblock {\em J. Geom. Anal.}, 33(3):Paper No. 77, 41, 2023.
\newblock \href {https://doi.org/10.1007/s12220-022-01124-6}
  {\path{doi:10.1007/s12220-022-01124-6}}.

\bibitem[MS79]{MS79_13}
Roberto~A. Mac\'ias and Carlos Segovia.
\newblock A decomposition into atoms of distributions on spaces of homogeneous
  type.
\newblock {\em Adv. in Math.}, 33(3):271--309, 1979.
\newblock \href {https://doi.org/10.1016/0001-8708(79)90013-6}
  {\path{doi:10.1016/0001-8708(79)90013-6}}.

\bibitem[MZ21]{MZ21_17}
Shirsho Mukherjee and Xiao Zhong.
\newblock {$C^{1,\alpha}$}-regularity for variational problems in the
  {H}eisenberg group.
\newblock {\em Anal. PDE}, 14(2):567--594, 2021.
\newblock \href {https://doi.org/10.2140/apde.2021.14.567}
  {\path{doi:10.2140/apde.2021.14.567}}.

\bibitem[Nak23a]{Nak23_48}
Kenta Nakamura.
\newblock {H}arnack’s estimate for a mixed local--nonlocal doubly nonlinear
  parabolic equation.
\newblock {\em Calculus of Variations and Partial Differential Equations},
  62(2):40, 2023.
\newblock URL: \url{https://doi.org/10.1007/s00526-022-02378-2}.

\bibitem[Nak23b]{Nak23_3}
Kenta Nakamura.
\newblock Local properties of fractional parabolic {D}e {G}iorgi classes of
  order {$s, p$}.
\newblock {\em J. Funct. Anal.}, 285(7):Paper No. 110049, 68, 2023.
\newblock \href {https://doi.org/10.1016/j.jfa.2023.110049}
  {\path{doi:10.1016/j.jfa.2023.110049}}.

\bibitem[Nhi96]{Nhi96_50}
Duy-Minh Nhieu.
\newblock {\em The extension problem for Sobolev spaces on the Heisenberg
  group}.
\newblock Purdue University, 1996.
\newblock URL:
  \url{http://gateway.proquest.com/openurl?url_ver=Z39.88-2004&rft_val_fmt=info:ofi/fmt:kev:mtx:dissertation&res_dat=xri:pqdiss&rft_dat=xri:pqdiss:9713571}.

\bibitem[Pra23]{Pra23_43}
Harsh Prasad.
\newblock On the weak {H}arnack estimate for nonlocal equations.
\newblock {\em Calc. Var.}, 63, 2023.
\newblock URL: \url{https://doi.org/10.1007/s00526-024-02670-3}.

\bibitem[Sev10]{Sev10_21}
Carlo Severini.
\newblock {\em Sulle successioni di funzioni ortogonali}.
\newblock Tip. Galatola, 1910.

\bibitem[SZ22]{SZ22_8}
Bin Shang and Chao Zhang.
\newblock H\"older regularity for mixed local and nonlocal {$p$}-{L}aplace
  parabolic equations.
\newblock {\em Discrete Contin. Dyn. Syst.}, 42(12):5817--5837, 2022.
\newblock \href {https://doi.org/10.3934/dcds.2022126}
  {\path{doi:10.3934/dcds.2022126}}.

\bibitem[SZ23]{SZ23_5}
Bin Shang and Chao Zhang.
\newblock Harnack inequality for mixed local and nonlocal parabolic
  {$p$}-{L}aplace equations.
\newblock {\em J. Geom. Anal.}, 33(4):Paper No. 124, 24, 2023.
\newblock \href {https://doi.org/10.1007/s12220-022-01173-x}
  {\path{doi:10.1007/s12220-022-01173-x}}.

\bibitem[Tru68]{Tru68_33}
Neil~S Trudinger.
\newblock Pointwise estimates and quasilinear parabolic equations.
\newblock {\em Communications on Pure and Applied Mathematics}, 21(3):205--226,
  1968.
\newblock URL: \url{https://doi.org/10.1002/cpa.3160210302}.


\end{thebibliography}
\end{document}